\documentclass[final
]{scrartcl}
\usepackage[utf8]{inputenc}
\usepackage{hyperref}
\usepackage[english]{babel}
\usepackage{enumerate}
\usepackage{aliascnt}
\usepackage{amssymb}
\usepackage{amsmath}
\usepackage{amsthm}
\usepackage{verbatim}
\usepackage{mathtools}
\usepackage[normalem]{ulem} 

\usepackage{tikz}
\usepackage{mathscinet} 

\usepackage[nottoc]{tocbibind} 
\usepackage[title]{appendix}

\usepackage[notcite,notref]{showkeys}
\usepackage{xcolor}
\numberwithin{equation}{section} 

\newcommand{\delete}[1]{}
\newcommand{\add}[1]{{#1}}

\usepackage{dsfont} 
\usepackage{hyperref} 
\usepackage[noabbrev]{cleveref}

\usepackage{autonum} 

\newcommand{\stkout}[1]{\ifmmode\text{\sout{\ensuremath{#1}}}\else\sout{#1}\fi}

\usepackage{tikz}

\allowdisplaybreaks[4] 

\newtheorem{prop}{Proposition}[section]

\newaliascnt{lem}{prop} 

\newtheorem{lem}[lem]{Lemma}

\aliascntresetthe{lem} 
\Crefname{lem}{Lemma}{Lemmas}

\newaliascnt{defi}{prop} 
 \newtheorem{defi}[defi]{Definition}

\aliascntresetthe{defi} 
\Crefname{defi}{Definition}{Definitions}

\newaliascnt{cor}{prop} 
 \newtheorem{cor}[cor]{Corollary}
 
\aliascntresetthe{cor}

\newaliascnt{remark}{prop} 
 \newtheorem{remark}[remark]{Remark}
 
\aliascntresetthe{remark} 

\newaliascnt{thm}{prop} 
 \newtheorem{thm}[thm]{Theorem}
 
\aliascntresetthe{thm} 

\newaliascnt{example}{prop} 
 
\aliascntresetthe{example}

\def\equationautorefname~#1\null{%
  (#1)\null
}

\DeclareMathOperator{\diam}{diam} 
\newcommand{\R}{\ensuremath{\mathbb{R}}}
\newcommand{\N}{\ensuremath{\mathbb{N}}}

\renewcommand{\S}{\ensuremath{\mathbb{S}}}

\newcommand*\diff{\mathop{}\!\mathrm{d}}

\newcommand{\defeq}{\vcentcolon=}
\newcommand{\eqdef}{=\vcentcolon}
\newcommand{\CalE}{\ensuremath{\mathcal{E}}}

\newcommand{\CalL}{\ensuremath{\mathcal{L}}}
\newcommand{\CalA}{\ensuremath{\mathcal{A}}}
\newcommand{\CalW}{\ensuremath{\mathcal{W}}}
\DeclareMathOperator{\grad}{grad}

\newcommand{\abs}[1]{\ensuremath{\lvert #1\rvert}}

\newcommand{\norm}[2]{\ensuremath{\Vert #1 \Vert_{#2}}}

\newcommand{\dtzero}{\left.\frac{\diff}{\diff t}\right\vert_{t=0}}

\DeclareMathOperator{\Id}{Id}

\DeclareMathOperator{\CalV}{\ensuremath{\mathcal{V}}}

\title{The volume-preserving Willmore flow}
\author{Fabian Rupp\thanks{Institute of Applied Analysis, Ulm University, Helmholtzstra\ss e 18, 89081 Ulm, Germany.} \thanks{Present address: Faculty of Mathematics, University of Vienna, Oskar-Morgenstern-Platz 1, 1090 Vienna, Austria. \texttt{fabian.rupp@univie.ac.at}.}}

\begin{document}
\date{\today}
\maketitle
\begin{abstract}
\noindent\textbf{Abstract:} We consider a closed surface in $\R^3$ evolving by the  volume{\-}preserving Willmore flow and prove a lower bound for the existence time of smooth solutions. For spherical initial surfaces with Willmore energy below $8\pi$ we show long time existence and convergence to a round sphere by performing a suitable blow-up and by proving a constrained {\L}ojasiewicz--Simon inequality. 
\end{abstract}

\bigskip
\noindent \textbf{Keywords:} Willmore flow, fixed volume, blow-up, {\L}ojasiewicz--Simon inequality, nonlocal geometric evolution equation.
 
 \noindent \textbf{MSC(2020)}: 53E40 (primary), 35B40, 35K41  (secondary).

\section{Introduction and main results}
For an immersion $f\colon\Sigma\to\R^3$ of a compact, connected {and oriented} surface $\Sigma$ without boundary, its \emph{Willmore energy} is defined by
\begin{align}\label{eq:defWillmore}
{{\CalW}}(f) \defeq \frac{1}{4} \int_{\Sigma} H^2\diff \mu.
\end{align}

Here $\mu=\mu_f$ denotes the area measure, induced by the pull-back of the Euclidean metric $g_f \defeq f^*\langle \cdot, \cdot\rangle$, and $H  = H_f\defeq \langle \vec{H}_f, \nu_f\rangle$ denotes the (scalar) mean curvature with respect to $\nu=\nu_f \colon\Sigma\to \S^2$, the unique unit normal along $f$ induced by the chosen orientation on $\Sigma$, {see \eqref{eq:def normal} below.} 
By the Gau\ss--Bonnet theorem, {the Willmore energy \eqref{eq:defWillmore}}  only differs by a topological constant from the squared $L^2$-norm of $A^0$, the trace-free part of the second fundamental form. {Indeed, we have}
\begin{align}\label{eq:WvstildeW}
\overline{\CalW}(f)\defeq  \int_{\Sigma}\abs{A^0}^2\diff \mu = 2\CalW(f)-4\pi\chi(\Sigma),
\end{align}
where $\chi(\Sigma)$ denotes the Euler characteristic. Note that both energies are not only \emph{geometric}, i.e.\ invariant under diffeomorphisms of $\Sigma$, but also \emph{conformally invariant}, i.e.\ invariant under \add{rigid motions and inversions provided the center of inversion does not lie on $f(\Sigma)$.} As already observed in \cite{Willmore65}, $\CalW(f)\geq 4\pi$ with equality only for round spheres. Therefore, $\CalW$ and hence also ${\overline{\CalW}}$ are a natural way to measure the total bending of an immersed surface with various applications {also beyond differential geometry, for instance in the  study of biological membranes \cite{Canham,Helfrich}, general relativity \cite{Hawking} and image restoration \cite{DroskeRumpf}.} 


The analysis of the \emph{Willmore flow}, i.e.\ the $L^2$-gradient flow associated to the energy $\overline{\CalW}$, started with the work of Kuwert and Schätzle.
In \cite{KSGF}, they proved a lifespan theorem under the assumption that the \emph{concentration of curvature} of the initial datum is controlled. In \cite{KSSI}, this was used to set up a blow-up procedure and to prove convergence to a round sphere if the energy is sufficiently small. Then in \cite{KSRemovability}, long-time existence and convergence was shown for the flow of spherical immersions with initial datum $f_0\colon\S^2\to\R^3$ satisfying ${\CalW}(f_0)\leq 8\pi$.
The threshold $8\pi$ already appears in the celebrated \emph{Li--Yau inequality for the Willmore energy} \cite[Theorem 6]{LiYau}, yielding that
\add{$f$ is an embedding if $\CalW(f)<8\pi$.}
Furthermore, this threshold is in fact sharp for the convergence result in \cite{KSRemovability}, see  \cite{MS2002} for numerical experiments and \cite{Blatt} for an analytic proof. \add{It remains an open problem to prove or disprove whether this singularity happens in finite time.}

Recently, similar convergence results have been established for the Willmore flow of \emph{tori of revolution} \cite{DMSS20} with the same energy threshold and \add{also for the Willmore flow of \emph{Hopf-tori} in the three-sphere $\S^3$ \cite{RubenConv}}. 

Moreover, various authors have extended the methods of Kuwert and Schätzle to related geometric evolution equations \add{also involving constraints}, including, \add{for instance}, the \emph{surface diffusion flow} \add{\cite{McCoyWheelerWilliams,MR2735555,MR2898774}}, \emph{Helfrich-type flows} \cite{McCoyWheeler,BlattHelfrich} \add{and other higher order flows \cite{MR3929519,MR3672995}}. In \cite{Jachan}, the \emph{area-preserving Willmore flow} was studied. 
{Related constrained evolution problems for the \emph{elastic energy} of curves have been considered in \cite{DKS}, \cite{DPL17} and \cite{RuppSpener}, for instance.}

In this article, we \delete{will consider}\add{introduce} a \emph{constrained gradient flow}, which evolves an initial {immersion} $f\colon\Sigma\to\R^3$ such that $\overline{\CalW}$ decreases as fast as possible, while $\CalV$, the \emph{signed volume} of $f(\Sigma)$, defined by
	\begin{align}\label{eq:defVol}
		\CalV(f)
		\defeq -\frac{1}{3} \int_{\Sigma}\langle f, \nu\rangle\diff\mu,
	\end{align}
is kept constant.\delete{, see \Cref{sec:volume} for a detailed discussion on this notion of volume.}
{The analogous problem for the \emph{mean curvature flow} was introduced by Huisken \cite{Huisken}.}
More explicitly, we say that a smooth family of {immersions} $f\colon [0,T)\times \Sigma \to\R^3$ is a \emph{volume-preserving Willmore flow}, if it satisfies the geometric evolution equation
\begin{align}\label{eq:VpWF}
\partial_t f &= \left(-\Delta H - | A^{0}|^2 H + \lambda\right)\nu,
\end{align}
where $\Delta=\Delta_f$ is the Laplace--Beltrami operator on $(\Sigma, \add{g_f})$ and the Lagrange multiplier $\lambda \defeq \lambda(t) \defeq \lambda(f_t)$ depends on the immersion $f_t\defeq f(t, \cdot)$ and is given by
\begin{align}\label{eq:deflambda}
\lambda \defeq \frac{\int_{\Sigma}  | A^{0}|^2H \diff \mu}{\CalA(f)},
\end{align}
where $\CalA(f)\defeq \int_{\Sigma} \diff \mu_f$ denotes the total area of $\Sigma$. In \Cref{subsec:geom evol}, we will prove that \eqref{eq:VpWF} actually decreases $\overline{\CalW}$, hence also $\CalW$ by \eqref{eq:WvstildeW}, while keeping $\CalV$ fixed. {Note that the energies defined in \eqref{eq:defWillmore} and \eqref{eq:WvstildeW} do not change, if we reverse the  orientation on $\Sigma$. While the normal $\nu$ and hence the volume \eqref{eq:defVol} change sign, the flow equation \eqref{eq:VpWF} with $\lambda$ as in \eqref{eq:deflambda} is also invariant under reversing the orientation.}

\add{The \emph{stationary solutions} to the volume-preserving Willmore flow \eqref{eq:VpWF} are characterized as solutions of the PDE
	\begin{align}\label{eq:stationary}
		\Delta H + \abs{A^0}^2H = \lambda\quad \text{for some }\lambda\in\R.
	\end{align}
	Formally, \eqref{eq:stationary} is the Euler--Lagrange equation of critical points of $\CalW$ (and hence of $\overline{\CalW}$) subject to a volume constraint, so we refer to solutions of \eqref{eq:stationary} as \emph{(volume)-constrained Willmore immersions,} see \Cref{lem:first Vari V and W} and \Cref{lem:constr Willmore} below. By changing from $\lambda$ to $-\lambda$, \eqref{eq:stationary} is preserved under reversing the orientation on $\Sigma$. Note that while the energies $\CalW$ and $\CalV$ may not be well-defined, \eqref{eq:stationary} makes sense even if $\Sigma$ is not compact, and we still term noncompact solutions of \eqref{eq:stationary} constrained Willmore immersions.}

\add{Our first main contribution extends the energy concentration-based lower bound on the lifespan for the Willmore flow \cite{KSGF} to the volume-preserving Willmore flow.}
\delete{In addition to the mathematical challenges of the Willmore flow, the \emph{nonlocal} nature of the Lagrange multiplier in \eqref{eq:deflambda} is a potential source of difficulties, making it unclear if controlling the concentration of curvature suffices to control the entire evolution. Nevertheless, we {are able to} extend the results on the Willmore flow in \cite{KSGF,KSSI,KSRemovability,CFS09} to the constrained flow \eqref{eq:VpWF}. The following result gives a lower bound on the lifespan.
}
\begin{thm}\label{thm:Lifespan}
	There exists an absolute constant $\bar{\varepsilon}>0$ such that if $f_0\colon \Sigma\to \R^3$ is an {immersion} with ${\CalW}(f_0)\leq K$ and $\rho>0$ is chosen such that
	\begin{align}
		\int_{B_\rho(x)} \abs{A_0}^2\diff \mu_0 \leq\varepsilon<\bar{\varepsilon} \quad\text{for all }x\in \R^3,
	\end{align}
	then the maximal existence time $T$ of the volume-preserving Willmore flow with initial data $f_0$ satisfies
	\begin{align}
		T> \hat{c}\rho^4,
	\end{align}
	for some \add{$\hat{c}=\hat{c}(K, \chi(\Sigma))>0$} and furthermore for all $0\leq t\leq\hat{c}\rho^4$ it holds
	\begin{align}\label{eq:LifeSpanAEstimate}
		\int_{B_\rho(x)} \abs{A}^2\diff \mu \leq \hat{c}^{-1}\varepsilon\quad \text{for all } x\in \R^3.
	\end{align}
\end{thm}
\delete{Here we identify the measure $\mu=\mu_t$ on $\Sigma$ with its \emph{pushforward} $(f_t)_{\ast}\mu$ on $\R^3$. Moreover, we suppress the dependence on the time variable and write
	\begin{align}\label{eq:mu_t measure R^3}
		\int_{B_\rho(x)} \psi\diff \mu \defeq  \int_{f_t^{-1}(B_\rho(x))}\psi(t,y)\diff \mu_t(y),
	\end{align}
	for $0\leq t<T, x\in \R^3$, $\rho>0$ and any $\psi\colon [0,T)\times\Sigma\to \R$ such that the latter integral exists.}
\add{Following the notation of \cite{KSGF}, the integrals above have to be understood over the preimages under $f_0$ and $f_t$, respectively.} 

\add{The proof of \Cref{thm:Lifespan} follows the concentration-compactness strategy developed by Kuwert and Sch\"atzle for the Willmore flow in \cite{KSGF}. Since this method is relying on smallness of the curvature in small balls, it is \emph{intrinsically local,} making the \emph{nonlocal} nature of the Lagrange multiplier a major difficulty. To compensate this, the $L^{4/3}$-norm of $\lambda$ naturally appears in these estimates, a scale-invariant quantity (see \Cref{rem:ParabolicScaling}) which we can control under certain assumptions, see \Cref{sec:Int est Lambda L43} below. In particular, we do not assume a-priori $L^{\infty}$-bounds on $\lambda$ as in \cite{McCoyWheelerWilliams,MR2735555}. However, to show this correct integrability, we have to allow the constant $\hat{c}$ to depend on an upper bound for the initial energy, as well as on the topology of $\Sigma$, in contrast to \cite[Theorem 1.2]{KSGF}.}
\delete{Note that in contrast to the classical lifespan result on the Willmore flow, cf. \cite[Theorem 1.2]{KSGF}, the constant $\hat{c}$ depends on an upper bound for the initial energy, as well as on the topology of $\Sigma$.}

Our \add{second main contribution} \delete{main result} extends the  convergence result of Kuwert and Schätzle \cite[Theorem 5.2]{KSRemovability} in the following
 
	\begin{thm}\label{thm:conv main}
	 	\add{Let $f_0 \colon\S^2\to\R^3$ be a smooth immersion such that ${\CalW}(f_0)\leq 8\pi$ and $\CalV(f_0)\neq 0$.} Then the volume-preserving Willmore flow with initial datum $f_0$ exists for all times and converges smoothly, after reparametrization, to a round sphere with radius $R = \sqrt[3]{\frac{3 \abs{\CalV(f_0)}}{4\pi}}$ as $t\to\infty$.
	\end{thm}
	\add{As in \cite{KSSI,KSRemovability}, the strategy to prove \Cref{thm:conv main} is a blow-up construction based on the lifespan bound in \Cref{thm:Lifespan}. The blow-up limit is a constrained Willmore immersion and in general noncompact. However, apart from a small energy regime \cite{MWHelfrich}, for $\lambda\neq 0$ there is no classification of solutions to \eqref{eq:stationary}. Nonetheless, under an $L^2$-integrability assumption on the Lagrange multiplier, we are able to conclude that the blow-up is a Willmore immersion, i.e.\ $\lambda=0$ in \eqref{eq:stationary}. In the energy regime of \Cref{thm:conv main}, this integrability can be deduced from a reverse isoperimetric inequality \cite{Schygulla,Blatt}.
		 Together with the removability result \cite{KSRemovability} and the classification of Willmore spheres \cite{Bryant1984}, we then conclude that the blow-up limit is compact. A result of independent interest is that the volume-constrained Willmore functional satisfies an appropriate constrained version of the {\L}ojasiewicz--Simon gradient inequality in the sense of \cite{Rupp}. Finally, this inequality yields a stability result in the spirit of \cite{CFS09} from which we conclude global existence and convergence of the flow if a blow-up is compact.
	
	Note that in view of \cite[Theorem 5.2]{KSRemovability}, we believe that the reparametrization in \Cref{thm:conv main} is not necessary, but a common consequence when relying on the {\L}ojasiewicz--Simon gradient inequality, cf.\ \cite[Lemma 4.1]{CFS09}.}
\delete{As we shall see later, the reason for restricting to the case of spherical surfaces in \Cref{thm:conv main} is the classification of \emph{Willmore spheres} by Bryant \cite{Bryant1984}.}

	\delete{
	 We remark that \emph{stationary solutions} to the volume-preserving Willmore flow are the \emph{(volume-)constrained Willmore immersions}, i.e.\ critical points of $\CalW$ (and hence of $\overline{\CalW}$) with respect to a volume constraint. They are characterized by the equation
	 \begin{align}\label{eq:stationary}
	 	\Delta H + \abs{A^0}^2H = \lambda\quad \text{for some }\lambda\in\R.
	 \end{align}
	 By changing from $\lambda$ to $-\lambda$, \eqref{eq:stationary} is preserved under reversing the orientation on $\Sigma$. Note that while the energies $\CalW$ and $\CalV$ may not be well-defined, \eqref{eq:stationary} still makes sense if $\Sigma$ is not compact. 
}

This article is structured as follows. First,  we recall some definitions and compute the evolution of relevant quantities in \Cref{sec:prelim}. \add{\Cref{sec:loc int est} is devoted to proving a key ingredient of the paper: localized integral estimates in the spirit of \cite{KSGF}, which now require an $L^{4/3}$ in time integrability of the Lagrange multiplier.} Combined with the careful \add{a-priori} estimates of $\lambda$ which we establish in \Cref{sec:Int est Lambda}, they are then used to prove the lifespan bound, {\Cref{thm:Lifespan}}, in \Cref{sec:lifespan}. 
\add{In \Cref{sec:BlowUp} we construct a blow-up limit and study its properties. We then \add{deduce a convergence result for compact blow-ups in the spirit of \cite{CFS09}} by \add{proving} a \emph{constrained {\L}ojasiewicz--Simon gradient inequality} in \Cref{sec:blow up not compact}, before proving \Cref{thm:conv main} in \Cref{sec:conv}.
For the sake of readability, some details and well-known arguments have been moved to the appendix and may be skipped by the eager or experienced reader.}

\section{Preliminaries}\label{sec:prelim}
In this section, we will review the geometric and analytic background and prove some first properties of the flow \eqref{eq:VpWF}.
In the following, $\Sigma$ will always denote a compact and connected oriented surface  without boundary. {Note that in contrast to \cite{KSGF,KSSI}, we work exclusively in codimension one, which simplifies the relevant geometric objects.}

\subsection{Immersed and embedded surfaces in \texorpdfstring{$\R^3$}{\R^3}}\label{subsec:geometry}

An immersion $f\colon\Sigma\to\R^3$  induces the pullback metric $g=f^{\ast}\langle\cdot, \cdot\rangle$ on $\Sigma$, which in local coordinates is given by
\begin{align}
g_{ij}\defeq \langle\partial_i f, \partial_j f\rangle,
\end{align} 
where $\langle \cdot, \cdot\rangle$ denotes the Euclidean metric. {The chosen orientation on $\Sigma$ determines a unit normal field $\nu\colon\Sigma\to\S^2$ along $f$, which in local coordinates in the orientation  is given by
\begin{align}\label{eq:def normal}
	\nu = \frac{\partial_1 f\times \partial_2 f}{\abs{\partial_1 f\times \partial_2 f}}.
\end{align} 
We will always work with this unit normal vector field.
}
The second fundamental form of $f$ is then given by projecting the second derivatives of $f$ in normal direction, i.e.\ in local coordinates we define $ A_{ij} \defeq \langle\partial_i\partial_j f, \nu\rangle$. The mean curvature and the tracefree part of the second fundamental form are
\begin{align}
H \defeq g^{ij}A_{ij} \text{ and }  A^{0}_{ij} \defeq  A_{ij} - \frac{1}{2} {H}g_{ij},
\end{align}
where $g^{ij}\defeq \left(g_{ij}\right)^{-1}$. {Note that $g$, $A$ and $A^0$ are scalar valued $2\choose 0$-tensors.}   Moreover, for any vector field $X$ along $f$, we have the tangential and normal projections
\begin{align}
	P^{\top} X \defeq P^{\top_f}X &\defeq  g^{ij}\langle X, \partial_i f\rangle \partial_j f \\
	P^{\perp} X\defeq P^{\perp_f}X &\defeq X-P^{\top}X.
\end{align}
The Levi-Civita connection $\nabla=\nabla_f$ induced by $g$ extends uniquely to a connection on tensors, which we also denote by $\nabla$. For \add{an orthonormal} basis $\{e_1, e_2\}$ of the tangent space, the Codazzi--Mainardi equations then yield
\begin{align}
\nabla_i H &= (\nabla_j A)(e_i, e_j) = 2(\nabla_j A^{0})(e_i, e_j) \label{eq:nabla H A A^0}.
\end{align}
The \emph{Laplace--Beltrami operator} on $(\Sigma, \add{g})$ is given by
\begin{align}
	\Delta_g \xi = g^{ij} \nabla_i \nabla_j \xi, \quad \text{ for }\xi \in C^{\infty}(\Sigma).
\end{align}

For a $2\choose 0$- tensor $T_{ij}$, its tensor norm is \add{$\abs{T}^2 \defeq g^{ij}g^{k\ell} T_{ik}T_{j\ell}$} and hence we get
\begin{align}\label{eq:AA0H}
	\abs{A}^2 = \abs{A^0}^2+\frac{1}{2}H^2.
\end{align}
Consequently, using \eqref{eq:WvstildeW}, we find
\begin{align}\label{eq:A^2GaussBonnet}
	\int_{\Sigma} \abs{A}^2\diff \mu  = \overline{\CalW}(f)+2\CalW(f)=4{\CalW}(f)-4\pi\chi(\Sigma).
\end{align}

\subsection{The PDE perspective}\label{subsec:PDE}
	Note that in the general context of \emph{constrained gradient flows} on Hilbert spaces, cf. \cite[Section 5]{Rupp}, the flow in a Hilbert space $H$ associated to the energy $\CalE =\overline{\CalW}$ with constraint $\mathcal{G}=-\CalV\equiv constant$ is formally given by 
	\begin{align}
		\left\lbrace\begin{array}{ll}
			\partial_t f &= -\nabla_H\overline{\CalW}(f) - \lambda(f)\nabla_H\CalV(f), \quad t>0 \\
			f(0)&=f_0,
		\end{array}\right.
	\end{align}
	where the Lagrange multiplier is defined by the formula
		\begin{align}
			\lambda(f) =- \frac{\langle \nabla\overline{\CalW}(f), \nabla\CalV(f)\rangle_H}{\norm{\nabla\CalV(f)}{H}^2}.
		\end{align}
	If we choose $H\defeq L^2(\diff \mu_f)$, by the explicit form of the $L^2$-gradients (see \Cref{lem:first Vari V and W} below), the divergence theorem and the fact that $\partial \Sigma=\emptyset$, this definition coincides with the flow in \eqref{eq:VpWF} with Lagrange multiplier $\lambda$ as in \eqref{eq:deflambda}. In particular, $\lambda$ does not contain any derivatives of the curvature and is therefore of lower order compared to the leading term $-\Delta H$ in \eqref{eq:VpWF}. {This will significantly simply the analysis of $\lambda$ later on.}
	
	Despite that, the flow equation \eqref{eq:VpWF} is still a \emph{quasilinear, degenerate parabolic} PDE of 4\textsuperscript{th} order which is \emph{nonlocal} due to the Lagrange multiplier. Hence, even short time existence and uniqueness is not \add{immediate}. However, as $\lambda$ is of lower order, for smooth initial data one can show the following local well-posedness \add{result by using an appropriate fixed-point argument in parabolic H\"older spaces, see for instance \cite[Section 7]{MR2164924} and \cite[Section 3.1]{KSLectureNotes}}.
	\begin{prop}\label{prop:STE}
		Let $f_0\colon\Sigma\to\R^3$ be a smooth immersion. Then there exist $T\in (0,\infty]$ and a unique, nonextendable smooth solution $f\colon[0,T)\times \Sigma\to\R^3$ of the volume-preserving Willmore flow with initial datum $f(0)=f_0$.
	\end{prop}
	\delete{
	\begin{proof}
		This can be obtained by mimicking the proof of the short time existence for the Willmore flow in \cite[Section 3.1]{KSLectureNotes}. The ansatz is to write the solution as a normal graph over the initial datum, which reduces the problem to a strictly parabolic equation, and then use an appropriate fixed-point argument in parabolic H\"older spaces, see also the proof of \Cref{lem:LojaAsymStabil} below.
	\end{proof}
	Note that if only Sobolev regularity of the initial datum is given, the involved contraction estimates can become significantly more complicated, see \cite{RuppSpener} for a related flow of curves.}
 
	An important property of the volume-preserving Willmore flow is the following \emph{parabolic scaling}, which directly follows from the scaling behavior of the geometric quantities.
	\add{
	\begin{remark}\label{rem:ParabolicScaling}
		If $f\colon [0,T)\times \Sigma\to\R^3$ is a volume-preserving Willmore flow, then for any $\rho>0$ the family of immersions $\tilde{f}(t,p) \defeq \rho^{-1} f(\rho^4t, p)$
	is also a volume-preserving Willmore flow on $[0, \tilde{T})\times \Sigma$ with $\tilde{T}=\rho^{-4}T$. Moreover, {by a direct computation we have}
	\begin{align}
		\int_0^T \abs{\lambda(t)}^{\frac{4}{3}}\diff t = \int_0^{\tilde{T}} \abs{\tilde{\lambda}(t)}^{\frac{4}{3}} \diff t, \label{eq:scale lambda^4/3}\quad 
		\int_0^T \abs{\lambda(t)}^2\CalA(f_t)\diff t = \int_0^{\tilde{T}} \abs{\tilde{\lambda}(t)}^2 \CalA(\tilde{f}_t)\diff t.\label{eq:scale lambda^2}
	\end{align}
	{Note that the power $p=\frac{4}{3}$ is the only exponent for which the $L^p$-norm of $\lambda$ behaves correctly with respect to the rescaling above and will naturally show up in \Cref{sec:loc int est} below. 
	In \Cref{sec:Int est Lambda}, we  show how to control both of these integrals.}
	\end{remark}}

	\delete{Instead of studying the $L^2$-gradient flow of the energy $\overline{\CalW}$, one could directly work with the (volume constrained) gradient flow of the Willmore energy \eqref{eq:defWillmore}. By \eqref{eq:WvstildeW}, this would yield the evolution equation
	\begin{align}\label{eq:VpWF2}
		\partial_t f = \frac{1}{2}\left(-\Delta H -\abs{A^0}^2H + \lambda\right),
	\end{align}
	with $\lambda$ as in \eqref{eq:deflambda}, see also \Cref{lem:first Vari V and W} below.
	Therefore, the velocities of the flows \eqref{eq:VpWF} and \eqref{eq:VpWF2} only differ by a factor. As a consequence $f\colon [0,T)\times \Sigma\to\R^3$ is a solution to \eqref{eq:VpWF} if and only if $h\colon [0,2T)\times \Sigma\to\R^3$, $h(t,p)\defeq f(\frac{t}{2},p)$ solves \eqref{eq:VpWF2}, so both flows only differ by a temporal rescaling. Following Kuwert and Schätzle, we work with \eqref{eq:VpWF} since its analytic form is more convenient. On the other hand, for geometric purposes, traditionally the Willmore energy \eqref{eq:defWillmore} is more common, which is why we state all energy bounds in terms of $\CalW$ and the Euler characteristic.}

\subsection{Evolution of the geometric quantities}\label{subsec:geom evol}
In this section, we \add{recall} the variation of the relevant geometric quantities and the (localized) evolution of the energy. 
\add{The proofs are standard and can be found in \cite{KSGF,KSSI}, for instance, or follow from direct computations.}
\begin{lem}\label{lem:geometricEvolutions}
	Let $f\colon [0,T)\times \Sigma \to \R^3$ be a smooth family of {immersions} with normal velocity $\partial_t f =\xi\nu$. For an  orthonormal basis $\{e_1, e_2\}$ of the tangent space, the geometric quantities induced by $f$ satisfy
	\begin{align}
		(\partial_tg)(e_i, e_j) &= -2 A_{ij}\xi,\label{eq:dtg}\\
		\partial_t (\diff \mu) &= -H\xi \diff \mu,\label{eq:dtdmu}\\
		\partial_t  H &= \Delta \xi+ |A^{0}|^2\xi + \frac{1}{2} H^2\xi,\label{eq:dtH}\\
		(\partial_t  A)(e_i,e_j) &= \nabla^2_{ij} \xi - A_{ik}A_{kj}\xi,\label{eq:dtA}\\
		(\partial_t  A^{0})(e_i,e_j) &= \left(\nabla^2_{ij}\xi\right)^0 - g_{ij}|A^{0}|^2 \xi,\label{eq:dtA^0} \\
		\partial_t \nu &= - \grad_{g} \xi \eqdef g^{ij}\partial_i \xi \partial_j f.\label{eq:dtnu}
	\end{align}
\end{lem}

\delete{
\begin{proof}
	See \cite[Lemma 2.2]{KSGF} for \eqref{eq:dtg}-\eqref{eq:dtA} and \cite[(27)-(30)]{KSSI} for \eqref{eq:dtA^0}. For \eqref{eq:dtnu}, we observe that  $\nu$ is smooth in time {by \eqref{eq:def normal}} with $\langle\partial_t \nu, \nu\rangle=0$ and hence
	\begin{align}
		\partial_t \nu &= g^{ij} \langle\partial_t \nu, \partial_j f\rangle\partial_i f = -g^{ij} \langle\nu, \partial_j (\xi\nu)\rangle\partial_i f = -g^{ij}\partial_j \xi \partial_i f = -\grad_{g} \xi.\qedhere
	\end{align}
\end{proof}
}

As a consequence, we find the first variation of the energy and the volume.
\begin{lem}\label{lem:first Vari V and W}
	Let $f\colon\Sigma\to\R^3$ be an immersion. Then, the first variations of $\overline{\CalW}$ and $\CalV$ are given by
	\begin{align}\label{eq:1 Vari Vol}
	\CalV^{\prime}(f)\varphi &= -\int\langle\nu, \varphi\rangle\diff \mu,\\
	 \label{eq:1 Vari Willmore}
	 \overline{\CalW}^{\prime}(f)\varphi &= \int \langle (\Delta H + \abs{A^0}^2H)\nu, \varphi\rangle \diff \mu,
	\end{align}
{for all $\varphi\in C^{\infty}(\Sigma;\R^3)$ normal along $f$. Here and in the following, we always integrate over the whole surface $\Sigma$ if the domain of integration is not specified.}
\end{lem} 

\delete{
\begin{proof}
	To prove \eqref{eq:1 Vari Vol}, we use \eqref{eq:defVol}, \eqref{eq:dtnu} and \eqref{eq:dtdmu} to compute
	\begin{align}\label{eq:1 Vari Vol 1}
		\dtzero \CalV(f+t\xi\nu) &= -\frac{1}{3} \left( \int \langle\xi\nu, \nu\rangle\diff \mu - \int \langle f, \grad_g \xi\rangle\diff \mu - \int \langle f, \nu\rangle H\xi\diff \mu \right).
	\end{align} 
	A short computation yields $\Delta f = H\nu$, where the Laplace--Beltrami operator is applied to each component of $f$. Hence, for the last term we integrate by parts to find
	\begin{align}
		-\int \langle f, \nu\rangle H\xi\diff \mu
		= \int \langle \nabla(f\xi), \nabla f\rangle \diff \mu = \int \abs{\nabla f}^2\xi \diff \mu + \int \langle f,\grad_g \xi\rangle\diff \mu.
	\end{align}
	Inserting this into \eqref{eq:1 Vari Vol 1} and using $\abs{\nabla f}^2 = g^{ij}\langle \partial_i f, \partial_j f\rangle =2$, the claim follows.
	Formula \eqref{eq:1 Vari Willmore} can be obtained similarly using \eqref{eq:WvstildeW}, \eqref{eq:dtH} and \eqref{eq:dtdmu} to obtain
	\begin{align}
		\overline{\CalW}^{\prime}(f)\varphi = 2\CalW^{\prime}(f)\varphi = \int H\partial_t H + \frac{1}{2}\int H^2\partial_t(\diff \mu) = \int (H\Delta\xi + \abs{A^0}^2H\xi)\diff \mu. & \qedhere
	\end{align}
\end{proof}}

This means that the \emph{$L^2(\diff \mu_f)$-gradients} of $\overline{\CalW}$ and $\CalV$ are given by the identities
\begin{align}
	\nabla \overline{\CalW}(f) &= (\Delta H +\abs{A^0}^2H)\nu, \\
	\nabla \CalV(f) &= -\nu.
\end{align}
These gradients are purely normal, so we will often work with the \emph{scalar $L^2(\diff\mu_f)$-gradient} 
\begin{align}\label{eq:nabla W sc}
	\nabla_{sc}\overline{\CalW}(f) \defeq \Delta H +\abs{A^0}^2H.
\end{align}
\add{By direct computation, along a solution} of \eqref{eq:VpWF} the volume is indeed preserved since
\begin{align}\label{eq:dtV=0}
\frac{\diff}{\diff t} \CalV(f) &= 0,
\end{align}
\delete{by \eqref{eq:1 Vari Vol}, \eqref{eq:deflambda} and the divergence theorem,} whereas by \eqref{eq:1 Vari Willmore} and \eqref{eq:dtV=0} the energy decreases by
\begin{align}
\frac{\diff}{\diff t}\overline{\CalW}(f) 
&= 
 -\int \abs{\partial_t f}^2\diff \mu \leq 0.\label{eq:dtW<=0}
\end{align}
\begin{remark}\label{rem:W strict Lyapunov}
	The computation in \eqref{eq:dtW<=0} implies that $\overline{\CalW}$ is a \emph{strict Lyapunov function}, i.e.\ $\overline{\CalW}$ is strictly decreasing unless $f$ is constant. {By \eqref{eq:WvstildeW} this also holds for $\CalW$.}
\end{remark}

{
	It is now easy to prove the following rigidity result for constrained Willmore immersions.
	\begin{lem}\label{lem:constr Willmore}
		Let $\Sigma$ be a compact, oriented surface without boundary and let $f\colon\Sigma\to\R^3$ be a solution to \eqref{eq:stationary} with $\CalV(f)\neq 0$. Then $f$ is a Willmore immersion, i.e.\ a solution to \eqref{eq:stationary} with $\lambda=0$.
	\end{lem}
	We note that the assumptions in \Cref{lem:constr Willmore} are automatically satisfied if $f$ is an embedding of a compact surface, since then $\CalV$ is exactly the volume of the domain enclosed by $f(\Sigma)$ \add{by the divergence theorem}\delete{, cf. \Cref{sec:volume}}. {On the other hand, the example of an infinitely long cylinder shows that the statement of the lemma is no longer true without the compactness assumption.}
	\begin{proof}[{Proof of \Cref{lem:constr Willmore}}]
		We observe that by the scaling invariance of the Willmore energy, we have $\overline{\CalW}(f+tf)= \overline{\CalW}(f)$ for all \add{$\abs{t}<1$}. Hence, by \eqref{eq:1 Vari Willmore} (which also holds for variations which are not necessarily normal, see for instance \cite[p. 11]{KSLectureNotes}), we find
		\begin{align}
			0=\dtzero \overline{\CalW}(f+tf) &= \int\langle \nabla\overline{\CalW}(f), f\rangle \diff \mu = \int \langle (\Delta H + \abs{A^0}^2H)\langle \nu, f\rangle \diff \mu = 3\lambda \CalV(f),
		\end{align}
	using \eqref{eq:stationary} and then \eqref{eq:defVol} in the last step. As $\CalV(f)\neq 0$, this yields the claim.
	\end{proof}
}

%

\delete{
Following \cite{KSGF,KSSI}, for tensors $\phi, \psi$ on $\Sigma$, we denote by $\phi*\psi$ any multilinear form, depending on $\phi$ and $\psi$ in a universal bilinear way. In particular,
we have $\abs{\phi*\psi}\leq c\abs{\phi}\abs{\psi}$ and $\nabla(\phi*\psi)=\nabla\phi*\psi+\phi*\nabla\psi$. Note that since we are in codimension one, we can work with tensors with scalar {values} and not with normal values.

Moreover, for $m\in \N_0$ and $r\in \N, r\geq 2$ we denote by $P^m_r(A)$ any term of the type
\begin{align}
P^m_r(A) = \sum_{i_1+\dots+i_r=m} \nabla^{i_1}A*\dots*\nabla^{i_r}A.
\end{align}
{In addition, for $r=1$ we extend this definition by denoting by $P^m_1(A)$ any contraction of $\nabla^mA$ with respect to the metric $g$.}  
Given an $\ell\choose 0$-tensor $\psi$ we denote by $\nabla^{\ast}\psi$ the formal adjoint of $\nabla$ given by $\nabla^{\ast} \psi = -(\nabla_{e_i} \psi)(e_i, \dots)$. In this notation the Laplacian on tensors is given by $\Delta = -\nabla^{\ast}\nabla$. An important commutator relation is the identity
\begin{align}\label{eq:Nabla*NablaCommutator}
(\nabla \nabla^{\ast} - \nabla^{\ast}\nabla)\nabla \psi= A*A*\nabla\psi + A*\nabla A*\psi
\end{align}
for any $\ell\choose 0$-tensor $\psi$, cf. \cite[(2.11)]{KSGF}. 
Moreover, we recall Simons' identity (cf. \cite{Simons})
\begin{align}\label{eq:Simons'Identity}
\Delta A = \nabla^2 H + 2 K A^0 = \nabla^2 H + A*A*A.
\end{align}
We can now compute the evolution of higher order derivatives of the second fundamental form.
\begin{lem}\label{lem:localEvolNabla^mA}
	Let $f\colon [0,T)\times\Sigma\to\R^3$ be a volume-preserving Willmore flow. Then for all $m\in \N_0$ we have
	\begin{align}
		\partial_t (\nabla^m A)+ \Delta^2(\nabla^m A) = P^{m+2}_3(A) + P^m_5(A) + \lambda P^{m}_2(A).
	\end{align}
\end{lem}

\begin{proof}
	{First, we note that $H$ is a contraction of $A$ and hence $H=P^{0}_1(A)$, and consequently also $A^0=P^0_1(A)$. Thus, by \eqref{eq:VpWF}, we have 
	\begin{align}\label{eq:dt f in P notation}
		\xi =-\Delta H + P_3^0(A)+\lambda.
	\end{align}
	}
	For $m=0$ we insert this into \eqref{eq:dtA} to obtain
	\begin{align}
		\partial_t A &= \nabla^2 \xi + A*A*\xi = -\nabla^2(\Delta H) + P^2_3(A) + P^0_5(A) + \lambda P_2^0(A),
	\end{align}
	Using \eqref{eq:Nabla*NablaCommutator} twice, we find $	\nabla^2\Delta H = \Delta\nabla^2 H +P^2_3(A)$, hence by \eqref{eq:Simons'Identity} we have
	\begin{align}
		\partial_t A = - \Delta^2 A + P^2_3(A)+P^0_5(A)+\lambda P^0_2(A).
	\end{align}
	Assume the statement is true for $m\geq 1$. Using \cite[Lemma 2.3]{KSGF} with $\phi = \nabla^m A$ and the fact that we are in codimension one yields
	\begin{align}
		\partial_t \nabla^{m+1} A + \Delta^2 \nabla^{m+1}A &= \nabla\left(P^{m+2}_3(A)+P^{m}_5(A)+\lambda P^{m}_2(A)\right)\\
		&\quad + \sum_{i+j+k=3} \nabla^{i}A*\nabla^j A*\nabla^{k+m}A \\
		&\quad + A*\nabla\xi*\nabla^m A + \nabla A*\xi*\nabla^m A \\
		& = P^{m+3}_3(A) + P^{m+1}_5(A)+\lambda P^{m+1}_2(A),
	\end{align}
where we used \eqref{eq:dt f in P notation} in the last step.
\end{proof}
}
\delete{
As in \cite[Section 3]{KSSI}, we compute a localized version of \eqref{eq:dtW<=0}. {We give the full details here to highlight how the dependence on $\lambda$ comes into play.}

\begin{lem}\label{lem:EvolutionOfCurvatureIntegrals}
	Let $f\colon[0,T)\times \Sigma\to\R^3$ be a smooth volume-preserving Willmore flow, $\tilde{\eta} \in C_c^{\infty}(\R^3)$ and $\eta\defeq  \tilde{\eta}\circ f$. Then, we have
	\begin{align}
		\partial_t \int\frac{1}{2}H^2\eta\diff \mu + \int \abs{\nabla\overline{\CalW}(f)}^2\eta\diff\mu &= \lambda \int |A^0|^2H\eta  \diff \mu -2 \int \nabla_{sc}\overline{\CalW}(f) \langle\nabla H, \nabla \eta\rangle \diff \mu  \\
		&\quad - \int \nabla_{sc}\overline{\CalW}(f) H \Delta\eta \diff \mu + \int\frac{1}{2}H^2\partial_t \eta \diff\mu \label{eq:dtIntH^2Equation}
	\end{align}
	and
	\begin{align}
		\partial_t \int |A^{0}|^2\eta\diff \mu + \int \abs{\nabla\overline{\CalW}(f)}^2\eta \diff \mu &= \lambda\int |A^0|^2 H \eta\diff \mu -2\int  \nabla_{sc}\overline{\CalW}(f) \langle \nabla H,\nabla \eta\rangle \diff \mu \\
		&\quad - 2 \int  \nabla_{sc}\overline{\CalW}(f) \langle A^{0},\nabla^2\eta\rangle \diff \mu  +  \int|A^{0}|^2\partial_t \eta\diff\mu.\label{eq:dtIntA^0^2Eq}
	\end{align}
\end{lem}
\begin{proof}
	We use a (local) orthonormal basis $\{e_i\}_{i=1,2}$. By \eqref{eq:dtdmu} and \eqref{eq:dtH} we find
	\begin{align}
	 	\partial_t \left(\frac{1}{2}H^2\diff  \mu\right) &= H\partial_t H \diff \mu - \frac{1}{2}H^3\xi\diff \mu = \left(H\Delta \xi+ H|A^{0}|^2\xi\right)\diff \mu \\
	 	&= (\Delta H +|A^{0}|^2H) \xi \diff \mu + (H\Delta\xi - \xi \Delta H)\diff \mu \\
	 	&= - \abs{\nabla\overline{\CalW}(f)}^2\diff \mu + \lambda \Delta H \diff \mu + \lambda |A^{0}|^2 H\diff \mu + \nabla_i\left(H\nabla_i \xi - \xi\nabla_i H\right)\diff \mu
	\end{align}
	Consequently, we compute using integration by parts
	\begin{align}
		&\partial_t \int\frac{1}{2}H^2\eta\diff \mu + \int \abs{\nabla\overline{\CalW}(f)}^2\eta\diff\mu \\
		&\qquad= \lambda \int (\Delta H + |A^0|^2H)\eta  \diff \mu +\int \left(2\xi\nabla_i H\nabla_i \eta + H\xi\Delta\eta\right)\diff \mu + \int\frac{1}{2}H^2\partial_t \eta \diff\mu.
	\end{align}
	Now, using \eqref{eq:nabla W sc} we observe that 
	\begin{align}
		\int \left(2\xi\nabla_i H\nabla_i \eta + H\xi\Delta\eta\right)\diff \mu &= -2 \int  \nabla_{sc}\overline{\CalW}(f) \nabla_i H \nabla_i \eta \diff \mu + 2\lambda \int \nabla_i H\nabla_i \eta\diff \mu \\
		&\quad - \int \nabla_{sc}\overline{\CalW}(f) H \Delta\eta \diff \mu + \lambda \int H\Delta\eta\diff \mu.
	\end{align}
	{Recalling that $\Delta(H\eta) = \Delta H \eta + 2 \nabla_i H \nabla_i \eta + H\Delta\eta$, the identity \eqref{eq:dtIntH^2Equation} follows.}

	For the second identity, we first observe that by \eqref{eq:dtg} we have 
	\begin{align}
		g(\partial_t e_i, e_j)+g(e_i, \partial_t e_j) = -(\partial_t g)(e_i, e_j) = 2 A_{ij}^0\xi + \delta_{ij} H\xi.
	\end{align}
	Moreover, writing $\partial_t e_i = g(\partial_t e_i, e_k) e_k$ and using the symmetry of $A^0$ we find
	\begin{align}
		A^0(\partial_t e_i, e_j) A^0(e_i, e_j) &= g(\partial_t e_i, e_k) A^0(e_k, e_j)A^0(e_i, e_j)\\
		&= \frac{1}{2}\big(g(\partial_t e_i, e_k) + g(e_i, \partial_t e_k)\big) A^{0}(e_k, e_j)A^0(e_i, e_j) \\
		&= \big(A^{0}(e_i, e_k) \xi + \frac{1}{2}\delta_{ik}H\xi\big)A^{0}(e_k,e_j)A^{0}(e_i,e_j) \\
		&= \frac{1}{2} |A^0|^2 H\xi, \label{eq:A^0dtei}
	\end{align}
	since $A^{0}_{ik}A^{0}_{kj}A^{0}_{ij}=0$ (using \cite[(2.5)]{KSGF}). Now, we apply \eqref{eq:dtdmu}, \eqref{eq:dtA^0} and \eqref{eq:A^0dtei} to find
	\begin{align}
		&\partial_t\left(|A^{0}|^2\diff \mu\right) =2 A^{0}(e_i, e_j)^{0} \partial_t\left(A^{0}(e_i, e_j)\right) \diff \mu - |A^{0}|^2 H\xi\diff \mu \\
		&\qquad = 2 A^{0}(e_i, e_j) \left[(\partial_t A^{0})(e_i, e_j) +A^{0}(\partial_t e_i, e_j) +  A^{0}(e_i, \partial_t e_j)\right]\diff\mu - |A^{0}|^2H\xi\diff\mu \\
			&\qquad = 2 A^{0}(e_i, e_j) \left[\nabla^2_{ij}\xi - g_{ij}|A^{0}|^2 \xi\right]\diff \mu + 2 |A^{0}|^2H\xi\diff \mu - |A^{0}|^2H\xi\diff\mu \\
			&\qquad = 2 \nabla_i(\nabla_j \xi A^{0}(e_i, e_j)) \diff \mu - \nabla_j\xi \nabla_j H\diff\mu + |A^{0}|^2H\xi\diff \mu,
	\end{align}
	where we used \eqref{eq:nabla H A A^0} and the fact that $A^0_{ij}(\nabla^2_{ij}\xi)^{0}= A^0_{ij}\nabla^2_{ij}\xi$ as $A^0$ is trace-free. Consequently we find
	\begin{align}
		\partial_t\left(|A^{0}|^2\diff \mu\right) &= 2\nabla_i(\nabla_j \xi A^{0}(e_i, e_j)) \diff \mu- \nabla_j(\xi \nabla_j H) \diff \mu + \xi \Delta H \diff \mu + |A^{0}|^2 H\xi\diff \mu  \\
		&= 2\nabla_i(\nabla_j \xi A^{0}(e_i, e_j)) \diff \mu- \nabla_j(\xi \nabla_j H) \diff \mu +  \nabla_{sc}\overline{\CalW}(f)\xi \diff \mu \\	
		&= 2\nabla_i(\nabla_j \xi A^{0}(e_i, e_j)) \diff \mu- \nabla_j(\xi \nabla_j H) \diff \mu - \abs{\nabla\overline{\CalW}(f)}^2\diff \mu \\
		&\quad + \lambda  \nabla_{sc}\overline{\CalW}(f) \diff \mu.
	\end{align}
	Again, integration by parts and \eqref{eq:nabla H A A^0} yield
	\begin{align}
		&\partial_t \int |A^{0}|^2\eta\diff \mu + \int \abs{\nabla\overline{\CalW}(f)}^2\eta \diff \mu \\
		&\qquad = \int \left[ -2 \nabla_j \xi A^{0}_{ij}\nabla_i \eta + \xi \nabla_j H\nabla_j\eta + \lambda  \nabla_{sc}\overline{\CalW}(f)\eta \right] \diff \mu +\int |A^{0}|^2 \partial_t \eta \diff \mu\\
		&\qquad = \int 2 \xi (\nabla_j A^{0}){(e_i, e_j)}\nabla_i \eta \diff \mu + 2 \xi A^{0}(e_i, e_j)\nabla^2_{ji} \eta \diff \mu + \int \xi \nabla_j H\nabla_j\eta \diff \mu \\
		&\qquad \quad  + \lambda\int \nabla_{sc}\overline{\CalW}(f)\eta  \diff \mu + \int|A^{0}|^2\partial_t \eta\diff\mu\\
		&\qquad = -2\int  \nabla_{sc}\overline{\CalW}(f) \nabla_i H\nabla_i \eta\diff \mu + 2\lambda\int \nabla_i H\nabla_i \eta\diff \mu  - 2 \int \nabla_{sc}\overline{\CalW}(f) A^{0}_{ij}\nabla^2_{ji} \eta \diff \mu\\
		&\qquad \quad + 2\lambda\int A^{0}_{ij}\nabla_{ij}^2\eta\diff \mu + \lambda \int \Delta H \eta\diff \mu +\lambda\int |A^{0}|^2 H\eta\diff \mu + \int|A^{0}|^2\partial_t \eta\diff\mu.\\
		&\qquad = -2\int  \nabla_{sc}\overline{\CalW}(f) \nabla_i H\nabla_i \eta\diff \mu - 2 \int  \nabla_{sc}\overline{\CalW}(f) A^{0}_{ij}\nabla^2_{ji} \eta \diff \mu + \int|A^{0}|^2\partial_t \eta\diff\mu\\
		&\qquad \quad + \lambda\left[ - \int \nabla_i H \nabla_i \eta\diff \mu + 2 \int \nabla_i H \nabla_i\eta\diff \mu+ \int \Delta H \eta\diff\mu + \int|A^{0}|^2 H \eta\diff \mu\right]
	\end{align}
	{The claim follows from integrating by parts in the terms involving $\lambda$.}
\end{proof}
}

\section{Localized energy estimates}\label{sec:loc int est}
In this section, we will use the interpolation inequalities developed in \cite{KSGF,KSSI}. As we shall see, control over the {concentration of curvature} {and $\lambda$ enables us} to estimate derivatives of arbitrary order of the second fundamental form.

In the following, we restrict to a particular class of test functions. Let $\tilde{\gamma}\in C^{\infty}_c(\R^3)$ with $0\leq \tilde{\gamma}\leq 1$ and $\norm{D\tilde{\gamma}}{\infty}\leq \Lambda, \norm{D^2\tilde{\gamma}}{\infty}\leq \Lambda^2$ for some $\Lambda>0$. Then setting
\begin{align}
\gamma&\defeq \tilde{\gamma}\circ f \colon[0,T)\times \Sigma\to\R\text{ we find }\\
\abs{\nabla\gamma}&\leq C\Lambda \text{ and }|\nabla^2\gamma| \leq C\Lambda^2 +C |A|\Lambda \label{eq:gamma},
\end{align}
\add{for a universal constant $C\in (0,\infty)$. Note} that $\gamma_t$ has compact support in space for all $0\leq t<T$. The estimates in \eqref{eq:gamma} follow by the identities
\begin{align}\label{eq:Derivativesgamma tildegamma}
\nabla \gamma &= (D\tilde{\gamma} \circ f) Df \\
\nabla^2 \gamma &= (D^2\tilde \gamma \circ f)(Df \cdot, Df\cdot) + (D\tilde{\gamma}\circ f) A(\cdot, \cdot).
\end{align}
{Unless specified otherwise, constants $C\in (0, \infty)$ are always universal and are allowed to change from line to line.}

\add{Following the strategy in \cite[Secion 3]{KSGF}, we can prove the following}

\begin{lem}\label{lem:curvatureIntegralsEstimate} Let $f\colon[0,T)\times \Sigma\to\R^3$ be a volume-preserving Willmore flow. We have
	\begin{align}
	\partial_t \int \frac{1}{2}\abs{H}^2\gamma^4 \diff \mu + \frac{1}{2}\int \abs{\nabla\overline{\CalW}(f)}^2\gamma^4\diff \mu &\leq C\Lambda^2 \int |A|^2H^2\gamma^2\diff \mu + C\Lambda^4\int_{[\gamma>0]}H^2\diff \mu \\
	& \quad	+ \lambda\int |A^{0}|^2 H \gamma^4\diff \mu + C\Lambda\abs{\lambda}\int H^2 \gamma^3 \diff \mu,
	\end{align}
	for some universal constant $C$ with $0<C<\infty$ and
	\begin{align}
	\partial_t \int|A^0|^2\gamma^4\diff \mu + \frac{1}{2} \int\abs{\nabla\overline{\CalW}(f)}^2\gamma^4\diff \mu &\leq C\Lambda^2 \int |A^0|^2 |A|^2 \gamma^2\diff \mu + C\Lambda^4 \int_{[\gamma>0]} \abs{A}^2\diff \mu \\
	&\quad + \lambda \int|A^0|^2H\gamma^4\diff \mu +C\Lambda \abs{\lambda}\int|A^0|^2\gamma^3\diff \mu.
	\end{align}
\end{lem}

\begin{proof}
	\add{See \Cref{sec:appL31}.}
\end{proof}

Under the assumption of non-concentrated curvature, the following  estimate by Kuwert and Schätzle allows us to locally control derivatives up to second order of the second fundamental form by the localized Willmore gradient and the localized energy. \add{In the form stated below, it follows directly from \cite[Proposition 2.6 and Lemma 4.2]{KSSI}.}

\begin{prop}[\add{\cite{KSSI}}]\label{prop:KSSIProp2.6MitA^6}
	There exist absolute constants $\varepsilon_0,C \in (0,\infty)$ such that if \linebreak $f\colon\Sigma\to\R^3$ is an immersion with
	\begin{align}
	\int_{[\gamma>0]} \abs{A}^2 \diff \mu <\varepsilon_0,
	\end{align}
	for some $\gamma$ as in \eqref{eq:gamma}, then  we have
	\begin{align}
	\int \left(\abs{\nabla^2A}^2+\abs{A}^2\abs{\nabla A}^2 + \abs{A}^6\right)\gamma^{4}\diff \mu \leq C\int \abs{\nabla\overline{\CalW}(f)}^2\gamma^4 \diff \mu + C \Lambda^4 \int_{[\gamma>0]} \abs{A}^2\diff \mu.
	\end{align}
\end{prop}
\delete{
\begin{proof}
	Follows from {\cite[Proposition 2.6]{KSSI}} and the estimate in \Cref{lem:A^6Estimate}.
\end{proof}
}
{This will be the crucial tool in studying the volume-preserving Willmore flow if the \emph{concentration of curvature} is controlled, cf. \Cref{sec:BlowUp,sec:Int est Lambda,sec:lifespan}. Note that in \Cref{lem:curvatureIntegralsEstimate}, a term involving $\lambda$ and a cubic power of $A$ occur. However, the energy decay only allows us to control square powers of $A$, hence we have to pay the price in terms of a higher power of the Lagrange multiplier. 
}
\begin{prop}\label{prop:NewLemma19}
	Suppose $f\colon [0,T)\times \Sigma\to\R^3$ is a volume-preserving Willmore flow. If
	\begin{align}
	\int_{[\gamma>0]} \abs{A}^2\diff \mu <\varepsilon_0 \quad\text{at time }t\in [0,T),
	\end{align}
	where $\varepsilon_0>0$ is as in \Cref{prop:KSSIProp2.6MitA^6} and $\gamma$ is as in \eqref{eq:gamma}, then we have
	\begin{align}
	&\partial_t \int \abs{A}^2\gamma^4\diff \mu + c_0 \int\left(\abs{\nabla^2 A}^2+ \abs{A}^2\abs{\nabla A}^2+\abs{A}^6\right)\gamma^4 \diff \mu \\
	&\qquad \leq C\Lambda^4 \int_{[\gamma>0]} \abs{A}^2\diff \mu + C \abs{\lambda}^{\frac{4}{3}} \int \abs{A}^2\gamma^4\diff \mu
	\end{align}
	at time $t$ for some universal constants $c_0,C\in (0,\infty)$.
\end{prop}
\begin{proof}
	\delete{First, by \eqref{eq:AA0H} and \Cref{lem:curvatureIntegralsEstimate}  we have
	\begin{align}
	&\partial_t \int \abs{A}^2\gamma^4 \diff \mu + \int\abs{\nabla\overline{\CalW}(f)}^2\gamma^4\diff \mu \\
	&\quad \leq C\Lambda^2 \int |A|^4\gamma^2\diff \mu + C \Lambda^4 \int_{[\gamma>0]} |A|^2\diff \mu + 2 \lambda\int |A^0|^2 H\gamma^4\diff \mu + C\Lambda\abs{\lambda} \int |A|^2\gamma^3\diff\mu.
	\end{align}
	Combining this with \Cref{prop:KSSIProp2.6MitA^6} {and estimating $\abs{A^0}^2\abs{H}\leq C\abs{A}^3$ by \eqref{eq:AA0H}} yields}\add{Combining \Cref{lem:curvatureIntegralsEstimate}, \Cref{prop:KSSIProp2.6MitA^6} and \eqref{eq:AA0H}, we find}
	\begin{align}
	&\partial_t \int \abs{A}^2\gamma^4 \diff \mu +c_0\int \left(\abs{\nabla^2A}^2+\abs{A}^2\abs{\nabla A}^2 + \abs{A}^6\right)\gamma^{4}\diff \mu \\
	&\quad \leq C\Lambda^2 \int |A|^4\gamma^2\diff \mu +C \Lambda^4 \int_{[\gamma>0]} |A|^2\diff \mu + C \abs{\lambda}\int |A|^3\gamma^4\diff \mu + C\Lambda\abs{\lambda} \int |A|^2\gamma^3\diff\mu.
	\end{align}
	For the first term on the right hand side above, for  $\varepsilon>0$ we estimate
	\begin{align}
	\Lambda^2 \int \abs{A}^4\gamma^2\diff \mu \leq \varepsilon \int \abs{A}^6\gamma^4\diff \mu +  {C(\varepsilon)}\Lambda^4 \int_{[\gamma>0]} \abs{A}^2\diff \mu.
	\end{align}
	{For the third term, we use Young's inequality with $p=4, q=\frac{4}{3}$ to obtain}
	\begin{align}
	\lambda\int \abs{A}^{\frac{3}{2}+\frac{3}{2}}\gamma^4\diff \mu \leq \varepsilon \int \abs{A}^6 \gamma^4\diff \mu + C(\varepsilon) \abs{\lambda}^{\frac{4}{3}} \int \abs{A}^2 \gamma^4\diff \mu.
	\end{align}
	Similarly for the fourth term, we find
	\begin{align}
	\Lambda \abs{\lambda} \int\abs{A}^{\frac{1}{2}+\frac{3}{2}}\gamma^3\diff \mu \leq C   \Lambda^4 \int_{[\gamma>0]}\abs{A}^2\diff \mu+C \abs{\lambda}^{\frac{4}{3}} \int\abs{A}^2\gamma^4\diff \mu.
	\end{align}
	Taking $\varepsilon>0$ small enough and absorbing yields the claim.
\end{proof}
The integrated form of \Cref{prop:NewLemma19} will be particularly useful.

\begin{cor}\label{prop:New30}
	Let $f\colon [0,T)\times \Sigma\to\R^3$ be a volume-preserving Willmore flow such that for  $\varepsilon_0>0$ as in \Cref{prop:KSSIProp2.6MitA^6} and $\gamma$ as in \eqref{eq:gamma} we have
	\begin{align}
	\int_{[\gamma>0]} \abs{A}^2\diff \mu  \leq \varepsilon< \varepsilon_0\quad \text{for all }0\leq t<T.
	\end{align}
	Then there exist universal constants  $c_0, C\in (0,\infty)$ such that for all $0\leq t< T$ we have
	\begin{align}
	&\int_{[\gamma=1]}\abs{A}^2\diff \mu + c_0 \int_{0}^t \int_{[\gamma=1]}\left(\abs{\nabla^2A}^2 + \abs{A}^2\abs{\nabla A}^2 + \abs{A}^6\right)\diff \mu \\
	&\qquad\leq \int_{[\gamma_0>0]}\abs{A_0}^2\diff \mu_0 + C\Lambda^4\varepsilon t + C \varepsilon\int_{0}^t |\lambda(\tau)|^{\frac{4}{3}}\diff \tau.\label{eq:New30 statement}
	\end{align}
	{Here we used the notation $\int_{[\gamma_0>0]}\abs{A_0}^2\diff \mu_0 = \left.\int_{[\gamma>0]}\abs{A}^2\diff \mu\right\vert_{t=0}$.}
\end{cor}
Note that in order to bound the left hand side of \eqref{eq:New30 statement} up to time $t=T$, the control of the curvature concentration alone does not suffice. Recalling the nonlocal nature of the evolution \eqref{eq:VpWF}, this is not entirely surprising. However, the above result shows that this lack of control can be compensated, if in addition we can bound the $L^{4/3}(0,T)$-norm of $\lambda$, a spatially global quantity, which behaves correctly under parabolic rescaling by \Cref{rem:ParabolicScaling}.
We will discuss under which assumptions $\lambda\in L^{4/3}(0,T)$ can be guaranteed in \Cref{sec:Int est Lambda}.

\delete{As a next step, we now want to establish a higher order version of \Cref{lem:curvatureIntegralsEstimate}. Again the $L^{4/3}$-norm of the Lagrange multiplier will play a crucial role. First, we estimate the evolution of the derivatives of $A$.}
\delete{
\begin{lem}\label{lem:EvolutionOfNabla^mAIntegrals}
	Let $f\colon [0,T)\times \Sigma\to\R^3$ be a volume-preserving Willmore flow and $\gamma$ as in \eqref{eq:gamma}. Then for $\phi=\nabla^m A, m\in \N_0$ and $s\geq 2m+4$ we have
	\begin{align}
	&\frac{\diff}{\diff t} \int |\phi|^2\gamma^s\diff \mu + \frac{1}{2} \int \abs{\nabla^2 \phi}\gamma^s\diff \mu \\
	&\leq  C\left( \abs{\lambda}^{\frac{4}{3}}  + \norm{A}{L^{\infty}([\gamma>0])}^4\right) \int\abs{\phi}^2\gamma^s\diff \mu  + C\left(1+ \abs{\lambda}^{\frac{4}{3}}+ \norm{A}{L^{\infty}([\gamma>0])}^4\right)\int_{[\gamma>0]}\abs{A}^2\diff\mu
	\end{align}
	where $C=C(s,m, \Lambda)>0$.
\end{lem}
\begin{proof}
	{See \Cref{sec:Higherorder}.}
\end{proof}
}
\add{As in \cite{KSSI}, an appropriate higher order version of \Cref{prop:New30} can be used to prove higher order interior estimates. 
}

\delete{We conclude this section with higher order estimates which do not explicitly depend on the initial datum $f_0$, cf. \cite[Theorem 3.5]{KSSI}. These will be crucial for the blow-up construction in \Cref{sec:BlowUp}.} 
\begin{prop}\label{thm:HigherOrderEstimatesTLocalized}
	Let $f\colon [0,T) \times \Sigma\to\R^3$ be a volume-preserving Willmore flow. Suppose $\rho>0$ is chosen such that $T\leq T^{\ast} \rho^4$ for some $0<T^{\ast}<\infty$ and
	\begin{align}
		\int_{B_{\rho}(x)} \abs{A}^2\diff \mu\leq \varepsilon<\varepsilon_0\quad \text{for all }0\leq t<T,
	\end{align}
	where $x\in \R^3$, $\varepsilon_0>0$ is as in \Cref{prop:KSSIProp2.6MitA^6} and $\int_{0}^T \abs{\lambda}^{\frac{4}{3}}\diff t \leq L <\infty$. Then, for all $m\in \N_0$ and $t\in (0,T)$ we have the estimates
	\begin{align}
		\norm{\nabla^m A}{L^2(B_{\rho/2}(x))} &\leq C(m,T^{*},L)\sqrt{\varepsilon} t^{-\frac{m}{4}}, \\
		\norm{\nabla^m A}{L^{\infty}(B_{\rho/2}(x))} &\leq C(m,T^{*},L) \sqrt{\varepsilon} t^{-\frac{m+1}{4}}.\label{eq:highorder local in time estimate}
	\end{align}
\end{prop}
\add{The proof of \Cref{thm:HigherOrderEstimatesTLocalized} is essentially the same as in \cite[Theorem 3.5]{KSSI}, so it is moved to \Cref{sec:Higherorder}.}

\section{Integral estimates for the Lagrange multiplier}\label{sec:Int est Lambda}
In \Cref{thm:HigherOrderEstimatesTLocalized}, we were able to control all derivatives of the second fundamental form, if the concentration of curvature is sufficiently small and the Lagrange multiplier has some sort of integrability. This section is devoted to showing time integrability of $\lambda$ under certain assumptions. 

\subsection{The \texorpdfstring{$L^{4/3}(0,T)$}{L^(4/3)}-norm of \texorpdfstring{$\lambda$}{lambda} in the case of non-concentration}\label{sec:Int est Lambda L43}

First, we will control the $L^{4/3}$-norm of $\lambda$, which will be the key ingredient in the proof of the lifespan result in \Cref{thm:Lifespan}. We begin by making the following observation for immersions with non-concentrated curvature.

\begin{lem}\label{lem:control rho via V}
	There exists an absolute constant $0<\varepsilon_1<8\pi$ such that if $f\colon \Sigma\to\R^3$ is an immersion, $x_0\in f(\Sigma)$ and $\rho>0$ satisfies 
	\begin{align}
		\int_{B_{\rho}(x_0)} \abs{A}^2\diff \mu \leq \varepsilon <\varepsilon_1,
	\end{align}
	then we have $\rho \leq C \CalA(f)^{\frac{1}{2}}$, where $0<C<\infty$ is an absolute constant.
\end{lem}
\begin{proof}
	By Simon's monotonicity formula \add{\cite[(1.4)]{SimonWillmore}} \delete{\eqref{eq:SimonMono 2}} and \eqref{eq:AA0H} for any $x_0 \in f(\Sigma)$ and some universal constant $0<C<\infty$ we have 
	\begin{align}
		\pi \leq C\left(\rho^{-2}\mu(\add{f^{-1}(B_{\rho}(x_0))})+ \int_{B_{\rho}(x_0)}\abs{H}^2\diff \mu\right) \leq C\rho^{-2}\mu(\add{f^{-1}(B_{\rho}(x_0))}) + 2C\varepsilon_1.
	\end{align}
	For $\varepsilon_1\defeq \frac{\pi}{4C}>0$ we thus find $\frac{\pi}{2}\leq C \rho^{-2}\mu(\add{f^{-1}(B_{\rho}(x_0))})\leq C\rho^{-2} \CalA(f)$.
\end{proof}

\begin{prop}\label{prop:lambda4/3Nonconcentration}
	Let $f\colon [0,T)\times \Sigma\to\R^3$ be a volume-preserving Willmore flow with ${\CalW}(f_0)\leq K$ such that $\rho>0$ satisfies
	\begin{align}
	\sup_{0\leq t\leq T} \int_{B_{\rho}(x)}\abs{A}^2\diff \mu\leq \varepsilon<\varepsilon_2\quad \text{ for all } x\in \R^3,
	\end{align}
	where $\varepsilon_2 \defeq \min\{\varepsilon_0, \varepsilon_1\}\in (0,8\pi)$, with $\varepsilon_0$  as in \Cref{prop:KSSIProp2.6MitA^6} and $\varepsilon_1>0$  as in \Cref{lem:control rho via V}. Then, we have
		\begin{align}
			\int_{0}^t \abs{\lambda}^{\frac{4}{3}}\diff \tau &\leq C(K, \chi(\Sigma))\left(\frac{t^{\frac{1}{2}}}{\rho^2}+\frac{t}{\rho^4}\right)\quad \text{ for all }0\leq t<T.
		\end{align}
\end{prop}
\delete{Note that by the invariance of the Willmore energy and \Cref{rem:ParabolicScaling}, this estimate is preserved under parabolic rescaling.}
\begin{proof}[{Proof of \Cref{prop:lambda4/3Nonconcentration}}]
	First, fix $x\in \R^3$. Let $\tilde{\gamma}\in C^{\infty}_c(\R^3)$ be a bump function with $\chi_{B_{\rho/2}(x)}\leq \tilde{\gamma} \leq \chi_{B_\rho(x)}$, $\norm{D\tilde{\gamma}}{\infty}\leq \frac{C}{\rho}$ and $\norm{D^2\tilde{\gamma}}{\infty}\leq \frac{C}{\rho^2}$. Therefore, $\gamma\defeq\tilde{\gamma}\circ f$ is as in \eqref{eq:gamma} with $\Lambda=\frac{C}{\rho}$, and thus by integrating \Cref{prop:NewLemma19} from $0$ to $\tau$ we find
	\begin{align}
		& \left.\int_{B_{\rho/2}(x)}\abs{A}^2\diff \mu\right\vert_{t=\tau} +  c_0 \int_0^\tau \int_{B_{\rho/2}(x)}\left(\abs{\nabla^2 A}^2+ \abs{A}^2\abs{\nabla A}^2+\abs{A}^6\right) \diff \mu \diff t \\
		&\qquad \leq \int_{B_{\rho}(x)} \abs{A_0}^2\diff \mu_0 + \frac{C}{\rho^4}  \int_0^\tau\int_{B_{\rho}(x)} \abs{A}^2\diff \mu \diff t+ C \int_0^\tau\abs{\lambda}^{\frac{4}{3}}  \int_{B_{\rho}(x)} \abs{A}^2\diff \mu\diff t.\label{eq:EstimateGlobalLocal1}
	\end{align}
	It is possible to find $(x_\ell)_{\ell\in\N}\subset \R^3$ with $\R^3= \bigcup_{\ell\in\N}B_{\rho/2}(x_\ell)$ such that each point $y\in \R^3$ is contained in at most $M$ of the balls $B_{\rho}(x_\ell)$, where $M>0$ is a universal constant, in particular independent of $\rho>0$. Therefore, choosing $x=x_\ell$ in \eqref{eq:EstimateGlobalLocal1} and summing over $\ell\in \N$ we find
	\begin{align}
		\begin{split}
		\int_0^\tau \int\abs{A}^6\diff \mu\diff t &\leq \sum_{\ell} \int_0^\tau\int_{B_{\rho/2}(x_\ell)}\abs{A}^{6}\diff \mu\diff t  \\
		&\leq M \int\abs{A_0}^2\diff \mu_0 + \frac{CM}{\rho^{4}}\int_0^\tau \int|A|^2\diff \mu\diff t + CM \int_0^\tau \abs{\lambda}^{\frac{4}{3}}\int \abs{A}^2\diff \mu\diff t
		\end{split}
	\end{align}
	Now, by \eqref{eq:A^2GaussBonnet} we have $\int\abs{A}^2\diff \mu\leq C(K, \chi(\Sigma))$ and hence
	\begin{align}\label{eq:GlobalA^6Estimate}
		\int_0^{\tau}\int \abs{A}^{6}\diff \mu\diff t \leq C\left(K,\chi(\Sigma)\right) \left(1+\frac{\tau}{\rho^4}+ \int_0^\tau \abs{\lambda}^{\frac{4}{3}}\diff t\right).
	\end{align}
	Thus, using \eqref{eq:AA0H}, H\"older's inequality and \Cref{lem:control rho via V} we find from \eqref{eq:deflambda}
	\begin{align}
		\int_0^\tau \abs{\lambda}^{\frac{4}{3}}\diff t &\leq C\int_0^\tau \CalA(f_t)^{-\frac{4}{3}} \left( {\int \abs{A}^3\diff \mu}\right)^{\frac{4}{3}}\diff t \leq C \int_0^{\tau} \mathcal{A}(f_t)^{-1} \left(\int \abs{A}^4\diff \mu\right)\diff t \\
		& \leq C \rho^{-2} \left(\int_0^{\tau} \int\abs{A}^6\diff \mu\diff t\right)^{\frac{1}{2}} \left(\int_0^{\tau} \int \abs{A}^2\diff \mu \diff t\right)^{\frac{1}{2}}.
	\end{align}
%
%
Therefore, using \eqref{eq:A^2GaussBonnet}, the energy decay \eqref{eq:dtW<=0}, \eqref{eq:GlobalA^6Estimate} and Young's inequality, we find 
	\begin{align}
		\int_0^{\tau}\abs{\lambda}^{\frac{4}{3}}\diff t &\leq C(K,\chi(\Sigma)) \frac{\tau^{\frac{1}{2}}}{\rho^2} \left(1+ \frac{\tau^{\frac{1}{2}}}{\rho^2}+ \left(\int_0^{\tau}\abs{\lambda}^{\frac{4}{3}}\diff t\right)^{\frac{1}{2}}\right) \\
		&\leq C(K, \chi(\Sigma)) \left( \frac{\tau^{\frac{1}{2}}}{\rho^2} + \frac{\tau}{\rho^4} \right) + \frac{1}{2} \int_0^{\tau}\abs{\lambda}^{\frac{4}{3}}\diff t. &&\qedhere
	\end{align}
\end{proof}

\subsection{An \texorpdfstring{$L^2$}{L^2}-type estimate}\label{subsec:L^2estimate}
In this section, we prove an $L^2(0,T)$-type estimate for $\lambda$, which will be crucial in the analysis of the blow-ups in \Cref{sec:BlowUp}. \add{Since we rely on a \emph{reverse isoperimetric inequality} \cite{blatt2020reverse}, this is the first instance where we require the Willmore energy to be below $8\pi$.}

\delete{Since we can control the volume but not necessarily the area along the flow, the following \emph{reverse isoperimetric inequality} is a key ingredient. Note that while the case of spherical surfaces already follows from \cite[Theorem 1]{Schygulla}, a very recent generalization shows that the assumption on the topology is not necessary and also provides a quantitative estimate, cf \cite[Theorem 1.1]{blatt2020reverse}.

\begin{thm}[Reverse Isoperimetric Inequality]\label{thm:ReverseIsoperimetric}
	For each $\delta>0$ there exists a constant $C(\delta)>0$ such that if $f\colon \Sigma\to\R^3$ is an embedding with $\CalW(f)\leq 8\pi-\delta$ we have
	\begin{align}
		\CalA(f)^{\frac{1}{2}} \leq C(\delta) \abs{\CalV(f)}^{\frac{1}{3}}.
	\end{align}
\end{thm}
\begin{remark}
	By the Li--Yau inequality \eqref{eq:LiYau}, it is no restriction to require $f$ to be an embedding in \Cref{thm:ReverseIsoperimetric}. Moreover, we observe that in order to prove our main result, \Cref{thm:conv main}, we only need the result in \cite{Schygulla} for $\Sigma=\S^2$.
\end{remark}

As a next step, we recall the following essential diameter estimate due to Simon, see \cite[Lemma 1]{ToppingDiamter} for the explicit constant.
\begin{prop}[{\cite[Lemma 1.1]{SimonWillmore}}]\label{prop:ToppingDiameter}
	If $f\colon\Sigma\to\R^3$ is an immersion, then we have
	\begin{align}
		\diam(f(\Sigma)) \leq \frac{2}{\pi} \CalA(f)^{\frac{1}{2}}\CalW(f)^{\frac{1}{2}}.
	\end{align}
\end{prop}
}

As a \add{first} step, we want to relate the diameter to the Lagrange multiplier. To that end, we use the different scaling of $\overline{\CalW}$ and $\CalV$ to obtain a different representation of  $\lambda$, cf. \cite[pp. 1236 -- 1237]{DKS} \add{and also \cite[Proof of Theorem 1.4]{MWHelfrich}.}
\begin{lem}\label{lem:lambdascaling} Let $f\colon[0,T)\times \Sigma\to\R^3$ be a volume-preserving Willmore flow. Then for all $t\in [0,T)$ and any $p\in \R^3$ we have 
	\begin{align}
		3\lambda \CalV(f_0)= -\int \langle \partial_t f, f\rangle\diff \mu = -\int \langle \partial_t f, f-p\rangle\diff \mu.
	\end{align}
\end{lem}
\begin{proof}
	Fix $t\in [0,T)$. For $\alpha>0$, consider the immersion $h_\alpha \defeq p+\alpha(f_t-p)\colon\Sigma\to\R^3$. We then have $\overline{\CalW}(h_\alpha)=\overline{\CalW}(f_t)$, whereas $\CalV(h_{\alpha}) = \alpha^3\CalV(f_0)$. Thus, we find
	\begin{align}
		\left.\frac{\diff}{\diff \alpha}\right\vert_{\alpha=1}\big(\overline{\CalW}(h_{\alpha})+ \lambda(t) \CalV(h_\alpha)\big) = 0 + 3\lambda(t) \CalV(f_0),
	\end{align}
	whereas by the definition of $L^2(\diff \mu_f)$-gradients we have
	\begin{align}
		\left.\frac{\diff}{\diff \alpha}\right\vert_{\alpha=1}\big(\overline{\CalW}(h_{\alpha})+ \lambda(t) \CalV(h_\alpha)\big) = \left.\int\langle \nabla\overline{\CalW}(f)+\lambda\nabla\CalV(f), f-p\rangle\diff \mu\right\vert_{t}.
	\end{align}
	Therefore, by \eqref{eq:VpWF} and \Cref{lem:first Vari V and W} we have the identity
	\begin{align}
		3 \lambda \CalV(f_0) = -\int \langle \partial_t f, f-p\rangle\diff \mu \quad\text{on }[0,T).
	\end{align}
	Picking $p=0$ yields the first equality.
\end{proof}

This finally enables us to prove the desired $L^2$-estimate.
\begin{prop}\label{prop:Lambda^2AVpWF}
	Let $f\colon[0,T)\times \Sigma \to\R^3$ be a volume-preserving Willmore flow with ${\CalW}(f_0)\leq 8\pi-\delta$ for $\delta>0$. Then, for all $0\leq t <T$ we have
	\begin{align}
		\int_{0}^t\lambda^2(\tau)\CalA(f_\tau)\diff \tau\leq C(\delta)\CalW(f_0)^2.
	\end{align} 
\end{prop}
\delete{Again, by \Cref{rem:ParabolicScaling} and the invariance of the Willmore, the above estimate is preserved under parabolic rescaling.}
\begin{proof}
	{Observe that by \eqref{eq:dtV=0}, we have $\abs{\CalV(f)}=\abs{\CalV(f_0)}$.
	Picking some $p(t)\in f_t(\Sigma)$ for each $t$, we find by \Cref{lem:lambdascaling} and Cauchy--Schwarz
	\begin{align}
		\abs{\lambda(t)}\leq \frac{1}{3 \abs{\CalV(f_0)}} \int \abs{\partial_t f_t}\diff \mu_t \diam f_t(\Sigma) \leq \frac{\CalA(f_t)^{\frac{1}{2}}}{3  \abs{\CalV(f_0)}} \left(\int \abs{\partial_t f_t}^2\diff \mu_t\right)^{\frac{1}{2}} \diam f_t(\Sigma).
	\end{align}
	Squaring this inequality, \add{by Simon's diameter estimate \cite[Lemma 1.1]{SimonWillmore}} we conclude
	\begin{align}
		\lambda^2(t)\CalA(f_t) &\leq \frac{C}{\abs{\CalV(f_0)}^2} \CalA(f_t)^3 \CalW(f_t)\int \abs{\partial_t f_t}^2\diff \mu_t.
	\end{align}
	Now, by \add{the reverse isoperimetric inequality \cite[Theorem 1.1]{blatt2020reverse} (see also \cite[Theorem 1]{Schygulla} for the spherical case),} and the assumption on the initial energy, we have 
	\begin{align}
		\lambda^2(t)\CalA(f_t) &\leq {C(\delta)}\CalW(f_t)\int \abs{\partial_t f_t}^2\diff \mu_t.
	\end{align}
}
	Integrating in time and using \eqref{eq:dtW<=0} and  \eqref{eq:WvstildeW} we conclude
	\begin{align}
		&\int_{0}^\tau \lambda^2(t) \CalA(f_t)\diff t \leq C(\delta)\int_0^\tau \CalW(f_t)\left(- \int \langle \nabla\overline{\CalW}(f_t), \partial_t f_t\rangle \diff \mu_t\right) \diff  t \\
		&\qquad =  -C(\delta) \int_{0}^\tau \CalW(f_t) \partial_t \overline{\CalW}(f_t)\diff t  =C(\delta) \int_{0}^{\tau}- \partial_t\left(\CalW(f_t)\right)^2  \diff t \leq C(\delta) \CalW(f_0)^2.
	\end{align}
	Renaming $\tau$ into $t$ yields the claim.
\end{proof}

\section{Proof of the lifespan theorem}\label{sec:lifespan}
In this section, we will prove \Cref{thm:Lifespan}, which yields a lower bound on the maximal existence time of the volume-preserving Willmore flow. This will be crucial for the construction of the blow-up in \Cref{sec:BlowUp}.

\delete{Note that in contrast to the lifespan estimate in \cite{McCoyWheelerWilliams} for the surface diffusion flow, }
Here we only work with the integrability of the constraint parameter $\lambda$ which we proved in \Cref{sec:Int est Lambda} \add{and do not require strong $L^\infty$-type bounds as in \cite[(A1)]{McCoyWheelerWilliams}, \cite[(7)]{MR2735555}.}

\add{
\begin{proof}[Proof of \Cref{thm:Lifespan}]
	This can now be achieved in the same fashion as \cite[Theorem 1.2]{KSGF}, so we focus on the differences arising from the Lagrange multiplier. Without loss of generality, $\rho=1$, cf.\ \Cref{rem:ParabolicScaling}. If $\Gamma>1$ denotes the number of radius $\frac{1}{2}$ balls necessary to cover $B_1(0)\subset \R^3$, we set $\bar{\varepsilon}\defeq \frac{\varepsilon_2}{3\Gamma}$ with $\varepsilon_2>0$ as in \Cref{prop:lambda4/3Nonconcentration}. We observe
	\begin{align}\label{eq:GammaEstimate1/2}
		\varepsilon(t)\defeq \sup_{x\in \R^3} \int_{B_1(x)} \abs{A}^2\diff \mu \leq \Gamma \cdot \sup_{x\in \R^3} \int_{B_{1/2}(x)} \abs{A}^2\diff \mu,
	\end{align}
	and, for a parameter $0<\beta< 1$, to be specified below, define
	\begin{align}\label{eq:deft_0}
		t_0 \defeq \sup\left\{ 0\leq t \leq \min\{T, \beta\} \mid \varepsilon(\tau)\leq 3\Gamma\varepsilon\text{ for all }0\leq \tau <t\right\}>0.
	\end{align}
	Picking an appropriate test function in \Cref{prop:New30}, we obtain
	\begin{align}\label{eq:star}
		\int_{B_{1/2}(x)} \abs{A}^2\diff \mu \leq \int_{B_1(x)}\abs{A_0}^2\diff \mu_0 + 3c \Gamma\Lambda^4 \varepsilon t + 3c \Gamma\varepsilon  \int_{0}^t\abs{\lambda}^{\frac{4}{3}}(\tau)\diff \tau
	\end{align}
	for all $0\leq t<t_0$ where $c, \Lambda\in (0, \infty)$ are universal constants. By the choice of $\bar{\varepsilon}$ and \Cref{prop:lambda4/3Nonconcentration}, the integral $\int_0^t \abs{\lambda}^{\frac{4}{3}}\diff \tau$ grows less than linearly in $t\in [0,t_0)$. Hence, by a suitable application of Young's inequality, we find
\begin{align}
	\int_{B_{1/2}(x)} \abs{A}^2\diff \mu &\leq \int_{B_1(x)}\abs{A_0}^2\diff \mu_0 + 3c \Gamma\Lambda^4\varepsilon t + \frac{\varepsilon}{2} + C(K, \chi(\Sigma), c, \Gamma) t\varepsilon  \\
	&\leq \int_{B_1(x)}\abs{A_0}^2\diff \mu_0 + \frac{\varepsilon}{2} + C(K, \chi(\Sigma),c,\Gamma,\Lambda) t\varepsilon.\label{eq:1/2CurvatureConcentrationEstimate}
\end{align}
 If we choose $\beta^{-1}\defeq 2 C(K, \chi(\Sigma),c,\Gamma,\Lambda)$, the assumption $t_0<\min\{T, \beta\}$ contradicts maximality of $t_0$ in \eqref{eq:deft_0}. Consequently, $t_0= \min \{T, \beta\}$ has to hold.  
 If $t_0=\beta$, we find $T\geq \beta$. In this case, \eqref{eq:LifeSpanAEstimate} then follows from \eqref{eq:GammaEstimate1/2}, \eqref{eq:1/2CurvatureConcentrationEstimate} and the definition of $\beta$.
 
 Assume $t_0 = T \leq \beta$. Then from \eqref{eq:1/2CurvatureConcentrationEstimate} we find $\int_{B_{1/2}(x)} \abs{A}^2\diff \mu \leq 2\varepsilon$, hence by \eqref{eq:GammaEstimate1/2}, we have
 \begin{align}\label{eq:estimate A 1.7}
 	\varepsilon(t)\leq 2\Gamma\varepsilon<\varepsilon_0 \quad \text{for all }0\leq t <t_0.
 \end{align}
 Now, $T\leq \beta$ by assumption and $L=\int_0^{T}\abs{\lambda}^{\frac{4}{3}}\diff t\leq C(K, \chi(\Sigma))$ by \Cref{prop:lambda4/3Nonconcentration}. We may use \Cref{thm:HigherOrderEstimatesTLocalized} and argue exactly as in \cite[Theorem 1.2]{KSGF} to prove $f(t)\to f(T)$ smoothly as $t\nearrow T$, which enables us to smoothly extend the flow past $T$. Taking $\hat{c}\in (0,\beta)\subset (0,1)$ small enough, \eqref{eq:estimate A 1.7} guarantees that \eqref{eq:LifeSpanAEstimate} is satisfied.
\end{proof}
}

\section{Construction of the blow-up}\label{sec:BlowUp}
In this section, we will rescale a volume-preserving Willmore flow as we approach the maximal existence time to obtain a blow-up limit,  combining the approaches in \cite[Section 4]{KSSI} and \cite[pp. 348 -- 349]{KSRemovability}. As we shall see, if the Lagrange multiplier has a certain integrability in time, then the limit is not only stationary, but even an \emph{unconstrained} Willmore immersion.
\begin{defi}\label{def:curvature concentration function}
For a smooth family of immersions $f \colon [0,T)\times \Sigma\to\R^3$, $t\in [0,T), r>0$, we define the \emph{curvature concentration function}
\begin{align}
\kappa(t,r)\defeq \sup_{x\in \R^3} \int_{B_r(x)}	|A_t|^2\diff \mu_t.
\end{align}
\end{defi}
\delete{For $t\in [0,T)$ and $\rho>0$, like in \cite[p. 348]{KSRemovability} we have the estimates
\begin{align}\label{eq:kappa continuity}
\lim_{r\nearrow\rho}\kappa(t, r) = \kappa(t, \rho)\leq 	\liminf_{r\searrow \rho}\kappa(t,r) \leq \limsup_{r\searrow \rho}\kappa(t,r)\leq \Gamma \kappa(t, \rho),
\end{align}
where $\Gamma>1$ is as in \eqref{eq:GammaEstimate1/2} by a covering argument.}

\begin{thm}\label{thm:BlowUpExistence}
Let $f\colon[0,T)\times \Sigma\to\R^3$ be a maximal volume-preserving Willmore flow with initial energy ${\CalW}(f_0)\leq K$. Let $(t_j)_{j\in \N}\subset [0,T), t_j \nearrow T, (r_j)_{j\in \N}\subset(0, \infty)$, $(x_j)_{j\in \N}\subset \R^3$ such that
\begin{align}\label{eq:blowup assumption}
	\kappa(t_j, r_j)&\leq \varepsilon_3 \defeq \bar{\varepsilon} \hat{c} \quad \text{ for all }j\in\N,
\end{align} 
where $\bar{\varepsilon}>0$ and $\hat{c}=\hat{c}(K, \chi(\Sigma))  \in (0, 1)$ are as in \Cref{thm:Lifespan}. Then we find
\begin{align}\label{eq:t_j+cr_j<T}
t_j+r_j^4\hat{c}<T,
\end{align}
and after passing to a subsequence,  the rescaled and translated immersions
\begin{align}
&\hat{f}_j\defeq r_j^{-1}\left(f(t_j+r_j^4\hat{c},\cdot)-x_j\right)
\end{align}
converge as $j\to\infty$ smoothly on compact subsets of $\R^3$, after reparametrization, to a proper constrained Willmore immersion $\hat{f}\colon\hat{\Sigma}\to\R^3$ of \eqref{eq:stationary}  with $\CalW(\hat{f})\leq K$.

Moreover, if $\int_0^T \lambda(t)^2\CalA(f_t)\diff t<\infty$,  then $\hat{f}$ is an \emph{unconstrained} Willmore {immersion}.
\end{thm}

\add{Note that while we cannot apply \Cref{lem:constr Willmore} to the limit immersion, under the $L^2$-integrability condition above, we still find that the Willmore part of the evolution dominates in the blow-up.}

\delete{The last part of the statement is particularly remarkable, since we cannot simply apply \Cref{lem:constr Willmore} as the limit surface $\hat{\Sigma}$ is not necessarily compact and could have zero volume.}

\begin{remark}\label{rem:lambda^2A 8pi}
	By \Cref{prop:Lambda^2AVpWF}, the condition $\int_0^T\lambda(t)^2\CalA(f_t)\diff t< \infty$ is automatically satisfied if ${\CalW}(f_0)<8\pi$, i.e.\ if $K<8\pi$ in \Cref{thm:BlowUpExistence}. 
\end{remark}

\begin{remark}
	\add{For general sequences $(t_j)_{j\in\N}, (r_j)_{j\in\N}$ and $(x_j)_{j\in\N}$, the limit may be trivial, for instance, $\hat{\Sigma}=\emptyset$, if $\hat{f}_j$ parametrizes the round spheres $\partial B_1(x_j)$ with $x_j\to\infty$.} In order to make use of the construction, we will select $t_j$ and $x_j$ such that this cannot happen.
\end{remark}
{
Any constrained Willmore immersion $\hat{f}\colon\hat{\Sigma}\to\R^3$ which arises from the process described in \Cref{thm:BlowUpExistence} is called a \emph{concentration limit}. More precisely, we call $\hat{f}$ a \emph{blow-up} if $r_j\to 0$, a \emph{blow-down} for $r_j \to\infty$ and a \emph{limit under translation} if $r_j\to r\in (0, \infty)$. Note that by \eqref{eq:t_j+cr_j<T} the last two can only occur if $T=\infty$.}

\begin{proof}[{Proof of \Cref{thm:BlowUpExistence}}]
For $j\in \N$, we consider the rescaled and translated flows
\begin{align}
{f}_j \colon [-r_j^{-4}t_j, r^{-4}_j(T-t_j))\times \Sigma\to\R^3, \\
{f}_j(t, p) = r_j^{-1}\left(f(t_j+r_j^4 t, p)-x_j\right)
\end{align}
and observe that $f_j$ is a volume-preserving Willmore flow  with initial datum given by $f_j(0)=r_j^{-1}\left(f(t_j, \cdot)-x_j\right)$ and maximal existence time $r_j^{-4}(T-t_j)$. In particular by \Cref{rem:W strict Lyapunov} we have ${\CalW}(f_j(0))\leq K$ for any $j\in \N$. Moreover, by \eqref{eq:blowup assumption} we have
\begin{align}
\sup_{x\in \R^3}\int_{B_1(x)}\abs{A_{f_j(0, \cdot)}}^2\diff\mu_{f_j(0, \cdot)} = \sup_{x\in \R^3}\int_{B_{r_j}(x)}\abs{A_{f_{t_j}}}^2\diff\mu_{f_{t_j}} \leq \varepsilon_3.
\end{align}
Hence, by \Cref{thm:Lifespan} the maximal existence time of the flow $f_j$ is bounded from below by $\hat{c} = \hat{c}(K, \chi(\Sigma))$ and \eqref{eq:t_j+cr_j<T} follows. Furthermore, \eqref{eq:LifeSpanAEstimate} yields
\begin{align}
\sup_{x\in \R^3}\int_{B_1(x)}\abs{A_{f_j(t, \cdot)}}^2\diff\mu_{f_j(t, \cdot)} \leq \bar{\varepsilon}<\varepsilon_2 \quad \text{for all }0\leq t\leq \hat{c},
\end{align}
{using that $\bar{\varepsilon}<\varepsilon_2$ by definition (cf. Proof of \Cref{thm:Lifespan}), where $\varepsilon_2>0$ is as in \Cref{prop:lambda4/3Nonconcentration}.}
Consequently, by \Cref{prop:lambda4/3Nonconcentration} the $L^{4/3}(0,\hat{c})$-norm of the Lagrange multiplier of ${f}_j$ is bounded by $C(K,\chi(\Sigma))$ for any $j\in \N$.
Therefore, using $\bar{\varepsilon}<\varepsilon_2\leq \varepsilon_0$ (cf. \Cref{prop:lambda4/3Nonconcentration})  by \Cref{thm:HigherOrderEstimatesTLocalized} we find
\begin{align}\label{eq:Af_j Cinfty bounds}
\norm{\nabla^m A_{f_j(t, \cdot)}}{\infty} \leq C(m,K, \chi(\Sigma)) t^{-\frac{m+1}{4}} \quad\text{for }0< t\leq \hat{c}.
\end{align}
Moreover, \add{using the scale-invariance of the Willmore energy, \eqref{eq:dtW<=0} and the a-priori energy bound, we can use Simon's monotonicity formula \cite{SimonWillmore} to conclude that}
\begin{align}\label{eq:SimonLocalAreabound}
R^{-2}\mu_{f_j(t, \cdot)}\left(B_R(0)\right) \leq C(K, \chi(\Sigma))<\infty \quad \text{for all }0<t\leq \hat{c}, R>0.
\end{align}
Thus, we may apply \Cref{thm:LocLangerCompactness} and \Cref{cor:langer_complete} to the sequence of immersions $\hat{f}_j \defeq f_j(\hat
c, \cdot)$.

After passing to a subsequence, we thus find a proper limit immersion $\hat{f}\colon\hat{\Sigma}\to\R^3$, where $\hat{\Sigma}$ is a complete surface without boundary, diffeomorphisms $\phi_j\colon\hat{\Sigma}(j)\to U_j$, where $U_j\subset {\Sigma}$ are open sets and $\hat{\Sigma}(j) = \{ p\in \hat{\Sigma}\mid \abs{\hat{f}(p)}<j\}$, and functions $u_j \in C^{\infty}(\hat{\Sigma}(j);\R^3)$ such that we have
\begin{align}\label{eq:hat f representation}
& \hat{f}_j \circ\phi_j = \hat{f} + u_j \quad \text{on }\hat{\Sigma}(j)
\end{align}
as well as $\norm{\hat{\nabla}^m u_j}{L^{\infty}(\hat{\Sigma}(j), \hat{g})}\to 0$ for $j\to\infty$ for all $m\in \N_0$. 

For $j\in\N$, we now define the flows $\tilde{f}_j \defeq f_j\circ\phi_j\defeq f_j(\cdot, \phi_j(\cdot))\colon(0, \hat{c}]\times \hat{\Sigma}(j)\to\R^3$ and observe that they also satisfy the curvature estimates \eqref{eq:Af_j Cinfty bounds}.

\delete{As in \eqref{eq:starstar}, }We use \eqref{eq:Af_j Cinfty bounds} to estimate
\begin{align}\label{eq:lambda bound}
	\abs{\lambda(f_j)} \leq C\norm{A_{f_j}}{\infty}^3\leq C(K, \chi(\Sigma), \xi).
\end{align}

\add{Now, using \eqref{eq:deflambda}, \eqref{eq:Af_j Cinfty bounds} and \eqref{eq:lambda bound}, it is not difficult to also bound $|\partial_t \lambda(f_j(t,\cdot))|$ and thus $\norm{\partial_t \tilde{f}_j(t,\cdot)}{L^{\infty}(\hat{\Sigma}(j))}$ for $t\in [\xi, \hat{c}]$, $j\in \N$ by $C(K, \chi(\Sigma), \xi)$.
From here on, it is a standard procedure to establish $L^{\infty}$-bounds in a local chart $(U, \psi)$ of $\hat{\Sigma}$, i.e.\ estimates of the form
	\begin{align}\label{eq:bounds}
		\norm{\partial^m \partial_t \tilde{f}_j}{L^{\infty}(U)}+ \norm{\partial^{m+1} \tilde{f}_j}{L^{\infty(U)}}\leq C(m,K, \chi(\Sigma), \xi) \quad \text{ for all } t\in[\xi, \hat{c}], m\geq 0, 
	\end{align}
for all $j\geq J$ sufficiently large, where $\partial$ denotes the coordinate derivative in the chart $(U, \psi)$, see for instance \cite[p.~331--332]{KSGF}. By \eqref{eq:def normal}, this also transfers to estimates of the induced normal field $\nu_{\tilde{f}_j}\defeq \nu_{f_j}\circ \phi_j$. 
}
\delete{
Using \eqref{eq:dtg}, \eqref{eq:Af_j Cinfty bounds} and \eqref{eq:dtf_jphi_j} {and \eqref{eq:lambda bound}} we have
\begin{align}\label{eq:dt g bounded}
\norm{\partial_t {g_{\tilde{f}_j}}}{\infty} = \norm{2\langle {A}_{\tilde{f}_j}, \partial_t \tilde{f}_j\rangle}{\infty} \leq C(K, \chi(\Sigma), \xi)
\end{align}
for $\xi\leq t\leq \hat{c}$.
As in \cite[pp. 331 -- 332]{KSGF}, \delete{see also the proof of \Cref{thm:Lifespan}}, for $J\in \N$ and a local chart $(U, \psi)$ with $U\subset \hat{\Sigma}(J)\subset \hat{\Sigma}$, one can use \eqref{eq:Af_j Cinfty bounds}, \eqref{eq:dt g bounded} and \delete{again} \cite[Lemma 14.2]{Hamilton82} to show the bounds
\begin{align}
\norm{\partial^{m+1}\tilde{f}_j}{L^{\infty}(U)}&\leq C(m, K, \chi(\Sigma), \xi) \quad\text{ for all }\xi\leq t\leq \hat{c}, m\geq 0,\\
\omega &\leq \abs{\det \tilde{g}_j} \leq \omega^{-1}\qquad~ \text{ on }[\xi, \hat{c}]\times U,\label{eq:dm tilde f_j bounds}
\end{align}
for all $j\geq J$ and some $\omega = \omega(K, \chi(\Sigma), \xi)>0$, where $\partial$ denotes the coordinate derivative in the chart $(U, \psi)$. }

\delete{We will now estimate the normal. Denote by $(y^1, y^2)$ the local coordinates induced by $(U, \psi)$ and observe that $(\phi_j(U), \psi_j)$ with $\psi_j \defeq \psi \circ \phi_j^{-1}$ is a local chart for $\Sigma$. Note that while $(\phi_j(U), \psi_j)$ might not necessarily belong to the orientation, we still have
	\begin{align}\label{eq:normal cross product}
		\nu_{f_j} = (-1)^{s_j} \frac{\partial_{z^1} f_j\times \partial_{z^2}f_j}{\abs{\partial_{z^1} f_j\times \partial_{z^2}f_j}}  \quad \text{on } \phi_j(U),
	\end{align}
	for some $s_j\in\{0,1\}$, where $(z^1, z^2)$ denote the local coordinates on $\Sigma$ induced by the chart $(\phi_j(U), \psi_j)$.
	The coordinate derivatives transform by
	$ \partial_{z^i} f_j \circ \phi_j= \partial_{y^i} (f_j \circ \phi_j)$ for $i=1,2$
	such that by \eqref{eq:normal cross product} we have 
	\begin{align}\label{eq:normal cross product 2}
		\nu_{\tilde{f}_j}\defeq \nu_{f_j}\circ \phi_j = (-1)^{s_j} \frac{\partial_{y^1} \tilde{f}_j\times \partial_{y^2} \tilde{f}_j}{\abs{\partial_{y^1} \tilde{f}_j\times \partial_{y^2} \tilde{f}_j}}\quad \text{on } U.
	\end{align}
}
	\delete{Consequently, by \eqref{eq:dm tilde f_j bounds} we obtain the bounds $\norm{\partial^{m}\nu_{\tilde{f}_j}}{L^{\infty}(U)}\leq C(m,K, \chi(\Sigma), \xi)$.}

Moreover, \add{using the scale-invariance, cf.\ \Cref{rem:ParabolicScaling}, and the invariance under repa\-ra\-metrization}, we have the evolution
\begin{align}
	\partial_t \tilde{f}_j 
	& = -\nabla\overline{\CalW}(\tilde{f}_j) + \lambda(f_j)\nu_{ \tilde{f}_j}.\label{eq:dtf_jphi_j}
\end{align}
\add{Using the established bounds and the evolution \eqref{eq:dtf_jphi_j}, it is not difficult to see that $\tilde{f_j}$ converges in $C^1([\xi, \hat{c}]; C^{m}(P;\R^3))$ for all $P\subset \hat{\Sigma}$ compact and for all $m\in \N$ to a limit flow $f_{lim}\colon [\xi, \hat{c}]\times \hat{\Sigma}\to \R^3$ and $\lambda(f_j)\to \lambda_{lim}$ in $C^0([\xi,\hat c];\R)$ as $j\to\infty$, after passing to a subsequence.}
\delete{
Therefore, also using \eqref{eq:normal cross product 2} we obtain the bounds 
\begin{align}
\norm{\partial^{m}\partial_t \tilde{f}_j}{L^{\infty}(U)}&\leq C(m,K, \chi(\Sigma), \xi)\\
\norm{\partial^m \partial_t \nu_{\tilde{f}_j}}{L^{\infty}(U)}&\leq C(m,K, \chi(\Sigma), \xi)\quad\text{ for all }\xi\leq t\leq \hat{c},m\geq 0, j\geq J.\label{eq:dm dt tilde f_j bounds}
\end{align}

Combining \eqref{eq:dt lambda gesamt} and \eqref{eq:dm dt tilde f_j bounds} and using the evolution \eqref{eq:dtf_jphi_j} again, it is not difficult to see that one also has the stronger estimates
\begin{align}
	\norm{\partial^{m}\partial_t^2 \tilde{f}_j}{L^{\infty}(U)}\leq C(m,K, \chi(\Sigma), \xi)) \quad\text{ for all }\xi\leq t\leq \hat{c},m\geq 0, j\geq J.
\end{align}

A covering argument extends these bounds to any compact $P\subset \hat{\Sigma}$. By the Arzelà--Ascoli Theorem and a diagonal argument, after passing to a subsequence and an appropriate translation we may thus assume that $\tilde{f_j}$ converges in $C^1([\xi, \hat{c}]; C^{m}(P;\R^3))$ for all $P\subset \hat{\Sigma}$ compact and for all $m\in \N$ to a limit flow $f_{lim}\colon [\xi, \hat{c}]\times \hat{\Sigma}\to \R^3$. Consequently, by \eqref{eq:normal cross product 2} we also have $\nu_{\tilde{f}_j} = \nu_{f_j}\circ \phi_j \to \nu_{lim}$ in $C^1([\xi, \hat{c}]; C^{m}(P;\R^3))$ for all $P\subset \hat{\Sigma}$ compact and for all $m\in \N$, where $\nu_{lim}(t, \cdot)$ is normal along $f_{lim}(t, \cdot)$ for all $\xi\leq t\leq \hat{c}$.
Moreover, by \eqref{eq:lambda bound} and \eqref{eq:dt lambda gesamt} we can assume $\lambda(f_j(t, \cdot)) \to \lambda_{lim}$ in $C^0([\xi, \hat{c}];\R)$ as $j\to\infty$.
}

Fix $P\subset \hat{\Sigma}$ compact and let $j\in \N$ \add{be} large enough. Then, using \eqref{eq:dtf_jphi_j}, \eqref{eq:dtV=0} and \eqref{eq:dtW<=0}
\begin{align}
	\int_\xi^{\hat{c}}\int_P \abs{\partial_t \tilde{f}_j}^2\diff \mu_{\tilde{f}_j} \diff t &\leq \int_{\xi}^{\hat{c}} \int_{\Sigma} \langle -\nabla\overline{\CalW}({f}_j) + \lambda(f_j)\nu_{f_j}, \partial_t f_j \rangle\diff \mu_{f_j}\diff t = \int_\xi^{\hat{c}} (-\partial_t \overline{\CalW}(f_j))\diff t.
\end{align}
In particular, using the convergence $\tilde{f}_j \to f_{lim}$ in $C^1([\xi, \hat{c}]; C^{m}(P;\R^3))$, we find
\begin{align}\label{eq:bar W to zero}
	\int_{\xi}^{\hat{c}} \int_{P} \abs{\partial_t f_{lim}}^2\diff \mu_{f_{\lim}}\diff t \leq \lim_{j\to\infty}\left( \overline{\CalW}(f(t_j+r_j^4\xi,\cdot))- \overline{\CalW}(f(t_j+r_j^4\hat{c},\cdot))\right) = 0,
\end{align}
by \add{scale-invariance and} monotonicity of the energy \delete{the limit $\lim_{t\nearrow T}\overline{\CalW}(f(t, \cdot))\leq \overline{\CalW}(f_0)<\infty$ exists.} Consequently, $f_{lim}$ is constant in time, hence $f_{lim}\equiv f_{lim}(\hat{c}, \cdot) = \lim_{j\to\infty} f_j(\hat{c}, \phi_j(\cdot)) = \lim_{j\to\infty}\hat{f}_j\circ\phi_j = \hat{f}$. 
We observe that $\hat{\nu} \defeq \nu_{lim}(\hat{c}, \cdot)$ is a global and smooth normal vector field on $\hat{\Sigma}$ and hence $\hat{\Sigma}$ is orientable. Setting $\hat{\lambda}\defeq \lambda_{lim}(\hat{c})$ and using \eqref{eq:dtf_jphi_j} we find
\begin{align}
	-\nabla\overline{\CalW}(\hat{f}) + \hat{\lambda} \hat{\nu} = \lim_{j\to\infty} \partial_t \tilde{f}_j(\hat{c}, \cdot) = \partial_t f_{lim}(\hat{c}, \cdot) = 0\quad \text{on }\hat{\Sigma},
\end{align}
so $\hat{f}$ solves \eqref{eq:stationary} and hence is a constrained Willmore immersion.
In addition, \delete{the energy dissipation, the invariance and }the lower semicontinuity of the Willmore functional $\CalW$ with respect to smooth convergence on compact sets (which is discussed in \cite[Appendix B]{DMSS20} for instance) yield
\begin{align}\label{eq:Willmore lower semicont}
	{\CalW}(\hat{f})\leq \liminf_{j\to\infty} {\CalW}(\hat{f}_j) \leq {\CalW}(f_0)\leq K.
\end{align}

For the ``moreover'' part of the theorem, we note that since $f_j$ is a volume-preserving Willmore flow we find by \eqref{eq:VpWF} and \eqref{eq:deflambda}
\begin{align}
&\int_\xi^{\hat{c}} \int_{P} \abs{\nabla\overline{\CalW}(\tilde{f}_j)}^2 \diff \mu_{\tilde{f}_j}  \diff t\leq \int_\xi^{\hat{c}}\int_{\Sigma} \abs{\nabla\overline{\CalW}(f_j)}^2\diff \mu_{f_j} \diff t  \\ &\qquad  = \int_\xi^{\hat{c}} \int_{\Sigma} \left\langle -\partial_t f_j + \lambda(f_j)\nu_{f_j}, \nabla\overline{\CalW}(f_j)\right\rangle \diff \mu_{f_j} \diff t   \\
&\qquad = - \int_\xi^{\hat{c}} \partial_t \overline{\CalW}(f_j) \diff t + \int_\xi^{\hat{c}} \abs{\lambda (f_j)}^2 \CalA(f_j)\diff t.\label{eq:blowupStatic1}
\end{align}
{As in \eqref{eq:bar W to zero}, the first term goes to zero as $j\to\infty.$} 
For the second term, note that by \eqref{eq:deflambda}, $\lambda$ scales by $\lambda(r^{-1}f) = r^3 \lambda(f)$ for  $r>0$. Therefore, we find
\begin{align}
&\int_\xi^{\hat{c}} \abs{\lambda(r_j^{-1}f(t_j+r_j^4t, \cdot))}^2 \CalA (r_j^{-1}f(t_j+r_j^4t, \cdot)) \diff t  \\
&\qquad = \int_{t_j+r_j^4 \xi}^{t_j+r_j^{4}\hat{c}} \abs{r_j^{3}\lambda (f(\tau, \cdot))}^2 r_j^{-2} \CalA(f(\tau, \cdot))r_j^{-4}\diff \tau = \int_{t_j+r_j^4 \xi}^{t_j+r^{4}_j\hat{c}} \abs{\lambda(f_\tau)}^2\CalA(f_\tau)\diff \tau,\label{eq:blowupStatic2}
\end{align}
after a change of variables. Recall that by assumption $\int_0^{T}\lambda^2\CalA\diff t<\infty$, so \delete{taking $t_j \nearrow T$,} the second term in \eqref{eq:blowupStatic1} also goes to zero as $j\to\infty$ using dominated convergence. Consequently, by \eqref{eq:blowupStatic1}, we have
\begin{align}
&\int_{\xi}^{\hat{c}} \int_{P} \abs{\nabla\overline{\CalW}(f_{lim})}^2\diff\mu_{f_{lim}}\diff t = \lim_{j\to\infty}\int_\xi^{\hat{c}} \int_{P}\abs{\nabla\overline{\CalW}(\tilde{f}_j)}^2 \diff \mu_{\tilde{f}_j} \diff t = 0.\label{eq:blowupStatic3}
\end{align}
Since $f_{lim}(t, \cdot)\equiv \hat{f}$ and as $P$ was arbitrary, we conclude $\nabla\overline{\CalW}(\hat{f})=0$, so $\hat{f}$ is a Willmore immersion.
\end{proof}

\delete{
\begin{lem}[\add{\cite[Lemma 4.3]{KSSI}}]\label{lem:KSSI Lemma 4.3}
	Let $\hat{f}\colon\hat{\Sigma}\to\R^3$ be {as in \Cref{thm:BlowUpExistence}} with $\hat{\Sigma}\neq \emptyset$. If one component $C$ of $\hat{\Sigma}$ is compact, then $C=\hat{\Sigma}$ and $\Sigma$ is diffeomorphic to $C= \hat{\Sigma}$.
\end{lem}
\begin{proof}
	Let $j_0$ be large enough such that $C\subset \hat{\Sigma}(j_0)$ and let $j\geq j_0$. Since $C\neq\emptyset$ is a connected component, it is open in $\hat{\Sigma}$ and thus so is its image $\phi_j(C)\subset \Sigma$. However, $\phi_j(C)$ is also compact, thus closed. So $\phi_j(C)=\Sigma$ as $\Sigma$ is connected, and consequently $\hat{\Sigma} = \bigcup_{j\geq j_0}\hat{\Sigma}(j) = C$, as $ \hat{\Sigma}(j)=\phi_j^{-1}(\Sigma)=C$ for $j\geq j_0$.
\end{proof}
}

\delete{
\begin{lem}\label{lem:existence of r_t} Fix an initial energy ${\CalW}(f_0)$. There exists $\varepsilon_4 = \varepsilon_4({\CalW}(f_0), \chi(\Sigma))\in (0, \varepsilon_3)$ and a radius $r_t = r_t({\CalW}(f_0), \chi(\Sigma))>0$ for every $t\in [0,T)$ such that we have
	\begin{align}
	\varepsilon_4\leq\kappa(t, r_t) \leq \varepsilon_3 \quad \text{for all }t\in [0,T).
	\end{align}
\end{lem}

\begin{proof}
	Let $0<\varepsilon_4< \frac{8\pi}{\Gamma}$ to be chosen and define $r_t \defeq \sup\{r>0\mid \kappa(t, r)<\Gamma \varepsilon_4\}$, {where $\Gamma>1$ is as in \eqref{eq:GammaEstimate1/2}.} First we observe $r_t<\infty$ since $\int\abs{A_t}^2 \diff \mu_t \geq 8\pi$ {by \eqref{eq:A^2GaussBonnet} using $\CalW(f_t)\geq 4\pi$ and $\chi(\Sigma)\leq 2$.}
	Considering $\mu_t$ as a measure on $\R^3$, we find that $A_t\in L^2(\diff\mu_t)$ and hence, by absolute continuity of the integral, there exists $\delta>0$ such that if an open set $B\subset \R^3$ satisfies $\mu_t(B)<\delta$ we may conclude $\int_{B}\abs{A_t}^2\diff\mu_t <\Gamma\varepsilon_4$. Taking $0<r$ small enough, we have using Simon's monotonicity formula \eqref{eq:SimonMono 3}
	\begin{align}
	\mu_t(B_{r}(x)) \leq C({\CalW}(f_0))r^2<\delta,
	\end{align}
	and hence $r_t\geq r>0$. By \eqref{eq:kappa continuity}, we easily find $\kappa(t,r_t)\leq \Gamma\varepsilon_4$. For the estimate from below assume $\kappa(t, r_t)<\varepsilon_4$, then by \eqref{eq:kappa continuity} we have $\limsup_{r\searrow r_t} \kappa(t,r) <\Gamma\varepsilon_4$, a contradiction to $r_t$ being the supremum. Therefore, taking $\varepsilon_4\defeq \min\{\frac{\varepsilon_3}{\Gamma}, \frac{8\pi}{\Gamma}\}>0$ yields the claim.
\end{proof}
}

\add{We can choose $(t_j)_{j\in\N}, (r_j)_{j\in \N}, (x_j)_{j\in \N}$ such that the concentration limit is nontrivial, even if $T=\infty$. The argument is exactly as in \cite[p.~348--349]{KSRemovability}, so the proof can safely be omitted.}
\begin{prop}\label{prop:nontrivial blow up}
	Let $f\colon[0,T)\times \Sigma\to\R^3$ be a volume-preserving Willmore flow \add{with $0<T\leq \infty$}. Then, 
	\add{we can choose}\delete{there exist} sequences $t_j\nearrow T$, $(r_j)_{j\in \N}\subset(0,\infty)$ and $(x_j)_{j\in \N}\subset \R^3$ \add{satisfying \eqref{eq:blowup assumption}} such that the concentration limit $\hat{f}\colon\hat{\Sigma}\to\R^3$ from \Cref{thm:BlowUpExistence} satisfies
	\begin{align}\label{eq:limit concentration in 0}
		\int_{\overline{B_1(0)}}\abs{A_{\hat{f}}}^2\diff \mu_{\hat{f}}>0,
	\end{align}
	in particular $\hat{\Sigma}\neq \emptyset$.
\end{prop}

\delete{
\begin{proof}
	We follow \cite[p. 349]{KSRemovability} and first claim that
	\begin{align}
		\label{eq:annals 5.24}
		\liminf_{t\to\infty}\frac{r_{t+\hat{c}r_t^4}}{r_t}<\infty.
	\end{align}
	Let $M>0$ and suppose that $r_{t+\hat{c}r_t^4}\geq M r_t$ for all $t\geq t_0$. Define $\hat{t}_0=t_0$ and $\hat{t}_{j+1}\defeq \hat{t}_j + \hat{c}r_{\hat{t}_j}$ for all $j\in \N$. Then we have $r_{\hat{t}_{j}}\geq M^{j}r_{\hat{t}_0}$. Now, fix $t\in [0,T)$. If $r_t>\frac{\diam f_t(\Sigma)}{2}$ was satisfied, we could find a ball $B_r(x)$ with $r<r_t$ such that $f_t(\Sigma)\subset B_r(x)$. But then 
	\begin{align}
		\int_{\Sigma}\abs{A_t}^2\diff\mu_t = \int_{B_r(x)}\abs{A_t}^2\diff \mu_t \leq \kappa(t, r_t) \leq \varepsilon_3\leq \bar{\varepsilon}\leq\varepsilon_2<8\pi,
	\end{align}
	{using that $\varepsilon_3=\bar{\varepsilon}\hat{c}$ with $\hat{c}\in (0,1)$, the definition $\bar{\varepsilon}$ in the proof of \Cref{thm:Lifespan} and the fact that $\varepsilon_2\in (0,8\pi)$ by \Cref{prop:lambda4/3Nonconcentration}.}
	However, $\int_{\Sigma}\abs{A_t}^2\diff\mu_t \geq 8\pi$ by \eqref{eq:A^2GaussBonnet}, a contradiction. Therefore, we have $r_t\leq \frac{\diam(f_t(\Sigma))}{2}$, and thus by \add{Simon's diameter bound \cite{SimonWillmore}}, we conclude
	\begin{align}
		r_t \leq \frac{1}{\pi}\CalW(f_0)^{\frac{1}{2}}\CalA(f_t)^{\frac{1}{2}} \quad\text{for all }t\in [0,T).
	\end{align}
	On the other hand for all $0\leq \tau <T$ using \eqref{eq:dtW<=0} and \eqref{eq:WvstildeW} we find
	\begin{align}
		\CalA(f_\tau)-\CalA(f_0) &= -\int_0^{\tau}\int_{\Sigma} \langle H_{f_t}, \partial_t f\rangle\diff \mu_t \diff t \\
		&\leq \tau^{\frac{1}{2}}\left(\int_{\Sigma} \abs{H_{f_0}}^2\diff \mu_{f_0}\right)^{\frac{1}{2}} \left(\int_0^{\tau} \int_{\Sigma}\abs{\partial_t f}^2 \diff \mu_t \diff t\right)^{\frac{1}{2}} \\
		&\leq C({\CalW}(f_0)) \tau^{\frac{1}{2}},
	\end{align}
	with a similar argument as in the proof of \Cref{prop:Lambda^2AVpWF}.	Consequently, we have 
	\begin{align}
		r_t^4 &\leq C({\CalW}(f_0),\CalA(f_0)) (1+t)\\
		1+\hat{t}_{j+1}&\leq 1+\hat{t}_j+\hat{c}r_{\hat{t}_j}^4 \leq \big(1+C({\CalW}(f_0), \CalA(f_0))\big)(1+\hat{t}_j) \\
		0<r_{\hat{t}_0}^4M^{4j}&\leq r_{\hat{t}_j}^4\leq \big(1+C({\CalW}(f_0), \CalA(f_0))\big)^{j}(1+\hat{t}_0).
	\end{align}
	Taking the limit $j\to\infty$ yields $M^4 \leq 1+C({\CalW}(f_0), \CalA(f_0))$ and \eqref{eq:annals 5.24} is proven. Consequently, there exist a sequence $t_j\nearrow T$ such that $r_{t_j+\hat{c}r_j^4}\leq M r_j$ for some $M>0 $. By a covering argument there exist $N=N(M)\in\N$ such that
	\begin{align}
		\kappa(t,Mr) \leq N\kappa(t,r) \quad\text{for all }t\in [0,T), r>0,
	\end{align}
	cf. \eqref{eq:GammaEstimate1/2}. Thus, using \Cref{lem:existence of r_t} we find
	\begin{align}
		\kappa(t_j+\hat{c}r_{t_j}^4,r_{t_j})\geq \frac{1}{N}\kappa(t_j+\hat{c}r_{t_j}^4, Mr_{t_j}) \geq \frac{1}{N} \kappa(t_j+\hat{c}r_{t_j}^4,r_{t_j+\hat{c}r_{t_j}^4}) \geq \frac{ \varepsilon_4}{N}.
	\end{align}
	Now, we pick $x_j \in \R^3$ such that 
	\begin{align}
		\int_{B_1(0)}\abs{A_{\hat{f}_j}}^2\diff\mu_{\hat{f}_j}=\int_{B_{r_{t_j}}(x_j)}\abs{A_{f_{t_j+\hat{c}r_{t_j}^4}}}^2\diff \mu_{f_{t_j+\hat{c}r_{t_j}^4}} \geq \frac{\varepsilon_4}{2N}>0.
	\end{align}
	Taking $j\to\infty$ and using the smooth convergence on compact sets yields \eqref{eq:limit concentration in 0} for $r_j \defeq r_{t_j}$.
\end{proof}
}

\section{Convergence for compact concentration limits}\label{sec:blow up not compact}
\add{The main result of this section is the following
\begin{thm}\label{thm: no compact blowups2}
	Let $f\colon[0,T)\times \Sigma\to\R^3$ be a volume-preserving Willmore flow and let $\hat{f}\colon\hat{\Sigma}\to\R^3$ be a concentration limit with $\hat{\Sigma}\neq \emptyset$.
	If $\hat{\Sigma}$ has a compact component and
	\begin{enumerate}[(i)]
		\item $\CalV(\hat{f})\neq 0$ or
		\item $\CalV(\hat{f}) = \CalV(f_0)=0$,
	\end{enumerate}
	then $\hat{f}$ is a limit under translation. Moreover, the flow exists globally and converges, as $t\to\infty$, after reparametrization by diffeomorphisms, to a constrained Willmore immersion $f_\infty$ with $\CalW(f_\infty)=\CalW(\hat{f})$.
\end{thm}
}
\delete{
\begin{thm}\label{thm: no compact blowups}
	Let $f\colon[0,T)\times \Sigma\to\R^3$ be a volume-preserving Willmore flow and let $\hat{f}\colon\hat{\Sigma}\to\R^3$ be a concentration limit with $\hat{\Sigma}\neq \emptyset$.
	If $\hat{\Sigma}$ has a compact component and
	\begin{enumerate}[(i)]
		\item $\CalV(\hat{f})\neq 0$ or
		\item $\CalV(\hat{f}) = \CalV(f_0)=0$,
	\end{enumerate}
	then $\hat{f}$ is a limit under translation.
\end{thm}
}
\add{
Under certain assumptions, the first part of the statement can also be directly obtained from the scaling behavior of the volume.
\begin{remark}
	Under the assumption that $\hat{\Sigma}$ is compact, we have
		$\CalV(\hat{f})=\lim_{j\to\infty}r_j^{-3} V_0$
	which immediately yields that
	\begin{enumerate}[(i)]
		\item if $\CalV(\hat{f})\neq 0$, then $\hat{f}$ cannot be a blow-up or a blow-down;
		\item if $V_0\neq 0$, then $\hat{f}$ cannot be a blow-up.
	\end{enumerate}
	Clearly, these arguments fail if $\CalV(\hat{f})=V_0=0$.
\end{remark}
The key ingredient to prove the powerful convergence result \Cref{thm: no compact blowups2} relies on a suitable extension of the \emph{{\L}ojasiewicz--Simon gradient inequality.}}
\subsection{The constrained \texorpdfstring{{\L}ojasiewicz}{Lojasiewicz}--Simon gradient inequality}

In this subsection, we will state and prove a \emph{constrained} or \emph{refined} {\L}ojasiewicz--Simon gradient inequality, cf.\ \cite{Rupp}, for the volume-preserving Willmore flow. A similar result for the length-preserving elastic flow of curves was recently proven in \cite{RuppSpener}.

The strategy is the same as in \cite[Section 3]{CFS09}. First, in order to get rid of the invariance of the Willmore and volume energy, we restrict ourselves to normal variations. 
Throughout this section we will fix some smooth immersion ${f}\colon\Sigma\to\R^3$. The \emph{normal Sobolev spaces} along $f$ are defined by
\begin{align}
	W^{k,2}(\Sigma;\R^3)^{\perp}\defeq \{\phi\in W^{k,2}(\Sigma;\R^3)\mid P^{\perp}\phi=\phi\},
\end{align}
for $k\in \N_0$, with $L^2(\Sigma;\R^3)^{\perp}\defeq W^{0,2}(\Sigma;\R^3)^{\perp}$. Note that $L^{2}(\Sigma;\R^3)^{\perp}$ is a Hilbert space with inner product
\begin{align}\label{eq: L^2 mu_f}
	\langle \phi_1, \phi_2\rangle_{L^{2}(\Sigma;\R^3)^{\perp}} = \int_{\Sigma} \langle \phi_1, \phi_2\rangle\diff \mu_f\quad \text{ for }\phi_1, \phi_2 \in L^{2}(\Sigma;\R^3)^{\perp}.
\end{align}
\begin{remark}\label{rem:vector vs scalar Sobolev}
	\begin{enumerate}[(i)]
		\item 
		Note that since we are in codimension one, we have
		\begin{align}
		W^{k,2}(\Sigma;\R^3)^{\perp} = \{u\nu_f\mid u\in W^{k,2}(\Sigma)\},
		\end{align}
		for $k\in \N_0$, where $\nu_{f}$ is the unit normal to $f$ and $W^{k,2}(\Sigma)\defeq W^{k,2}(\Sigma;\R)$. In fact, the map $W^{k,2}(\Sigma)\to W^{k,2}(\Sigma;\R^3)^{\perp },u\mapsto\phi = u\nu_f$ is an isomorphism of Banach spaces and for $k=0$ an isometry between the Hilbert spaces $L^2(\Sigma)$ and $L^2(\Sigma;\R^3)^{\perp}$.
		\item Since $\Sigma$ is compact, the spaces $W^{k,2}(\Sigma;\R^3)$ and $L^2(\Sigma;\R^3)$ do not depend on the metric, cf. \cite[Theorem 2.20]{Aubin}.
	\end{enumerate}
\end{remark}

First, we prove a constrained {\L}ojasiewicz--Simon gradient inequality in normal directions in a neighborhood of a constrained Willmore immersion, i.e.\ a solution to \eqref{eq:stationary}.
\begin{prop}\label{prop:Loja Normal}
	Let $f\colon\Sigma\to\R^3$ be a smooth constrained Willmore immersion. Then, there exists $C, \sigma>0$ and $\theta\in (0, \frac{1}{2}]$ such that for all  $\phi\in W^{4,2}(\Sigma;\R^3)^{\perp}$ with $\norm{\phi}{W^{4,2}}\leq \sigma$ and $\CalV(f+\phi)=\CalV(f)$ we have
	\begin{align}
		\abs{\overline{\CalW}(f+\phi)-\overline{\CalW}(f)}^{1-\theta}\leq C\norm{\nabla\overline{\CalW}(f+\phi)- \lambda(f+\phi)\nu_{f+\phi}}{L^2(\diff\mu_{f+\phi})},
	\end{align}
	where $\lambda$ is as in \eqref{eq:deflambda}.
\end{prop}
{
\Cref{prop:Loja Normal} will follow from \cite{CFS09} and \cite[Corollary 5.2]{Rupp}. To that end, we need to show the analyticity of certain maps and study their second variations. Most of the results will follow from \cite{CFS09} in the case of codimension one, only the volume needs to be studied in detail.
}
{
\begin{lem}\label{lem:unit normal props}
	Let $U\defeq B_\rho(0)\subset W^{4,2}(\Sigma)$.
Then for $\rho>0$ small enough and writing $f_u\defeq f+u\nu_f$ for $u\in U$ we have
	\begin{enumerate}[(i)]
		\item $f_u \colon\Sigma\to\R^3$ is an immersion and $U\to W^{4,2}(\Sigma;\R^3), u\mapsto f_u$ is analytic; 
		\item the map $U\to C^{0}(\Sigma;\R^3), u\mapsto \nu_{f_u}$ is analytic.
	\end{enumerate}
\end{lem}
\begin{proof}
	\begin{enumerate}[(i)]
		\item Taking $\rho>0$ small enough and using the Sobolev embedding $W^{4,2}(\Sigma)\hookrightarrow C^{1}(\Sigma)$ we find that $f_u$ is an immersion for all $u\in U$. 
		The map $u\mapsto f_u$ is linear and bounded, hence analytic.
		\item In local coordinates $(y^1,y^2)$ in the orientation on $\Sigma$, by (i) and the Sobolev embedding theorem $W^{4,2}(\Sigma)\hookrightarrow  C^{1}(\Sigma)$ , the map $B_{\rho}(0)\to C^{1}(\Sigma;\R^3), u\mapsto \partial_{y^1} f_u \times \partial_{y^2} f_u$ is bilinear and bounded, hence analytic. Moreover, since $f_u$ is a $C^1$-immersion by (i), the denominator in the definition of $\nu_{f_u}$ in \eqref{eq:def normal} is uniformly bounded away from zero. Since $\R^3\setminus B_{\delta}(0)\to\R^3, x\mapsto\frac{x}{\abs{x}}$ is analytic for any $\delta>0$ the claim follows from the characterization of analytic Nemytskii operators on $C(\Sigma)$, cf. \cite [Theorem 6.8]{Appell90}. \qedhere
	\end{enumerate}
\end{proof}
}
Let
$\tilde{U}\defeq \{\phi\in W^{4,2}(\Sigma;\R^3)^{\perp}\mid \phi = u\nu_f \text{ with } u\in U\}$. By \Cref{rem:vector vs scalar Sobolev} (i), $\tilde{U}$ is open in $W^{4,2}(\Sigma;\R^3)^{\perp}$. We consider the shifted energies, defined by
\begin{align}
	&W\colon \tilde{U} \to \R, W(\phi)\defeq \overline{\CalW}(f+\phi),\\
	&V\colon \tilde{U}\to \R, V(\phi)\defeq \CalV(f+\phi).
\end{align}
\begin{lem}\label{lem:analytic}
	Under the assumptions of \Cref{lem:unit normal props}, the following maps are analytic:
	\begin{enumerate}[(i)]
		\item the function $\tilde{U}\to C^{0}(\Sigma), \phi\mapsto \rho_{f+\phi}$, where $\diff \mu_{f+\phi}=\rho_{f+\phi} \diff \mu_f$;
		\item the function $\tilde{U}\to \R, \phi\mapsto W(\phi)$;
		\item the function $\tilde{U}\to L^2(\Sigma;\R^3)^{\perp}, \phi\mapsto \add{P^{\perp}} \nabla\overline{\CalW}(f+\phi)\rho_{f+\phi}$;
		\item the function $\tilde{U}\to\R, \phi \mapsto V(\phi)$;
		\item the function $\tilde{U}\to L^2(\Sigma;\R^3)^{\perp},  \phi\mapsto \add{P^{\perp}}\left(-\nu_{f+\phi}\rho_{f+\phi}\right)$.
	\end{enumerate}
\end{lem}
\begin{proof}
	Statement (i) is  \cite[Lemma 3.2 (vii)]{CFS09} and
	(ii) follows from \cite[Lemma 3.2 (iv) and (vii)]{CFS09}.
	
	By \cite[Lemma 3.2 (v) and (vi)]{CFS09}, $\tilde{U}\to L^2(\Sigma;\R^3)^{\perp}, \phi\mapsto\nabla\overline{\CalW}(f+\phi)$ is analytic, and hence (iii) follows from (i).
	
	 Note that by \Cref{rem:vector vs scalar Sobolev}, $\tilde{U}\to U, \phi\mapsto u = \langle\phi, \nu_f\rangle$ is linear and bounded, thus analytic. Therefore, $\tilde{U}\to  C^{0}(\Sigma;\R^3), \phi\mapsto\nu_{f+\phi}$ is analytic, hence so is $V$ and by (i) statement (v) follows.
 \end{proof}
As a last missing ingredient towards proving the constrained {\L}ojasiewicz--Simon gradient inequality, we compute the first and second variations.

\begin{lem}\label{lem:variations}
	Let $H\defeq L^2(\Sigma;\R^3)^{\perp}$ . Under the assumption of \Cref{lem:unit normal props}, for each $\phi\in \tilde{U}$, the $H$-gradients of $W$ and $V$ are given by
	\begin{align}
		\nabla_{H}W(\phi) &= \add{P^{\perp}} \nabla\overline{\CalW}(f+\phi)\rho_{f+\phi},\\
		\nabla_{H}V(\phi) &= \add{P^{\perp}}\left( -\nu_{f+\phi}\rho_{f+\phi}\right).\label{eq:WV H grads}
	\end{align}
	Moreover, the Fréchet-derivatives of the $H$-gradient maps of $W$ and $V$ at $u=0$ satisfy
	\begin{align}
		\left(\nabla_{H}W\right)^{\prime}(0) &\colon W^{4,2}(\Sigma,\R^3)^{\perp}\to L^2(\Sigma,\R^3)^{\perp} \quad\text{ is a Fredholm operator with index zero,} \\
		\left(\nabla_H V\right)^{\prime}(0)&\colon W^{4,2}(\Sigma,\R^3)^{\perp}\to L^2(\Sigma,\R^3)^{\perp} \quad\text{ is compact.}
	\end{align}
\end{lem}

\begin{proof}
	For $\phi,\psi\in \tilde{U}$, we have by the first variation of the Willmore energy and \eqref{eq: L^2 mu_f}
	\begin{align}
		\dtzero W(\phi+t\psi) = \int_{\Sigma}\langle \nabla\overline{\CalW}(f+\phi), \psi\rangle\diff \mu_{f+\phi} = \left\langle \add{P^{\perp}} \nabla\overline{\CalW}(f+\phi)\rho_{f+\phi}, \psi\right\rangle_{H},
	\end{align}
	\add{where we also used $\psi = P^{\perp}\psi$.}

	Similarly, $\dtzero V(\phi+t\psi) = -\int_{\Sigma} \langle\nu_{f+\phi}, \psi\rangle\diff \mu_{f+\phi} = -\left\langle P^{\perp}\nu_{f+\phi} \rho_{f+\phi}, \psi\right\rangle_H$. The Fredholm property of $(\nabla_HW)^{\prime}(0)$ follows from \eqref{eq:WvstildeW} and \cite[Lemma 3.3 and p. 356]{CFS09}. For the last statement, we use \eqref{eq:dtnu} and \Cref{rem:vector vs scalar Sobolev} (i) to obtain for $\phi = u\nu_f$
	\begin{align}
		\dtzero \nu_{f+t\phi} = - \grad_g u = - \grad_g \langle \phi, \nu_f\rangle.
	\end{align}
Now, by \eqref{eq:dtdmu}, we find $\dtzero\rho_{f+t\phi} \diff \mu_f = \dtzero(\diff \mu_{f+{t\phi}}) = -\langle H_{f}\nu_f, \phi\rangle\diff \mu_f$.  Using \eqref{eq:WV H grads} we obtain, \add{since the gradient term is tangential,}
	\begin{align}
		(\nabla_H V)^{\prime}(0)\add{\phi} =- \add{P^{\perp}} \dtzero\left(\nu_{f+t\phi}\rho_{f+t\phi}\right) &=\add{P^{\perp}}\grad_g \langle \phi, \nu_f\rangle + \add{P^{\perp}}\nu_f \langle H_f\nu_f, \phi\rangle \\
		&= \nu_f \langle H_f\nu_f, \phi\rangle.
	\end{align}
	\add{As this is only of zeroth} order in $\phi\in W^{4,2}(\Sigma;\R^3)^{\perp}$, the claim follows from the Rellich--Kondrachov Theorem, see for instance \cite[Theorem 2.34]{Aubin}.
\end{proof}
\begin{proof}[{Proof of \Cref{prop:Loja Normal}}]
	We verify the assumptions of \cite[Corollary 5.2]{Rupp} for the Hilbert space $W^{4,2}(\Sigma;\R^3)^{\perp}$ which embeds densely into $H=L^2(\Sigma;\R^3)^{\perp}$. The functionals $W$ and $V$ are analytic with analytic $H$-gradients in a neighborhood $\tilde{U}$ of zero by \Cref{lem:analytic}. By \Cref{lem:variations}, the second variation of $W$ at zero is Fredholm of index zero, whereas the second variation of $V$ at zero is compact. Note that $\nabla_H V(0)\neq 0$ since we have
	\begin{align}
		\langle\nabla_H V(0), \nu_f\rangle_H = -\int_{\Sigma} \langle \nu_f, \nu_f\rangle\diff \mu_f= -\CalA(f)<0.
	\end{align}
	Thus, by \cite[Corollary 5.2]{Rupp}, $W$ satisfies a constrained {\L}ojasiewicz--Simon gradient inequality near $\phi=0$, i.e.\ there exist $C, \sigma>0$ and $\theta\in (0,\frac{1}{2}]$ such that for all $\phi\in \tilde{U}$ with $\norm{\phi}{W^{4,2}}\leq \sigma$ and $V(\phi) = V(0)$, we have
	\begin{align}
		\abs{W(\phi)-W(0)}^{1-\theta}\leq C \norm{P_{\phi}\nabla_H W(\phi)}{H},
	\end{align}
	where $P_{\phi}\colon H\to H$ is the $H$-orthogonal projection onto $\{y\in H\mid \langle y, \nabla V(\add{\phi})\rangle_H = 0\}$, cf. \cite[Proposition 3.3]{Rupp}. Thus, for $\lambda(f+\phi)$ as in \eqref{eq:deflambda}, we find
	\begin{align}
		\norm{P_{\phi}\nabla W(\phi)}{H}^2 &= \norm{P_{\phi}(\nabla W(\phi) + \lambda(f+\phi)\nabla V(\phi))}{H}^2  \leq \norm{\nabla W(\phi)+ \lambda(f+\phi)\nabla V(\phi)}{H}^2  \\
		& = \int_{\Sigma}\abs{\nabla\overline{\CalW}(f+\phi)-\lambda(f+\phi)\nu_{f+\phi}}^2 \rho_{f+\phi} \diff\mu_{f+\phi}.\label{eq:Loja0}
	\end{align}
	Now, by the Sobolev embedding theorem $W^{4,2}(\Sigma;\R^3)\hookrightarrow C^1(\Sigma;\R^3)$, we may bound $\norm{\rho_{f+\phi}}{\infty}$ for all $\norm{\phi}{W^{4,2}}\leq \sigma$. Using \eqref{eq:Loja0} and the definition of $W$ and $V$ yields the claim.
\end{proof}
This finally yields the inequality for all directions.
\begin{thm}\label{thm:Loja}
	Let $f\colon\Sigma\to\R^3$ be a constrained Willmore immersion. Then, there exist $C, \sigma>0$ and $\theta\in (0, \frac{1}{2}]$ such that for all ${h}\in W^{4,2}(\Sigma;\R^3)$ with $\norm{h-f}{W^{4,2}}\leq \sigma$ and $\CalV(h)=\CalV(f)$ we have
	\begin{align}
		\abs{\overline{\CalW}(h)-\overline{\CalW}({f})}^{1-\theta}\leq C\norm{\nabla\overline{\CalW}(h)-\lambda(h)\nu_h}{L^2(\diff \mu_{{h}})}.
	\end{align}
\end{thm}
\begin{proof}
	Let $C, \sigma, \theta$ as in \Cref{prop:Loja Normal}. Like in \cite[p. 357]{CFS09}, there exists $\sigma'>0$ such that every $h\in W^{4,2}(\Sigma;\R^3)$ with $\norm{h-f}{W^{4,2}}\leq \sigma'$ can be written as $h\circ \Phi= f+\phi$ where $\Phi\colon\Sigma\to\Sigma$ is an \add{orientation-preserving} diffeomorphism and $\phi\in W^{4,2}(\Sigma;\R^3)^{\perp}$ with $\norm{\phi}{W^{4,2}}\leq \sigma$. Then, we have $\overline{\CalW}(h)=\overline{\CalW}(f+\phi)$ and $\CalV(h)=\CalV(f+\phi)=V(f)$ by invariance under diffeomorphism, and moreover by the geometric transformation of the $L^2$-norms
	\begin{align}
		\norm{\nabla\overline{\CalW}(h)-\lambda(h)\nu_h}{L^2(\diff\mu_h)} = \norm{\nabla\overline{\CalW}(f+\phi)-\lambda(f+\phi)\nu_{f+\phi}}{L^2(\diff \mu_{f+\phi})}.
	\end{align}
	Renaming $\sigma'$ into $\sigma$, the statement then follows from \Cref{prop:Loja Normal}.
 \end{proof}

\subsection{An asymptotic stability result}
 
The following stability result is an analogue of \cite[Lemma 4.1]{CFS09}.

\begin{lem}\label{lem:LojaAsymStabil}
	Let $f_W\colon\Sigma\to\R^3$ be a constrained Willmore immersion and let $k\in \N$, $k\geq 4$, $\delta>0$. Then there exists $\varepsilon=\varepsilon(f_W)>0$ such that if $f\colon [0,T)\times \Sigma\to\R^3$ is a volume-preserving Willmore flow with \add{$\CalV(f)\equiv \CalV(f_W)$} satisfying
	\begin{enumerate}[(i)]
		\item $\norm{f_0-f_W}{C^{k,\alpha}}<\varepsilon$ for some $\alpha>0$;
		\item $\overline{\CalW}(f(t))\geq \overline{\CalW}(f_W)$ whenever $\norm{f(t)\circ \Phi(t) - f_W}{C^k}\leq \delta$, for some diffeomorphisms $\Phi(t) \colon\Sigma\to\Sigma$;
	\end{enumerate}
	then, the flow exists globally, i.e.\ we may take $T=\infty$. Moreover, it converges, after reparametrization by some diffeomorphisms $\tilde{\Phi}(t)\colon\Sigma\to\Sigma$, smoothly to a constrained Willmore immersion $f_\infty$, satisfying  $\overline{\CalW}(f_W)=\overline{\CalW}(f_\infty)$ and $\norm{f_{\infty}-f_W}{C^k}\leq\delta$.
\end{lem}

\add{The proof of \Cref{lem:LojaAsymStabil} is essentially a nonlocal version of the one of \cite[Lemma 4.1]{CFS09}, with the classical {\L}ojasiewicz--Simon inequality replaced with the constrained one. It is thus moved to \Cref{sec:proof_asympt_stabil}.}
\add{This finally enables us to prove \Cref{thm: no compact blowups2}.}

\begin{proof}[{Proof of \Cref{thm: no compact blowups2}}]
	By \Cref{thm:BlowUpExistence}, there are $t_j \nearrow T, r_j \to r \in [0, \infty]$ and $x_j \in \R^3$ for all $j\in \N$ such that $t_j+\hat{c}r_j^4<T$ and
	\begin{align}\label{eq:no comp blowups 1}
		\hat{f}_j \defeq r_j^{-1}\left( f(t_j+\hat{c}r_j^4, \cdot)-x_j\right)\to \hat{f}
	\end{align}
	smoothly, after reparametrization, on compact subsets of $\R^3$, {where $\hat{f}\colon\hat{\Sigma}\to\R^3$ is a constrained Willmore immersion}. \add{By assumption, $\hat{\Sigma}$ contains a compact component and thus, by \add{the same argument as in \cite[Lemma 4.3]{KSSI}}, we may assume $\hat{\Sigma}=\Sigma$ is compact. Consequently $\hat{f}_j \circ\Phi_j \to\hat{f}$ smoothly on $\Sigma$,
	where $\Phi_j \colon\Sigma\to\Sigma$ are diffeomorphisms.
	Note that now $\hat{f}$ is a constrained Willmore immersion of the compact surface $\hat{\Sigma}=\Sigma$. Thus, there exists $\varepsilon=\varepsilon(\hat{f})$ as in \Cref{lem:LojaAsymStabil}. We would like to apply \Cref{lem:LojaAsymStabil} for the flow with the initial datum $\hat{f}_{j}\circ \Phi_j$, however, this might not have the correct volume. Under the assumptions of the theorem, we can fix that by another rescaling.
	Note that by smooth convergence and since $\Sigma$ is compact, if $\CalV(\hat{f})\neq 0$, then also $\CalV(\hat{f}_j \circ \Phi_j)\neq 0$ and $\Phi_j$ is orientation-preserving for all $j$ sufficiently large. For such $j\in \N$, we define
	\begin{align}\label{eq:def v_j}
		v_j \defeq \left\lbrace\begin{array}{ll}
			\left(\frac{\CalV(\hat{f})}{\CalV(\hat{f}_j \circ \Phi_j)}\right)^{\frac{1}{3}}& \text{if }\CalV(\hat{f})\neq 0, \\
			1 &\text{if }\CalV(\hat{f})=\CalV(f_0)=0,
		\end{array}\right. 
	\end{align}
	By smooth convergence and convergence of the volume, we have $v_j\to 1$ as $j\to\infty$, so we may assume $v_j \in (0,2)$ and
	\begin{align}\label{eq:v_jhat f_j C^4 distance}
		\norm{v_{j_0} \hat{f}_{j_0}\circ \Phi_{j_0} - \hat{f}}{C^{4,\alpha}} \leq \abs{v_{j_0} - 1}\norm{\hat{f}_{j_0}\circ \Phi_{j_0}}{C^{4,\alpha}} + \norm{\hat{f}_{j_0}\circ \Phi_{j_0}-\hat{f}}{C^{4,\alpha}}<\varepsilon,
	\end{align}
	if we choose $j=j_0$ sufficiently large. We define $\bar{r}_{j_0}\defeq v_{j_0}^{-1} r_{j_0}\in (0,\infty)$.
	By \Cref{rem:ParabolicScaling}, the flow
	\begin{align}
		h_{j_0}(t,\cdot)\defeq \bar{r}_{j_0}^{-1}\left(f(t_{j_0}+\bar{r}_{j_0}^4t, \cdot)-x_{j_0}\right)\circ \Phi_{j_0}, \quad t\in [0, \bar{r}_{j_0}^{-4}(T-t_{j_0})),
	\end{align}
	is again a volume-preserving Willmore flow with $h_{j_0}(v_{j_0}^4\hat{c}) =v_{j_0}\hat{f}_{j_0} \circ\Phi_{j_0}$ and volume $\CalV(h_{j_0})\equiv \CalV(v_{j_0} \hat{f}_{j_0} \circ \Phi_{j_0})=\CalV(\hat{f})$ by definition of $v_{j_0}$.}	
	Moreover, for $t\in [0, \bar{r}_{j_0}^{-4}(T-t_{j_0}))$, we have using monotonicity of the energy, the invariances of the Willmore energy and $t_k\nearrow T$
	\begin{align}
		\overline{\CalW}(h_{j_0}(t)) \geq \lim_{s\to \bar{r}_{j_0}^{-4}(T-t_{j_0})}\overline{\CalW}(f(t_{j_0}+\bar{r}_{j_0}^4s)) = \lim_{s\to T}\overline{\CalW}(f(s)) = \lim_{k\to\infty}\overline{\CalW}(\hat{f}_k)=\overline{\CalW}(\hat{f}).
	\end{align}
	The last equality holds since the convergence $\hat{f}_k\circ \Phi_k \to \hat{f}$ is smooth. \add{This together with \eqref{eq:v_jhat f_j C^4 distance} yields that the assumptions of \Cref{lem:LojaAsymStabil} are satisfied, and thus the flow $h_{j_0}$} exists globally with 
	\begin{align}
		h_{j_0} (t)\circ \tilde{\Phi}(t) \to f_{\infty} \quad\text{ smoothly as }t\to\infty,
	\end{align}
	where $\tilde{\Phi}(t)\colon \Sigma\to\Sigma$ are diffeomorphisms and $f_{\infty}$ is a constrained Willmore immersion.  \add{Hence, $f$ also exists globally, so we may take $T=\infty$. Moreover, for all $t\geq t_{j_0}$ we have
	\begin{align}
		&f\left(t, \Phi_{j_0}\circ \tilde{\Phi}(\bar{r}_{j_0}^{-4}(t-t_{j_0}))\right)\nonumber\\
		&= \bar{r}_{j_0} h_{j_0}\left(\bar{r}_{j_0}^{-4}(t-t_{j_0}), \tilde\Phi(\bar{r}_{j_0}^{-4}(t-t_{j_0}))\right)+x_{j_0} \to \bar{r}_{j_0}f_\infty + x_{j_0}\label{eq:f_flow_conv}
	\end{align}
	as $t\to\infty$ smoothly on $\Sigma$. It remains to show that $\hat{f}$ is a limit under translation. Let $r_j\to r\in [0,\infty]$.	Picking $t \defeq t_k+\hat{c}r_k^4$, $k\in\N$, in \eqref{eq:f_flow_conv}, we obtain for the diameters
	\begin{align}
		d_k\defeq \diam{f(t_k+\hat{c}{r}_k^4)(\Sigma)} \to \bar{r}_{j_0} \diam f_{\infty}(\Sigma), \quad \text{ as }k\to\infty,
	\end{align}
	whence $\lim_{k\to\infty}d_k\in (0, \infty)$ since $\Sigma$ is compact. On the other hand, using \eqref{eq:no comp blowups 1} we find
	\begin{align}
		\diam \hat{f}(\hat{\Sigma}) =\lim_{k\to\infty} r_k^{-1}d_k \in (0,\infty),
	\end{align}
	as $\hat{\Sigma}\neq \emptyset$ is compact by assumption. Consequently, $\lim_{k\to\infty}r_k \in (0,\infty)$.
	}
\end{proof}

\section{Convergence to the sphere}\label{sec:conv}
In this section, we will prove our main convergence result. While $\Sigma$ was a general surface before, in this section we will 
	 work exclusively with $\Sigma=\S^2$. The key ingredients {in proving \Cref{thm:conv main}} are the blow-up procedure, the classification of Willmore spheres in $\R^3$ due to Bryant \cite{Bryant1984}, and a removability result for point singularities \cite{KSRemovability}.

\delete{\begin{thm}\label{thm:Willmore sphere classification}
	Let $f\colon\S^2 \to \R^3$ be a Willmore immersion, which is not the round sphere. Then ${\CalW}(f)\geq 16\pi$.
\end{thm}
\begin{proof}
	As discussed in \cite[Section 5]{Bryant1984}, the possible critical values of the functional $W(f) \defeq \int (\frac{1}{4}H^2-K)\diff \mu$ for immersions $f\colon\S^2\to\R^3$ are $4\pi d$ with $d\in \N_0\setminus\{1,2\}$. Note that $W(f) = \CalW(f)-4\pi$ by Gau\ss--Bonnet. If $f$ is not a round sphere, we have $W(f)>0$ and hence $W(f)\geq 12\pi$, so ${\CalW}(f)\geq 16\pi$.
\end{proof}
We are finally able to prove our main result.
}
\begin{proof}[{Proof of \Cref{thm:conv main}}]
	Let $f\colon[0,T)\times \S^2\to\R^3$ be a volume-preserving Willmore flow with initial datum $f_0$ with $T$ maximal and ${\CalW}(f_0)\leq 8\pi$. If $f_0$ is a constrained Willmore immersion, then it is a Willmore immersion by \Cref{lem:constr Willmore}, since $\CalV(f_0)\neq 0$. Hence it has to be a round sphere \delete{by \Cref{thm:Willmore sphere classification}}\add{since by \cite[Section 5]{Bryant1984}, the critical values of Willmore immersions of spherical type are $4\pi d$ with $d\in \N\setminus \{2,3\}$ and the global minimizers are the round spheres \cite{Willmore65}}. In this case the result follows. If $f_0$ is not a constrained Willmore immersion, then the energy instantaneously drops below $8\pi$ by \Cref{rem:W strict Lyapunov}, so we can assume $\CalW(f_0)<8\pi$.
	
	By \Cref{thm:BlowUpExistence}, \Cref{rem:lambda^2A 8pi} and \Cref{prop:nontrivial blow up}, there exist $t_j \nearrow T, (r_j)_{j\in \N}\subset (0,\infty)$ and $(x_j)_{j\in \N}\subset\R^3$ such that the corresponding concentration limit $\hat{f}\colon\hat{\Sigma}\to\R^3$ is a unconstrained Willmore immersion satisfying
	\begin{align}\label{eq:convergence 1}
	\int_{\hat{\Sigma}}\abs{A_{\hat{f}}}^2\diff \mu_{\hat{f}}>0.
	\end{align}
	\add{Moreover, by \Cref{thm:BlowUpExistence} we have $\CalW(\hat{f})<8\pi$.} 
	\add{Suppose $\hat{\Sigma}$ is not compact. There is $x_0\not\in \hat{f}(\hat{\Sigma})$ and with the inversion $I(x)\defeq \abs{x-x_0}^{-2}(x-x_0)$, we set $\bar{\Sigma}\defeq I(\hat{f}(\hat{\Sigma}))\cup \{0\}$.}  By the removability result \cite[Lemma 5.1]{KSRemovability}, $\bar{\Sigma}$ is a smooth Willmore surface. Moreover, since $\hat{\Sigma}$ is complete \add{by \Cref{cor:langer_complete},} so is $\hat{f}(\hat{\Sigma})$. Hence, $\operatorname{dist}(x_0, \hat{f}(\hat{\Sigma}))>0$ and consequently $\bar{\Sigma}$ is bounded. Using the definition of $\bar{\Sigma}$ and the completeness of $\hat{f}(\hat{\Sigma})$ again, it is not difficult to show that $\bar{\Sigma}$ is closed in $\R^3$ and thus compact.
	Furthermore, by \cite[Lemma 5.1]{KSRemovability}, we have $\CalW(\bar{\Sigma})<8\pi$ and $g(\bar{\Sigma})=0$ and hence  $\bar{\Sigma}$ is a Willmore sphere. \add{Using \cite{Bryant1984,Willmore65} as above, we conclude that} $\bar{\Sigma}$ has to be a round sphere. Since $\hat{f}(\hat{\Sigma})$ is not compact by assumption, this yields that $\hat{f}(\hat{\Sigma}) = I^{-1}(\bar{\Sigma})$ is a plane, contradicting \eqref{eq:convergence 1}.
	
	Thus, $\hat{\Sigma}$ is compact, hence \add{by arguing as in \cite[Lemma 4.3]{KSSI}}, we can assume $\hat{\Sigma}=\Sigma=\S^2$. \add{By \cite{Bryant1984,Willmore65} and the Li--Yau inequality \cite{LiYau}, we then have that $\hat{f}$ parametrizes an embedded round sphere, in particular $\CalV(\hat{f})\neq 0$.}
	\add{Hence, \Cref{thm: no compact blowups2} yields global existence and convergence to a constrained Willmore immersion $f_\infty$ with $\CalW(f_\infty)=\CalW(\hat{f})$. By \cite{Willmore65}, $f_\infty$ parametrizes a round sphere.}
	Since the volume is preserved by \eqref{eq:dtV=0}, we conclude that $\CalV(\bar{r}_{j_0}f_{\infty}+x_{j_0}) = \CalV(f_0)$ and consequently the radius is $R\defeq (\frac{3\abs{\CalV(f_0)}}{4\pi})^{\frac{1}{3}}>0$.
\end{proof}

\begin{appendices}
	\crefalias{section}{appendix}
	
		\delete{\section{On the notion of volume}\label{sec:volume}
	{
	The signed volume defined in \eqref{eq:defVol} is motivated by the situation when $f\colon\Sigma\to\R^3$ is  an embedding. In this case, by the Jordan--Brouwer Separation Theorem, see for instance \cite[Chapter 2, § 5]{GuiPoll}, the complement $\R^3\setminus f(\Sigma)$ has two connected components $\Omega$ and $\Omega^o$, where $\Omega$, the ``inside'', is bounded and $\Omega^o$ is unbounded. \\
	Hence, there exists a unique orientation on $\Sigma$ such that $\nu$ defined as in \eqref{eq:def normal} is the interior unit normal, pointing towards the interior $\Omega$. It is then an application of the divergence theorem that the volume enclosed by $f(\Sigma)$, i.e.\ the $3$-dimensional Lebesgue measure of $\Omega$, can be computed by
	\begin{align}
		\abs{\Omega} =  -\frac{1}{3} \int_{\Sigma}\langle f, \nu\rangle\diff\mu = \CalV(f).
	\end{align}
	While this calculation justifies the definition of the volume in \eqref{eq:defVol}, we note that $\CalV(f)$ can also be negative in the case of an embedding, namely if the orientation on $\Sigma$ makes $\nu$ defined by \eqref{eq:def normal} the exterior unit normal.
		Nevertheless, even in this case the enclosed volume is still preserved along solutions of \eqref{eq:VpWF} with $\lambda$ as in \eqref{eq:deflambda}, as $\abs{\Omega} = \frac{1}{3}\int_{\Sigma}\langle f, \nu\rangle\diff \mu = -\CalV(f_0)$ if $\nu$ is the exterior unit normal. \\
		We remark that in the literature, different definitions of the volume of an immersion appear. Using the language of differential forms, in \cite[(1.1)]{BuergerKuwert} the volume of $Z\colon \S^2\to\R^3$ is given by 
		\begin{align}\label{eq:vol Kuwert}
			V(Z) = \int_{\S^2} Z^{\ast}\omega, \quad\text{where }\omega(x) = \frac{1}{3} x~\llcorner (\diff x^1\wedge\diff x^2\wedge\diff x^3), 
		\end{align}
		where $\llcorner$ denotes the interior product of differential forms and $\diff x^1\wedge\diff x^2\wedge\diff x^3$ is the standard volume form on $\R^3$. Another definition is given in \cite[p.2]{BlattHelfrich}, where 
		\begin{align}\label{eq:vol Blatt}
			\operatorname{vol}(f) = \int_{[0,1]\times \Sigma} \phi_f^{\ast}(\diff x^1\wedge\diff x^2\wedge\diff x^3),
		\end{align}
		with $\phi_f(t,p) = tf(x)$. We note that both \eqref{eq:vol Kuwert} and \eqref{eq:vol Blatt} depend on the orientation on $\Sigma$. A direct computation shows that for the unit normal induced by the orientation, cf. \eqref{eq:def normal}, both \eqref{eq:vol Kuwert} and \eqref{eq:vol Blatt} coincide with our definition of the volume in \eqref{eq:defVol}.} 
	}
	\delete{
	\section{Simon's monotonicity formula}
	The following monotonicity formula has become one of the fundamental tools in understanding the Willmore energy. It was originally proven for embedded compact surfaces in \cite{SimonWillmore} and then extended to the varifold setting in \cite[Appendix A]{KSRemovability}. The version we state here is as in \cite{KSLectureNotes}.
	\begin{lem}[{\cite[Section 2]{KSLectureNotes}}]
		Let $f\colon\Sigma\to\R^3$ be a proper immersion of a surface $\Sigma$ without boundary. Then, there exists an absolute constant $0<C<\infty$ such that for all $x\in \R^3$ and $0<\sigma \leq \rho<\infty$ we have
		\begin{align}\label{eq:SimonMono 1}
			\sigma^{-2}\mu(B_{\sigma}(x))\leq C \left(\rho^{-2}\mu(B_{\rho}(x))+ \int_{B_{\rho}(x)}\abs{H}^2\diff \mu\right).
		\end{align}
		In particular, for $x\in f(\Sigma)$ and for all $0<\rho<\infty$ we find
		\begin{align}\label{eq:SimonMono 2}
			\pi &\leq C \left(\rho^{-2}\mu(B_{\rho}(x))+ \int_{B_{\rho}(x)}\abs{H}^2\diff \mu\right),
		\end{align}
		whereas if $f(\Sigma)$ is compact, and $0<\sigma<\infty$ we get
		\begin{align}\label{eq:SimonMono 3}
			\sigma^{-2}\mu(B_{\sigma}(x)) \leq C \CalW(f).
		\end{align}
	\end{lem}
	}
	
	\section{Smooth convergence on compact sets}
	The essential tool in the construction of the blow-up in \Cref{thm:BlowUpExistence} was the following local version of Langer's compactness theorem \cite{LangerCompactness} \add{by Kuwert and Sch\"atzle \cite{KSSI}}, see also \cite{Breuning} and \cite[Appendix B]{DMSS20} for some consequences of this notion of convergence. 
	\begin{thm}[\add{\cite[Theorem 4.2]{KSSI}}]\label{thm:LocLangerCompactness}
		Let $f_j \colon\Sigma_j\to \R^3$ be a sequence of proper immersions, where $\Sigma_j$ is a $2$-manifold without boundary.
		 Let $\Sigma_j(R)\defeq \{p\in \Sigma_j\mid \abs{f_j(p)}<R\}$ and assume the bounds
		\begin{align}
			\mu_j(\Sigma_j(R))&\leq C(R) \text{ for any }R>0,\\
			\norm{\nabla^m A_j}{L^{\infty}(\Sigma_j)}&\leq C(m) \text{ for all  }m\in \N_0.
		\end{align}
		Then, there exist a proper immersion $\hat{f}\colon \hat{\Sigma}\to\R^3$, where $\hat{\Sigma}$ is a $2$-manifold without boundary, such that after passing to a subsequence we have a representation
		\begin{align}\label{eq:graphRep}
			f_j \circ \phi_j = \hat{f} + u_j \text{ on }\hat \Sigma(j) = \{p\in \hat\Sigma\mid \abs{\hat{f}(p)}<j\}
		\end{align}
		with the following properties:
		\begin{align}
			&\phi_j \colon\hat{\Sigma}(j)\to U_j \subset \Sigma_j \text{ is a diffeomorphism,} \\
			&\Sigma_j(R)\subset U_j \text{ if }j\geq j(R),\\
			&u_j \in C^{\infty}(\hat{\Sigma}(j);\R^3)\text{ is normal along }\hat{f},\\
			&\norm{\hat{\nabla}^{m}u_j}{L^{\infty}(\hat{\Sigma}(j))} \to 0 \text{ as } j\to\infty, \text{ for any }m\in \N_0.
			\end{align}
	\end{thm}
	\add{
	\begin{cor}\label{cor:langer_complete}
		In \Cref{thm:LocLangerCompactness}, the manifold $(\hat{\Sigma}, g_{\hat{f}})$ is complete.
	\end{cor}
\begin{proof}
	Suppose $(p_n)_{n\in \N}\subset \hat{\Sigma}$ is a Cauchy-Sequence with respect to the Riemannian distance $\hat{d}$ on $\hat{\Sigma}$. Recall that the metric $g_{\hat{f}}=\hat{f}^*\langle \cdot, \cdot\rangle$ on $\hat{\Sigma}$ induced by the immersion $\hat{f}$ makes $\hat{f}$ an isometry. Now, for any  curve $\gamma\colon[0,1]\to\hat{\Sigma}$ such that $\eta(0)=p_n, \gamma(1)=p_m$ we have
	\begin{align}
		\abs{\hat{f}(p_n)- \hat{f}(p_m)} \leq  \mathcal{L}(\hat{f}\circ\gamma) = \CalL(\gamma),
	\end{align}
and hence we find $\abs{\hat{f}(p_n)- \hat{f}(p_m)} \leq \hat{d}(p_n, p_m)$ for all $n,m\in\N$. In particular there exists $R>0$ such that $(\hat{f}(p_n))_{n\in\N}\subset \overline{B_R(0)}$. As $\hat{f}$ is proper we find $p_n \in \hat{f}^{-1}(\overline{B_R(0)})$ which is compact. Since $(p_n)_{n\in\N}$ is Cauchy, $\lim_{n\to\infty}p_n\in\hat{\Sigma}$ exists.
\end{proof}
}
	\delete{
	\begin{remark}\label{rem:langer loc empty}
		Note that the limit $\hat{\Sigma}$ in \Cref{thm:LocLangerCompactness} can be empty, in which case $\hat{f}$ is the empty map. This can in fact happen, if we take $f_j$ to parameterize a sequence of round spheres $B_1(x_j)$ with $x_j \to \infty$, for instance.
	\end{remark}
	}
	\delete{
	\begin{remark}\label{rem:volume convergence}
		In the case that $\Sigma_j =\hat{\Sigma}=\Sigma$ is compact in \Cref{thm:LocLangerCompactness}, after taking $j$ sufficiently large we may assume that $\phi_j \colon\Sigma\to\Sigma$ are diffeomorphisms on $\Sigma$ and we have $f_j\circ \phi_j = \hat{f}+u_j$ on $\Sigma$ with $\norm{u_j}{C^m(\Sigma, g_{\hat{f}})} \to 0$ for all $m\in \N_0$, see \cite[Proposition C.9]{DMSS20}. This yields $\norm{f_j \circ \phi_j - \hat{f}}{C^m(\Sigma, g_{\hat{f}})} \to 0$ as $j\to\infty$ for all $m\in \N_0$. By \Cref{lem:unit normal props} (ii), we then find $\nu_{f_j\circ \phi_j} \to \nu_{\hat{f}}$ in $C^0(\Sigma, g_{\hat{f}})$, so that using \Cref{lem:analytic} (i) and \eqref{eq:defVol}, we deduce $\CalV(f_j\circ \phi_j)\to \hat{f}$ as $j\to\infty$.
	\end{remark}
	}

	\delete{
	\section{Interpolation inequalities}
	
	The interpolation inequalities from \cite{KSGF} and \cite{KSSI} are the crucial tools in proving the integral estimates in \Cref{sec:loc int est}. We restate some of them here in the case of codimension one for the convenience of the reader.
}
	\delete{
	\begin{lem}\label{lem:A^6Estimate}
		There exist absolute constants $\varepsilon_0, C\in (0,\infty)$ such that if $\int_{[\gamma>0]}\abs{A}^2\diff\mu<\varepsilon_0$ for $\gamma\in C^\infty_c(\Sigma)$ with $\norm{\nabla\gamma}{\infty}\leq \Lambda$ we have
		\begin{align}
			\int \abs{A}^6\gamma^4\diff \mu \leq C \int_{[\gamma>0]} \abs{A}^2\diff \mu \int\abs{\nabla^2 A}^2\gamma^4 \diff \mu + C\Lambda^4 \left( \int_{[\gamma>0]}\abs{A}^2\diff \mu\right)^2.
		\end{align}
	\end{lem}
	\begin{proof}
		By a rescaling argument, from \cite[Lemma 4.2]{KSGF} one concludes
		\begin{align}
			\int \abs{A}^6\gamma^4\diff \mu \leq C \int_{[\gamma>0]}\abs{A}^2\diff \mu \int\left(\abs{\nabla^2 A}^2 + \abs{A}^6\right)\gamma^4\diff \mu + C \Lambda^4 \left(\int_{[\gamma>0]}\abs{A}^2\diff \mu\right)^2,
		\end{align}
		where $C$ is an absolute constant. For $\varepsilon_0>0$ small enough absorbing yields the claim.
	\end{proof}
	}
	
	\delete{
	\begin{lem}\label{lem:LinftyEstimate}
		Let $\phi$ be an $\ell\choose 0$-tensor and $\gamma\in C^{\infty}_c(\Sigma)$ with $\norm{\nabla\gamma}{\infty}\leq \Lambda$. Then we have
		\begin{align}\label{eq:LinftyGeneral}
			\norm{\gamma^2\phi}{\infty}^4 \leq C \norm{\gamma^2\phi}{L^2}^2\left[ \int \left(|\nabla^2\phi|^2+\abs{H}^4|\phi|^2\right)\gamma^{4}\diff \mu + \Lambda^4 \int_{[\gamma>0]} \abs{\phi}^2\diff\mu\right].
		\end{align}
		Moreover, there exist universal constants $\varepsilon_0, C\in(0,\infty)$ such that if $\int_{[\gamma>0]} \abs{A}^2\diff \mu <\varepsilon_0$  we have
		\begin{align}\label{eq:LinftyA}
			\norm{\gamma^2 A}{\infty}^4 \leq C \norm{\gamma^2 A}{L^2}^2\left(\int \abs{\nabla^2 A}^2 \gamma^4 \diff \mu+ \Lambda^4\int \abs{A}^2\gamma^4 \diff \mu\right).
		\end{align}
	\end{lem}
	\begin{proof}
		The first statement is as in \cite[Lemma 2.8]{KSSI}. The second then follows from \Cref{lem:A^6Estimate}.
	\end{proof}
}
	
	

	\add{
		\section{Proof of \texorpdfstring{\Cref{lem:curvatureIntegralsEstimate}}{Lemma 3.1}}\label{sec:appL31}
		This section is devoted to proving \Cref{lem:curvatureIntegralsEstimate}. First, we compute a localized version of \eqref{eq:dtW<=0}. {Although the calculations are essentially the same as in \cite[Section 3]{KSGF}, we give some details here how the dependence on $\lambda$ comes into play.}
		
		\begin{lem}\label{lem:EvolutionOfCurvatureIntegrals}
			Let $f\colon[0,T)\times \Sigma\to\R^3$ be a smooth volume-preserving Willmore flow, $\tilde{\eta} \in C_c^{\infty}(\R^3)$ and $\eta\defeq  \tilde{\eta}\circ f$. Then, we have
			\begin{align}
				\partial_t \int\frac{1}{2}H^2\eta\diff \mu + \int \abs{\nabla\overline{\CalW}(f)}^2\eta\diff\mu &= \lambda \int |A^0|^2H\eta  \diff \mu -2 \int \nabla_{sc}\overline{\CalW}(f) \langle\nabla H, \nabla \eta\rangle \diff \mu  \\
				&\quad - \int \nabla_{sc}\overline{\CalW}(f) H \Delta\eta \diff \mu + \int\frac{1}{2}H^2\partial_t \eta \diff\mu \label{eq:dtIntH^2Equation}
			\end{align}
			and
			\begin{align}
				\partial_t \int |A^{0}|^2\eta\diff \mu + \int \abs{\nabla\overline{\CalW}(f)}^2\eta \diff \mu &= \lambda\int |A^0|^2 H \eta\diff \mu -2\int  \nabla_{sc}\overline{\CalW}(f) \langle \nabla H,\nabla \eta\rangle \diff \mu \\
				&\quad - 2 \int  \nabla_{sc}\overline{\CalW}(f) \langle A^{0},\nabla^2\eta\rangle \diff \mu  +  \int|A^{0}|^2\partial_t \eta\diff\mu.\label{eq:dtIntA^0^2Eq}
			\end{align}
		\end{lem}
		\begin{proof}
			We use a (local) orthonormal basis $\{e_i\}_{i=1,2}$. As in \cite[(31) and (32)]{KSSI}, using \eqref{eq:dtdmu} and \eqref{eq:dtH} we find
			\begin{align}
				\partial_t \left(\frac{1}{2}H^2\diff  \mu\right) 
				&= - \abs{\nabla\overline{\CalW}(f)}^2\diff \mu + \lambda \Delta H \diff \mu + \lambda |A^{0}|^2 H\diff \mu + \nabla_i\left(H\nabla_i \xi - \xi\nabla_i H\right)\diff \mu
			\end{align}
			Consequently, we compute using integration by parts
			\begin{align}
				&\partial_t \int\frac{1}{2}H^2\eta\diff \mu + \int \abs{\nabla\overline{\CalW}(f)}^2\eta\diff\mu \\
				&\qquad= \lambda \int (\Delta H + |A^0|^2H)\eta  \diff \mu +\int \left(2\xi\nabla_i H\nabla_i \eta + H\xi\Delta\eta\right)\diff \mu + \int\frac{1}{2}H^2\partial_t \eta \diff\mu.
			\end{align}
			Now, using \eqref{eq:nabla W sc} we observe that 
			\begin{align}
				\int \left(2\xi\nabla_i H\nabla_i \eta + H\xi\Delta\eta\right)\diff \mu &= -2 \int  \nabla_{sc}\overline{\CalW}(f) \nabla_i H \nabla_i \eta \diff \mu + 2\lambda \int \nabla_i H\nabla_i \eta\diff \mu \\
				&\quad - \int \nabla_{sc}\overline{\CalW}(f) H \Delta\eta \diff \mu + \lambda \int H\Delta\eta\diff \mu.
			\end{align}
			{Recalling that $\Delta(H\eta) = \Delta H \eta + 2 \nabla_i H \nabla_i \eta + H\Delta\eta$, the identity \eqref{eq:dtIntH^2Equation} follows.}
			
			For the second identity, 
			we proceed as in \cite[p.~423]{KSSI}. Using \eqref{eq:dtg} and the identity $A^{0}_{ik}A^{0}_{kj}A^{0}_{ij}=0$ (see \cite[(2.5)]{KSGF}), a short computation yields
			\begin{align}
				A^0(\partial_t e_i, e_j) A^0(e_i, e_j) 
				&= \frac{1}{2} |A^0|^2 H\xi. \label{eq:A^0dtei}
			\end{align}
			Applying \eqref{eq:dtdmu}, \eqref{eq:dtA^0} and \eqref{eq:A^0dtei} yields
			\begin{align}
				&\partial_t\left(|A^{0}|^2\diff \mu\right) 
				= 2 \nabla_i(\nabla_j \xi A^{0}(e_i, e_j)) \diff \mu - \nabla_j\xi \nabla_j H\diff\mu + |A^{0}|^2H\xi\diff \mu,
			\end{align}
			where we used \eqref{eq:nabla H A A^0} and the fact that $A^0_{ij}(\nabla^2_{ij}\xi)^{0}= A^0_{ij}\nabla^2_{ij}\xi$ as $A^0$ is trace-free. Consequently we find
			\begin{align}
				\partial_t\left(|A^{0}|^2\diff \mu\right) &= 
				 2\nabla_i(\nabla_j \xi A^{0}(e_i, e_j)) \diff \mu- \nabla_j(\xi \nabla_j H) \diff \mu - \abs{\nabla\overline{\CalW}(f)}^2\diff \mu + \lambda  \nabla_{sc}\overline{\CalW}(f) \diff \mu.
			\end{align}
			Integration by parts and \eqref{eq:nabla H A A^0} then yield
			\begin{align}
				&\partial_t \int |A^{0}|^2\eta\diff \mu + \int \abs{\nabla\overline{\CalW}(f)}^2\eta \diff \mu \\
				&\qquad = -2\int  \nabla_{sc}\overline{\CalW}(f) \nabla_i H\nabla_i \eta\diff \mu - 2 \int  \nabla_{sc}\overline{\CalW}(f) A^{0}_{ij}\nabla^2_{ji} \eta \diff \mu + \int|A^{0}|^2\partial_t \eta\diff\mu\\
				&\qquad \quad + \lambda\left[ - \int \nabla_i H \nabla_i \eta\diff \mu + 2 \int \nabla_i H \nabla_i\eta\diff \mu+ \int \Delta H \eta\diff\mu + \int|A^{0}|^2 H \eta\diff \mu\right]
			\end{align}
			{The claim follows from integrating by parts in the terms involving $\lambda$.}
		\end{proof}
		Equipped with this evolution identity, we can now prove \Cref{lem:curvatureIntegralsEstimate}.
		\begin{proof}[Proof of \Cref{lem:curvatureIntegralsEstimate}]
			Again, we use a local orthonormal basis $\{e_i\}_{i=1,2}$. To prove both inequalities in \Cref{lem:curvatureIntegralsEstimate}, we estimate the evolution in \Cref{lem:EvolutionOfCurvatureIntegrals} {with $\eta=\gamma^4$}. The last term in \eqref{eq:dtIntH^2Equation} and \eqref{eq:dtIntA^0^2Eq} generates an additional term with $\lambda$, since
			\begin{align}\label{eq:dt eta}
				|\partial_t \eta| \leq \delete{4}\add{C} \Lambda \gamma^3  \abs{\partial_ t 	f} \leq \delete{4}\add{C} \Lambda \gamma^{3} \left(\abs{\nabla\overline{\CalW}(f)} + |\lambda|\right).
			\end{align}	
			Therefore, both \eqref{eq:dtIntH^2Equation} and \eqref{eq:dtIntA^0^2Eq} contain two terms involving $\lambda$. The terms without $\lambda$ can be estimated exactly as in \cite[Lemma 3.2]{KSSI} (with $\rho^{-1}=\Lambda$). The claim follows after we estimate the $\lambda$-term generated by $\partial_t\eta$ as in \eqref{eq:dt eta} and keep the term $\lambda\int\abs{A^0}^2H\gamma^4\diff \mu$.
		\end{proof}
	}

	\section{\texorpdfstring{Proof of \Cref{thm:HigherOrderEstimatesTLocalized}}{Proof of Proposition 3.5}}\label{sec:Higherorder}
	\add{This section is devoted to proving \Cref{thm:HigherOrderEstimatesTLocalized}.}
	\delete{In this section, we will prove \Cref{lem:EvolutionOfNabla^mAIntegrals}, {which extends the estimates in \cite[Proposition 3.3 and Proposition 4.5]{KSGF}.}}
	
	\add{Following \cite{KSGF,KSSI}, for tensors $\phi, \psi$ on $\Sigma$, we denote by $\phi*\psi$ any multilinear form, depending on $\phi$ and $\psi$ in a universal bilinear way. In particular, we have $\abs{\phi*\psi}\leq c\abs{\phi}\abs{\psi}$ and $\nabla(\phi*\psi)=\nabla\phi*\psi+\phi*\nabla\psi$. Note that since we are in codimension one, we can work with tensors with scalar {values} and not with normal values.
		
		Moreover, for $m\in \N_0$ and $r\in \N, r\geq 2$ we denote by $P^m_r(A)$ any term of the type
		\begin{align}
			P^m_r(A) = \sum_{i_1+\dots+i_r=m} \nabla^{i_1}A*\dots*\nabla^{i_r}A.
		\end{align}
		{In addition, for $r=1$ we extend this definition by denoting by $P^m_1(A)$ any contraction of $\nabla^mA$ with respect to the metric $g$.}}
		\delete{  
		Given an $\ell\choose 0$-tensor $\psi$ we denote by $\nabla^{\ast}\psi$ the formal adjoint of $\nabla$ given by $\nabla^{\ast} \psi = -(\nabla_{e_i} \psi)(e_i, \dots)$. In this notation the Laplacian on tensors is given by $\Delta = -\nabla^{\ast}\nabla$. An important commutator relation is the identity
		\begin{align}\label{eq:Nabla*NablaCommutator}
			(\nabla \nabla^{\ast} - \nabla^{\ast}\nabla)\nabla \psi= A*A*\nabla\psi + A*\nabla A*\psi
		\end{align}
		for any $\ell\choose 0$-tensor $\psi$, cf. \cite[(2.11)]{KSGF}. 
		Moreover, we recall Simons' identity (cf. \cite{Simons})
		\begin{align}\label{eq:Simons'Identity}
			\Delta A = \nabla^2 H + 2 K A^0 = \nabla^2 H + A*A*A.
		\end{align}
		}
		\add{We can now compute the evolution of higher order derivatives of the second fundamental form.
		\begin{lem}\label{lem:localEvolNabla^mA}
			Let $f\colon [0,T)\times\Sigma\to\R^3$ be a volume-preserving Willmore flow. Then for all $m\in \N_0$ we have
			\begin{align}
				\partial_t (\nabla^m A)+ \Delta^2(\nabla^m A) = P^{m+2}_3(A) + P^m_5(A) + \lambda P^{m}_2(A).
			\end{align}
		\end{lem}
		\begin{proof}
			{First, we note that $H$ is a contraction of $A$ and hence $H=P^{0}_1(A)$, and consequently also $A^0=P^0_1(A)$. Thus, by \eqref{eq:VpWF}, we have 
				\begin{align}\label{eq:dt f in P notation}
					\xi =-\Delta H + P_3^0(A)+\lambda.
				\end{align}
			}
			For $m=0$ we insert this into \eqref{eq:dtA} to obtain
			\begin{align}
				\partial_t A &= \nabla^2 \xi + A*A*\xi = -\nabla^2(\Delta H) + P^2_3(A) + P^0_5(A) + \lambda P_2^0(A),
			\end{align}
			Using \cite[(2.11)]{KSGF} twice, we find $	\nabla^2\Delta H = \Delta\nabla^2 H +P^2_3(A)$, hence by Simons' identity \cite{Simons} we have
			\begin{align}
				\partial_t A = - \Delta^2 A + P^2_3(A)+P^0_5(A)+\lambda P^0_2(A).
			\end{align}
			Assume the statement is true for $m\geq 1$. Using \cite[Lemma 2.3]{KSGF} with $\phi = \nabla^m A$ and the fact that we are in codimension one yields
			\begin{align}
				\partial_t \nabla^{m+1} A + \Delta^2 \nabla^{m+1}A &= \nabla\left(P^{m+2}_3(A)+P^{m}_5(A)+\lambda P^{m}_2(A)\right)\\
				&\quad + \sum_{i+j+k=3} \nabla^{i}A*\nabla^j A*\nabla^{k+m}A \\
				&\quad + A*\nabla\xi*\nabla^m A + \nabla A*\xi*\nabla^m A \\
				& = P^{m+3}_3(A) + P^{m+1}_5(A)+\lambda P^{m+1}_2(A),
			\end{align}
			where we used \eqref{eq:dt f in P notation} in the last step.
		\end{proof}
	}
	
	 \delete{To that end, we first recall the following
	\begin{lem}[{\cite[Lemma 3.2]{KSGF}}]\label{lem:KSGFL3.2}
	Let $f\colon [0,T)\times \Sigma \to\R^3$ be a normal variation, $\partial_t f = \xi \nu$. Let $\phi$ be a $\ell \choose 0$-tensor satisfying $\partial_t \phi + \Delta^2 \phi = Y$. Then for any $\gamma\in C^2([0,T)\times \Sigma)$ and $s\geq 4$ we have
	\begin{align}
		&\frac{\diff}{\diff t} \int |\phi|^2\gamma^s\diff \mu + \int \abs{\nabla^2 \phi}\gamma^s\diff \mu - 2 \int\langle Y, \phi\rangle\gamma^{s}\diff \mu  \\
		&\leq \int A * \phi * \phi* \xi\gamma^s \diff \mu + \int |\phi|^2 s\gamma^{s-1}\partial_t \gamma \diff \mu \\
		&\quad + C\int \abs{\phi}^2\gamma^{s-4} \left(\abs{\nabla\gamma}^4 + \gamma^2\abs{\nabla^2\gamma}^2\right)\diff\mu + C\int \abs{\phi}^2 \left(\abs{\nabla A}^2+\abs{A}^4\right)\gamma^s\diff \mu,
	\end{align}
	where $C=C(s)$.
	\end{lem}
	}
	\add{
		In analogy to \cite[Proposition 3.3]{KSSI}, we have localized energy estimates for higher order derivatives of $A$.
		\begin{lem}\label{lem:EvolutionOfNabla^mAIntegrals}
			Let $f\colon [0,T)\times \Sigma\to\R^3$ be a volume-preserving Willmore flow and $\gamma$ as in \eqref{eq:gamma}. Then for $\phi=\nabla^m A, m\in \N_0$ and $s\geq 2m+4$ we have
			\begin{align}
				&\frac{\diff}{\diff t} \int |\phi|^2\gamma^s\diff \mu + \frac{1}{2} \int \abs{\nabla^2 \phi}\gamma^s\diff \mu \\
				&\leq  C\left( \abs{\lambda}^{\frac{4}{3}}  + \norm{A}{L^{\infty}([\gamma>0])}^4\right) \int\abs{\phi}^2\gamma^s\diff \mu  + C\left(1+ \abs{\lambda}^{\frac{4}{3}}+ \norm{A}{L^{\infty}([\gamma>0])}^4\right)\int_{[\gamma>0]}\abs{A}^2\diff\mu
			\end{align}
			where $C=C(s,m, \Lambda)>0$.
		\end{lem}
	}

\begin{proof}
	In the following, note that the value of $C=C(s,m,\Lambda)$ is allowed to change from line to line. \add{Using \cite[Lemma 3.2]{KSGF}, we find}
	\begin{align}
		&\frac{\diff}{\diff t} \int |\phi|^2\gamma^s\diff \mu + \int \abs{\nabla^2 \phi}\gamma^s\diff \mu \\
		&\leq   2 \int\langle Y, \phi\rangle\gamma^{s}\diff \mu + \int  A * \phi * \phi* \xi\gamma^s \diff \mu + \int |\phi|^2 s\gamma^{s-1}\partial_t \gamma \diff \mu \\
		&\quad + C\int \abs{\phi}^2\gamma^{s-4} \left(\abs{\nabla\gamma}^4 + \gamma^2\abs{\nabla^2\gamma}^2\right)\diff\mu + C\int \abs{\phi}^2 \left(\abs{\nabla A}^2+\abs{A}^4\right)\gamma^s\diff \mu,\label{eq:nabla^mApartial_tgamma0}
	\end{align}
	where $\partial_t \phi +\Delta^2 \phi = Y$ and $\xi= P^2_1(A)+P_3^0(A)+\lambda$ by \eqref{eq:dt f in P notation}. By \Cref{lem:localEvolNabla^mA}, we have
	\begin{align}
		&2\int \langle Y, \phi\rangle\gamma^s\diff \mu +\int A*\phi*\phi*\xi\gamma^s \diff \mu + C\int|\phi|^2(\abs{\nabla A}^2+\abs{A}^4)\gamma^s\diff\mu \\
		&\quad = \int \left(P^{m+2}_3(A)+P^m_5(A)\right)*\phi\gamma^s\diff \mu + \lambda \int P^m_2(A)*\phi\gamma^{s}\diff \mu.\label{eq:nabla^mApartial ohne dtgamma}
	\end{align}
	Moreover, by \eqref{eq:gamma} we find
	\begin{align}\label{eq:nabla^mApartial_tgamma}
		\int |\phi|^2\gamma^{s-1}\partial_t \gamma\diff \mu &= \int \abs{\phi}^2\gamma^{s-1}\langle D\tilde{\gamma}\circ f, \nu\rangle \left(-{\Delta H} - |A^0|^2 H + \lambda\right)\diff \mu.
	\end{align}
	{We proceed by estimating all the terms involving $\lambda$ in \eqref{eq:nabla^mApartial ohne dtgamma} and \eqref{eq:nabla^mApartial_tgamma}. For the $\lambda$-term on the right hand side of \eqref{eq:nabla^mApartial ohne dtgamma}, using \cite[Corollary 5.5]{KSGF} with $k=m$, $r=3$ we find
	\begin{align}\label{eq:higherorderlambda2}
		\lambda\int P^m_2(A)*\phi\gamma^s\diff \mu \leq C(s,m, \Lambda) \abs{\lambda} \norm{A}{L^\infty([\gamma>0])} \left(\int \abs{\phi}^2\gamma^s\diff \mu + \int_{[\gamma>0]} \abs{A}^2\diff \mu\right).
	\end{align}
	The $\lambda$-term on the right hand side of \eqref{eq:nabla^mApartial_tgamma} is estimated using Young's inequality with $p=\frac{4}{3}$ and $q=4$ to obtain
\begin{align}\label{eq:higherorderlambda1}
	C\abs{\lambda} \int \abs{\phi}^2\gamma^{s-1}\diff \mu &\leq C\abs{\lambda}^{\frac{4}{3}} \int\abs{\phi}^2\gamma^{s}\diff \mu + C\int \abs{\phi}^2\gamma^{s-4}\diff \mu.
\end{align}
	Consequently, we find from \eqref{eq:nabla^mApartial_tgamma0}, \eqref{eq:nabla^mApartial ohne dtgamma}, \eqref{eq:nabla^mApartial_tgamma} and Young's inequality
	\begin{align}
		&\frac{\diff}{\diff t} \int |\phi|^2\gamma^s\diff \mu + \int \abs{\nabla^2 \phi}\gamma^s\diff \mu \\
		&\quad\leq  \int \left(P^{m+2}_3(A)+P^m_5(A)\right)*\phi\gamma^s\diff \mu +
		\int \abs{\phi}^2\gamma^{s-1}\langle D\tilde{\gamma}\circ f, \nu\rangle \left(-{\Delta H} - |A^0|^2 H \right)\diff \mu\\
		&\qquad + C\left(\abs{\lambda}^{\frac{4}{3}} + \norm{A}{L^{\infty}([\gamma>0])}^4\right) \int\abs{\phi}^2\gamma^s\diff \mu +C \left(\abs{\lambda}^{\frac{4}{3}}+ \norm{A}{L^\infty([\gamma>0])}^4\right) \int_{[\gamma>0]} \abs{A}^2\diff \mu \\
		&\qquad + C\int \abs{\phi}^2\gamma^{s-4}\diff \mu +\int \abs{\phi}^2\gamma^{s-4} \left(\abs{\nabla\gamma}^4 + \gamma^2\abs{\nabla^2\gamma}^2\right)\diff\mu.\label{eq:D6}
	\end{align}
	Now, all the terms involving $\lambda$ on the right hand side of \eqref{eq:D6} are as in the statement.  For the second {and the last} term in \eqref{eq:D6}, one may proceed exactly as in the proof of \cite[Proposition 3.3]{KSGF}. This way, one creates additional terms which can be estimated by
	\begin{align}
		&\int\abs{\phi}^2\gamma^{s-4}\diff \mu + \int\abs{\nabla\phi}^2\gamma^{s-2}\diff \mu\leq \varepsilon \int|\nabla^2\phi|^2\gamma^s\diff \mu + C_\varepsilon \int_{[\gamma>0]}\abs{A}^2\gamma^{s-4-2m}\diff \mu,
	\end{align}
for every $\varepsilon>0$, using twice the interpolation inequality \cite[Corollary 5.3]{KSGF} (which trivially also holds in the case $k=m=0$). The first term on the right hand side of \eqref{eq:D6} can then be estimated by means of \cite[(4.15)]{KSGF}. After choosing $\varepsilon>0$ small enough and absorbing, the claim follows.
}
\end{proof}

\add{
\Cref{thm:HigherOrderEstimatesTLocalized} can now be deduced from a Gronwall-type argument exactly as in \cite[Theorem 3.5]{KSSI}. To keep track of the role of $\lambda$, we give the details here.	
	
\begin{proof}[Proof of \Cref{thm:HigherOrderEstimatesTLocalized}]
	\delete{For $\rho>0$, we may rescale as in \Cref{rem:ParabolicScaling} to obtain the volume-preserving Willmore flow $\tilde{f}$ on $[0, \tilde{T})$ with $\tilde{T}=\rho^{-4}T \leq T^{*}$ and the same bound on the $L^{4/3}$-norm of the Lagrange multiplier. Moreover, $\tilde{f}$ satisfies the assumptions of the statement with $\tilde{\rho}=1$.
		If \eqref{eq:highorder local in time estimate} holds for $\tilde{f}$, it is not difficult to see that the desired estimates in for $f$ follow, since the estimates in \eqref{eq:highorder local in time estimate} scale with the correct power of $\rho$. Hence, without loss of generality, we may assume $\rho=1$.}
	\add{Without loss of generality, after rescaling as in \Cref{rem:ParabolicScaling}, we may assume $\rho=1$.}
	
	We pick a cutoff function $\tilde{\gamma}\in C^\infty_c(\R^3)$ with $\chi_{B_{3/4}(x)}\leq \tilde{\gamma} \leq \chi_{B_1(x)}$ such that $\gamma\defeq \tilde{\gamma}\circ f$ is as in \eqref{eq:gamma} with a universal constant $\Lambda>0$. Now, using \Cref{prop:New30}, we deduce
	\begin{align}\label{eq:53}
		\int_0^T \int_{B_{3/4}(x)} \left(\abs{\nabla^2 A}^2+\abs{A}^6\right)\diff \mu \diff t \leq C\varepsilon+ C \Lambda^4 \varepsilon T+ C\varepsilon L = C(T^*,L)\varepsilon,
	\end{align}
	using $T\leq T^{*}$. Consequently, by using \delete{the interpolation inequality in \cite[Lemma 2.8]{KSGF}} we find
	\begin{align}\label{eq:54}
		\int_0^T \norm{A}{L^{\infty}(B_{3/4}(x))}^4 \diff t\leq C(T^*,L)\varepsilon.
	\end{align}
	Now, we change to another test function $\tilde{\gamma}\in C^{\infty}_c(\R^3)$ with  $\chi_{B_{1/2}(x)}\leq \tilde{\gamma} \leq \chi_{B_{3/4}(x)}$ and $\gamma\defeq \tilde{\gamma}\circ f$. Note that \eqref{eq:gamma} still remains satisfied with a universal $\Lambda>0$. We now define Lipschitz cutoff functions in time via
	\begin{align}
		\xi_j(t) \defeq \left\lbrace\begin{array}{ll}
			0,& \text{for } t\leq (j-1)\frac{T}{m}, \\
			\frac{m}{T}\left(t-(j-1)\frac{T}{m}\right),& \text{for } (j-1)\frac{T}{m}\leq t\leq j\frac{T}{m} \\
			1, &\text{for } t\geq j\frac{T}{m},
		\end{array}\right.
	\end{align}
	where $m\in \N$ and $0\leq j \leq m$. We also define $\xi_{-1}(t)\defeq 0$ and  $\xi_0(t)\defeq 
	1$ for all $t\in \R$ if $m=0$. We note that $\xi_m(T) =1$ and 
	\begin{align}\label{eq:xi'bound}
		0\leq  \frac{\diff}{\diff t}{\xi}_j \leq \frac{m}{T}\xi_{j-1},\quad \text{for all }j\in \N_0.
	\end{align}
	We now define $a(t) = \norm{A}{L^{\infty}(B_{3/4}(x))}^4, E_j(t) = \int |\nabla^{2j} A|^2\gamma^{4j+4}\diff \mu$.
	Then, by \Cref{lem:EvolutionOfNabla^mAIntegrals} and using $\gamma\leq 1$ we have
	\begin{align}
		\frac{\diff}{\diff t} E_j(t) + \frac{1}{2} E_{j+1}(t) &\leq C(j,m)\left(\abs{\lambda(t)}^{\frac{4}{3}} + a(t)\right) E_j(t) + C(j,m)\left(1+\abs{\lambda(t)}^{\frac{4}{3}} + a(t)\right) \varepsilon.
	\end{align}
	Therefore, if we define $e_j \defeq \xi_j E_j$ this implies using \eqref{eq:xi'bound}
	\begin{align}
		\frac{\diff}{\diff t} e_j(t) &\leq \frac{m}{T}\xi_{j-1}(t)E_j(t) + C(j,m)\left(\abs{\lambda(t)}^{\frac{4}{3}} + a(t)\right) e_j(t) \\
		&\quad + C(j,m)\left(1+\abs{\lambda(t)}^{\frac{4}{3}} + a(t)\right) \varepsilon	-\frac{1}{2} \xi_j(t) E_{j+1}(t).\label{eq:56}
	\end{align}
	We will now show that this implies for $0\leq j\leq m$ and $t\in (0,T)$
	\begin{align}\label{eq:HigherOrderLocalizedInduction}
		e_j(t) + \frac{1}{2} \int_0^t \xi_j(s) E_{j+1}(s)\diff s \leq \frac{C(j,m,T^{\ast},L) \varepsilon}{T^j}.
	\end{align}
	We proceed by induction on $j$. For $j=0$ we have $\xi_0 \equiv 1$ on $(0,T)$. Therefore, we have $e_0 =\int |A|^2\gamma^4\diff \mu \leq \varepsilon$ by assumption. Moreover, by \eqref{eq:53} we find  $\int_0^tE_{1}(s)\diff s = \int_0^t\int |\nabla^2 A|^2 \gamma^{8}\diff \mu \diff s \leq C(T^{\ast},L)\varepsilon$.
	
	For $j\geq 1$ we have, integrating \eqref{eq:56} on  $[0,t]$ and using $e_j(0)=0$
	\begin{align}
		&e_j(t) + \frac{1}{2} \int_0^t \xi_j(s)E_{j+1}(s)\diff s\\
		&\quad \leq C(j,m) \int_0^t \left(\abs{\lambda(s)}^{\frac{4}{3}} + a(s)\right)e_j(s)\diff s + C(j,m) \varepsilon \int_0^t\left(1+\abs{\lambda(s)}^{\frac{4}{3}} + a(s)\right) \diff s \\
		&\qquad + \frac{m}{T} \int_0^t \xi_{j-1}(s) E_j(s) \diff s \\
		&\leq  C(j,m) \int_0^t \left(\abs{\lambda(s)}^{\frac{4}{3}} + a(s)\right)e_j(s)\diff s + C(j,m, T^*, L) \varepsilon  + \frac{C(j,m,T^*,L)\varepsilon}{T^{j-1}}~\frac{m}{T} \\
		&\leq C(j,m) \int_0^t \left(\abs{\lambda(s)}^{\frac{4}{3}} + a(s)\right)e_j(s)\diff s + \frac{C(j,m,T^{\ast},L)\varepsilon}{T^j},
	\end{align}
	using \eqref{eq:54}, the induction hypothesis and $T\leq T^*$. Therefore, Gronwall's inequality yields using \eqref{eq:53} and \eqref{eq:54}
	\begin{align}
		e_j(t) &\leq - \frac{1}{2} \int_0^t \xi_{j}(s)E_{j+1}(s)\diff s + \frac{C(j,m,T^{\ast},L)\varepsilon}{T^j} \\
		&\quad + \int_0^t \frac{C(j,m,T^{\ast},L)\varepsilon}{T^j}\left(\abs{\lambda(s)}^{\frac{4}{3}} + a(s)\right) \exp\left(C(T^*,L)\right)\diff s \\
		&\leq - \frac{1}{2} \int_0^t \xi_{j}(s)E_{j+1}(s)\diff s + \frac{C(m,L,T^{\ast})\varepsilon}{T^j},
	\end{align}
	which proves \eqref{eq:HigherOrderLocalizedInduction}. Now evaluating at $t=T$ with $j=m$, we find
	\begin{align}
		\int |\nabla^{2m}A|^2\gamma^{4m+4}\diff \mu \leq \frac{C(m,L,T^{\ast})\varepsilon}{T^{m}}\quad\text{for all } m\in\N.
	\end{align}
	The estimate for $\nabla^{2m+1}A$ follows from the interpolation inequality in \cite[Lemma 5.1]{KSGF} with $r=1, p=q=2, \alpha=1, \beta=0, s=4m+6$ and $t=\frac{1}{2}\in [-\frac{1}{2}, \frac{1}{2}]$. Renaming $T$ into $t$ proves the $L^2$-estimate. The $L^{\infty}$-estimate then follows using the $L^\infty$-interpolation estimate in \cite[Lemma 2.8]{KSGF}, together with \cite[Lemma 4.2]{KSGF}.
\end{proof}}

\section{\texorpdfstring{Proof of \Cref{lem:LojaAsymStabil}}{Proof of Lemma 7.9}}\label{sec:proof_asympt_stabil}
\begin{proof}[{Proof of \Cref{lem:LojaAsymStabil}}]
	We follow \cite[Lemma 4.1]{CFS09}. There exists a diffeomorphism \linebreak$\Phi\colon\Sigma\to \Sigma$, such that for $\varepsilon>0$ small enough, $f_0\circ \Phi$ can be written as a normal graph over $f_W$, i.e.
	\begin{align}
		f_0\circ \Phi = f_W+ \nu_{f_W}\varphi_0 \eqdef \tilde{f}_0,
	\end{align}
	for some $\varphi_0\colon\Sigma\to\R$, such that 
	\begin{align}\label{eq:gl ex 0}
		\norm{\varphi_0}{C^{4,\alpha}} \leq C\varepsilon,
	\end{align}
	for $C$ independent of $\varepsilon$. We now wish to solve the equation
	\begin{align}\label{eq:PDE repara}
		\partial_t^{\perp}\tilde{f}_t = - \nabla\overline{\CalW}(\tilde{f}_t) + \lambda(\tilde{f}_t) \nu_{\tilde{f}_t},
	\end{align} 
	with initial datum $\tilde{f}_0$, where $\partial_t^{\perp}= P^{\perp_{\tilde{f}_t}}\partial_t$ and $\tilde{f}_t\defeq f_W+\varphi_t\nu_{f_W}$, for smooth functions $\varphi_t \colon\Sigma\to\R$. By \eqref{eq:PDE repara} and \eqref{eq:nabla W sc} (as in \cite[(4.4)]{CFS09} with codimension one) we compute
	\begin{align}
		\partial_t (\varphi_t) P^{\perp_{\tilde{f}_t}}\nu_{f_W} &= -(\Delta H_{\tilde{f}_t} + \abs{A^0_{\tilde{f}_t}}^2H_{\tilde{f}_t})\nu_{\tilde{f}_t} + \lambda(\tilde{f}_t)\nu_{\tilde{f}_t}  \\
		&= - (g_{\tilde{f}_t}^{ij}g_{\tilde{f}_t}^{k\ell} \partial_{ijk\ell}\varphi_t)P^{\perp_{\tilde{f}_t}}\nu_{f_W} \\
		& \quad + \left(1+ \int_{\Sigma}  B_0(\cdot, \varphi_t, D\varphi_t, D^2\varphi_t)\diff \mu_{f_W}\right) B_1(\cdot, \varphi_t, D\varphi_t, D^2 \varphi_t, D^3\varphi_t),\label{eq:gl ex 1}
	\end{align}
	using \eqref{eq:deflambda}, where $B_0, B_1$ are smooth functions depending on $f_W$. {Note that the nonlocal terms appear due to $\lambda$.} Now, if $\norm{\tilde{f}_t-f_W}{C^1}\leq \delta$ is small enough, $g_{\tilde{f}_t}^{ij}g_{\tilde{f}_t}^{k\ell}$ is uniformly elliptic and we may assume that 
	\begin{align}\label{eq:gl ex 1.5}
		\abs{P^{\perp_{\tilde{f}_t}} X} \geq \abs{X}- \abs{P^{\top_{f_W}}X- P^{\top_{\tilde{f}_t}}X} \geq \frac{1}{2}\abs{X}, \text{ for all } X \text{ normal along }f_W.
	\end{align}
	Therefore, \eqref{eq:gl ex 1} is equivalent to 
	\begin{align}
		&\partial_t \varphi_t + g_{\tilde{f}_t}^{ij}g_{\tilde{f}_t}^{k\ell} \partial_{ijk\ell}\varphi_t  \\
		&\qquad=\left(1+ \int_{\Sigma}  B_0(\cdot, \varphi_t, D\varphi_t, D^2\varphi_t)\diff \mu_{f_W}\right) B_1(\cdot, \varphi_t, D\varphi_t, D^2\varphi_t, D^3\varphi_t).\label{eq:gl ex 2}
	\end{align}
	Since the right hand side of \eqref{eq:gl ex 2} is only of third order in $\varphi_t$, it is not too difficult to see that the \add{parabolic initial value} problem \eqref{eq:gl ex 2} with initial datum $\varphi_0$ satisfying \eqref{eq:gl ex 0} has a unique local solution in the H\"older space $H^{\frac{k+\alpha}{4}, k+\alpha}([0, T_1]\times \Sigma;\R)$ for some $0<T_1\leq T$. This follows from maximal regularity results for linear parabolic problems, in H\"older spaces, cf. \cite{LSU68}, and a fixed-point argument using the contraction principle, \add{see also \Cref{prop:STE} and the corresponding references.} Here the order reduction for $\lambda$ discussed in \Cref{subsec:PDE} is crucial.
	\add{Now, we apply \Cref{thm:Loja} to $f_W$. By the embedding $C^{4}(\Sigma)\hookrightarrow W^{4,2}(\Sigma)$, we may assume that the constrained {\L}ojasiewicz--Simon inequality is satisfied for all $\norm{h-f_W}{C^{4}}\leq \sigma$ with exponent $\theta\in (0,\frac{1}{2}]$.}
	 Choosing $\varepsilon>0$ sufficiently small, we may without loss of generality assume $C\varepsilon<\sigma<\delta$ with $\sigma$ as in \Cref{thm:Loja} and that $T_1$ is the maximal existence interval for \eqref{eq:gl ex 2} \add{for which we have} (as part of our definition of $T_1$)
	\begin{align}\label{eq:gl ex 4}
		\norm{\tilde{f}_t-f_W}{C^k}\leq \sigma<\delta \text{ for all }t\in [0,T_1).
	\end{align}
	By parabolic Schauder estimates, from \eqref{eq:gl ex 2} and \eqref{eq:gl ex 0} we obtain a bound on the parabolic Hölder space norm, i.e.\ $\norm{\varphi}{H^{\frac{k+\alpha}{4},k+\alpha}}\leq C$, and hence for $k$ as in the statement
	\begin{align}\label{eq:gl ex 5}
		\norm{\tilde{f}_t-f_W}{C^{k, \alpha}}\leq C \text{ for all }t\in [0, T_1).
	\end{align}
	\delete{As a next step, we will  establish a connection between the two flows.} By \eqref{eq:PDE repara}, we find
	\begin{align}
		\partial_t \tilde{f}_t + \xi_t d\tilde{f}_t = -\nabla\overline{\CalW}(\tilde{f}_t)+\lambda(\tilde{f}_t)\nu_{\tilde{f}_t},
	\end{align}
	where $\xi_t$ denotes the tangential velocity. Next, by classical flow theory, see for instance \cite[Chapter 17]{Lee}, there exists a unique smooth family of diffeomorphisms satisfying
	\begin{align}
		\partial_t \Phi_t &= \xi_t \circ\Phi_t \text{ on }\Sigma\text{ for }0\leq t<T_1 \\
		\Phi_0& = \Id_{\Sigma}.
	\end{align}
	A direct calculation yields $\partial_t (\tilde{f}_t \circ \Phi_t) = - \nabla\overline{\CalW}(\tilde{f}_t \circ \Phi_t) + \lambda(\tilde{f}_t\circ \Phi_t) \nu_{\tilde{f}_t\circ \Phi_t}$, so 
	\begin{align}
		[0, T_1)\times \Sigma \to\R^3, (t,p)\mapsto\tilde{f}_t \circ \Phi_t \circ \Phi^{-1}(p)
	\end{align}
	is a smooth volume-preserving Willmore flow with initial data $\tilde{f}_0 \circ \Phi_0\circ \Phi^{-1} = f_0$. As the solution to the volume-preserving Willmore flow is unique, {cf. \Cref{prop:STE}}, we conclude $T_1\leq T$ and
	\begin{align}
		f_t = \tilde{f}_t \circ \Phi_t \circ \Phi^{-1} \text{ for all }0\leq t<T_1.
	\end{align}
	It suffices to prove that $\tilde{f}$ is global and converges as $t\to\infty$ to a smooth Willmore immersion $f_\infty$ with the desired properties. 
	
	First, we show that we may assume $\overline{\CalW}(f_t)>\overline{\CalW}(f_W)$ for all $t\in [0, T_1)$. By assumption and \eqref{eq:gl ex 4}, we have $\overline{\CalW}({f}_t)\geq \overline{\CalW}(f_W)$. If $\overline{\CalW}(f_t)=\overline{\CalW}(f_W)$ for some $t\in [0,T_1)$, \add{then by \Cref{rem:W strict Lyapunov}, $f$ and $\tilde{f}$ are stationary and the claim follows. Hence, we may indeed assume the strict inequality $\overline{\CalW}(f_t)>\overline{\CalW}(f_W)$.}
	
	Let $\theta,C$ as in \Cref{thm:Loja}. By \eqref{eq:dtV=0}, \eqref{eq:gl ex 4} and \add{since $\CalV(f_t)=\CalV(f_W)$, we may apply the constrained {\L}ojasiewicz--Simon gradient inequality to obtain}
	\begin{align}
		&-\frac{\diff}{\diff t}\left(\overline{\CalW}(f_t)-\overline{\CalW}(f_W)\right)^{\theta} \\
		&\quad= -\theta \left(\overline{\CalW}(f_t)-\overline{\CalW}(f_W)\right)^{\theta-1} \langle \nabla\overline{\CalW}(\tilde{f}_t), \partial_t^{\perp} \tilde{f}_t\rangle_{L^2(\diff \mu_{\tilde{f}_t})} \\
		& \quad = -\theta \left(\overline{\CalW}(f_t)-\overline{\CalW}(f_W)\right)^{\theta-1} 
		\langle \nabla\overline{\CalW}(\tilde{f}_t) - \lambda(\tilde{f}_t)\nu_{\tilde{f}_t}, \partial_t^{\perp} \tilde{f}_t\rangle_{L^2(\diff \mu_{\tilde{f}_t})}\\
		&\quad = \theta \left(\overline{\CalW}(f_t)-\overline{\CalW}(f_W)\right)^{\theta-1} 
		\norm{\nabla\overline{\CalW}(\tilde{f}_t) - \lambda(\tilde{f}_t)\nu_{\tilde{f}_t}}{L^2(\diff \mu_{\tilde{f}_t})} \norm{\partial_t^{\perp} \tilde{f}_t}{L^2(\diff \mu_{\tilde{f}_t})}\\
		&\quad \geq \frac{\theta}{C} \norm{\partial_t^{\perp}\tilde{f}_t}{L^2(\diff\mu_{\tilde{f}_t})}
	\end{align}
	for $0\leq t<T_1$. Now, using \eqref{eq:gl ex 1.5}, \eqref{eq:gl ex 4} and the resulting equivalence of the metrics $g_{f_W}$ and $g_{\tilde{f}_t}$, we find
	\begin{align}\label{eq:gl ex 6}
		\norm{\partial_t \tilde{f}_t}{L^2(\diff \mu_{f_W})} \leq - \frac{C}{\theta} \frac{\diff}{\diff t}\left(\overline{\CalW}(f_t)-\overline{\CalW}(f_W)\right)^{\theta} \text{ for every }t\in [0, T_1).
	\end{align}
	{Integrating in time and using the triangle inequality we find}
	\begin{align}
		\norm{\tilde{f}_t-f_W}{L^2(\diff \mu_{f_W})}&\leq  \norm{\tilde{f}_0-f_W}{L^2(\diff \mu_{f_W})} + C\left(\overline{\CalW}(\tilde{f}_0)-\overline{\CalW}(f_W)\right)^{\theta} \\
		&\leq C \norm{\tilde{f}_0-f_W}{C^2(\Sigma)}^{\theta},
	\end{align}
	using the mean value theorem for the Willmore energy and assuming that $\varepsilon>0$ is small enough.
	\add{As in \cite[p.~361]{CFS09}, by interpolation for some $\beta\in (0,1)$ we find} for $t\in [0, T_1)$ and $k$ as in the statement, using \eqref{eq:gl ex 5} and \eqref{eq:gl ex 0}
	\begin{align}
		\norm{\tilde{f}_t-f_W}{C^k(\Sigma)}&\leq C \norm{\tilde{f}_t-f_W}{C^{k, \alpha}(\Sigma)}^{1-\beta} \norm{\tilde{f}_t-f_W}{L^2(\Sigma, \diff\mu_{f_W})}^{\beta}\\
		&\leq C \norm{\tilde{f}_0-f_W}{C^2(\Sigma)}^{\beta\theta} \leq C \varepsilon^{\beta\theta} \leq \frac{\sigma}{2},\label{eq:gl ex 7}
	\end{align}
	if $\varepsilon>0$ is sufficiently small. Since $T_1>0$ is chosen maximal with respect to \eqref{eq:gl ex 4}, this implies $T_1=\infty$, which yields that $\tilde{f}$ exist globally and satisfies $\norm{\tilde{f}_t-f_W}{C^k(\Sigma)}\leq \sigma$ for all $t\geq 0$. Therefore, \eqref{eq:gl ex 6} yields $\partial_t \tilde{f}_t \in L^1([0, \infty);L^2(\Sigma,\diff \mu_{f_W}))$, and consequently, there exists $f_\infty\defeq \lim_{t\to\infty}\tilde{f}_t$ in $L^2(\Sigma, \diff \mu_{f_W})$. Similar to \eqref{eq:gl ex 7}, an interpolation argument and \eqref{eq:gl ex 5}  yield $\lim_{t\to\infty}\tilde{f}_t = f_\infty$ in $C^k(\Sigma)$. {By parabolic Schauder estimates, one can then obtain $L^{\infty}$-bounds on higher order derivatives, such that by interpolation again, one can show that the convergence $\lim_{t\to\infty}\tilde{f}_t=f_\infty$ is even smooth.} Since the volume-preserving Willmore flow is a gradient flow, $f_\infty$ is a constrained Willmore immersion. Using that $\norm{f_\infty-f_W}{C^k(\Sigma)}\leq \sigma$, we find by \Cref{thm:Loja} that
	\begin{align}
		\abs{\overline{\CalW}(f_\infty)-\overline{\CalW}(f_W)}^{1-\theta}\leq C \norm{\nabla\overline{\CalW}(f_\infty)-\lambda(f_\infty)\nu_{f_\infty}}{L^2(\diff \mu_{f_\infty})} =0,
	\end{align} so $\overline{\CalW}(f_\infty)=\overline{\CalW}(f_W)$.
\end{proof}

\end{appendices}

\section*{Acknowledgments}
This project has been supported by the Deutsche Forschungsgemeinschaft (DFG,
German Research Foundation), project no. 404870139.
The author would like to thank
Anna Dall’Acqua and Marius M\"uller for many helpful discussions and comments. \add{The author would also like to thank Mattia Fogagnolo for pointing out that the proof of \Cref{thm:conv main} works for initial energy equal to $8\pi$ as well. In addition, the author is grateful to the referee for their careful reading and their valuable comments on the original manuscript.}

\bibliography{Lib}
\bibliographystyle{abbrv}

\end{document}